\documentclass{article}

\usepackage{amssymb,amsfonts, stmaryrd, amsmath, amsthm}
\usepackage[all,arc]{xy}
\usepackage{enumerate}
\usepackage{graphicx, enumerate}
\usepackage[tmargin=0.75in,bmargin=0.75in,lmargin=0.75in,rmargin=0.75in]{geometry}
\usepackage[utf8]{inputenc}
\usepackage[parfill]{parskip}

\makeatletter
\def\thm@space@setup{
  \thm@preskip=\parskip \thm@postskip=0pt
}
\makeatother

\newtheorem{thm}{Theorem}[section]
\newtheorem*{thm*}{Theorem}
\newtheorem{cor}[thm]{Corollary}
\newtheorem*{cor*}{Corollary}
\newtheorem{prop}[thm]{Proposition}
\newtheorem*{prop*}{Proposition}
\newtheorem{lem}[thm]{Lemma}

\theoremstyle{definition}
\newtheorem{defn}[thm]{Definition}

\newtheorem{exmp}[thm]{Example}

\theoremstyle{remark}
\newtheorem{rem}[thm]{Remark}

\newcommand{\Q}{\mathbb{Q}}
\newcommand{\Z}{\mathbb{Z}}
\newcommand{\C}{\mathbb{C}}

\newcommand{\kos}{\text{!`}}
\newcommand{\sh}{\,\rotatebox[origin=c]{90}{$\in$}\,}

\title{The $A_\infty$-Centre of the Yoneda Algebra and the Characteristic Action of Hochschild Cohomology on the Derived Category}

\author{Benjamin Briggs\\
Vincent G\'elinas\thanks{Both authors are affiliated with the University of Toronto. The second author was supported by an NSERC CGS grant.}}
\renewcommand\footnotemark{}

\begin{document}

\maketitle

\begin{abstract}
\noindent
For $A$ a dg (or $A_\infty$) algebra and $M$ a module over $A$, we study the image of the characteristic morphism $\chi_M: {\rm HH}^*(A, A) \to {\rm Ext}_A(M, M)$ and its interaction with the higher structure on the Yoneda algebra ${\rm Ext}_A(M, M)$. To this end, we introduce and study a notion of $A_\infty$-centre for minimal $A_\infty$-algebras, agreeing with the usual centre in the case that there is no higher structure. We show that the image of $\chi_M$ lands in the $A_\infty$-centre of ${\rm Ext}_A(M, M)$. When $A$ is augmented over $k$, we show (under mild connectedness assumptions) that the morphism $\chi_k: {\rm HH}^*(A, A) \to {\rm Ext}_A(k,k)$ into the Koszul dual algebra lands exactly onto the $A_\infty$-centre, generalising the situation from the Koszul case established by Buchweitz, Green, Snashall and Solberg. We give techniques for computing $A_\infty$-centres, hence for computing the image of the characteristic morphism, and provide worked-out examples. We further study applications to topology. In particular we relate the $A_\infty$-centre of the Pontryagin algebra to a wrong way map coming from the homology of the free loop space, first studied by Chas and Sullivan.
\end{abstract}

\tableofcontents

\section{Introduction}

Let $A$ be an algebra over a field $k$. It is well-known that the Hochschild cohomology ${\rm HH}^*(A, A)$ acts on the derived category through the characteristic morphism $\chi_M: {\rm HH}^*(A, A) \to {\rm Ext}^*_A(M, M)$ for each dg module $M$. This refines to a morphism of graded commutative algebras $\chi: {\rm HH}^*(A, A) \to Z_{gr}{\rm D}A$ into the graded centre of the derived category ${\rm D}A$. The characteristic morphism $\chi$ has numerous applications, having been studied in connections with the classical cohomology operators of Gulliksen, the theory of support varieties, with obstruction spaces to deformations of modules, with the Atiyah-Hochschild character and being related to Umkehr maps for loop space fibrations. % References.

In spite of this, the image of $\chi_M: {\rm HH}^*(A, A) \to {\rm Ext}_A(M, M)$ is not well understood beyond the initial estimate ${\rm im}(\chi_M) \subseteq Z_{gr}{\rm Ext}_A(M, M)$. The projection morphism $\Pi := \chi_A: {\rm HH}^*(A, A) \to {\rm Ext}_A(A, A) = A$ lands precisely on the centre of $A$. But at the other end of the spectrum, when $A$ is augmented over $k$ the map $\chi_k: {\rm HH}^*(A, A) \to Z_{gr}{\rm Ext}_A(k, k)$ can and typically does fail to be surjective. % I think GSS give an example which should probably be cited?

The characteristic morphism lifts to a morphism of dg algebras $C^*(A, A) \to {\rm RHom}_A(M, M)$.  This suggests that the image of $\chi_M$ should take into account the homotopy-theoretic information  in ${\rm RHom}_A(M, M)$, or equivalently, the higher structure on cohomology. So, in this paper we study the interaction of $\chi_M$ with the $A_\infty$-algebra structure on ${\rm Ext}_A(M, M)$.

One task we undertake is to introduce and study a notion of $A_\infty$-centre for minimal $A_\infty$-algebras. In particular, the cohomology of any dg algebra naturally has an $A_\infty$-centre, typically smaller than its graded centre. 
The $A_\infty$-centre is defined concretely in terms of higher commutators for the higher products, and often can be computed explicitly in the way. Alternatively, if $B$ is a dg (or $A_\infty$) algebra the $A_\infty$-centre $Z_\infty {\rm H}(B)$ of ${\rm H}(B)$ can be described as the image of the projection map $\Pi:{\rm HH}^*(B,B)\to {\rm H}(B)$.

Assume that $A$ is augmented over a commutative semi-simple base $k$-algebra $\Bbbk$. Buchweitz, Green, Snashall and Solberg \cite{MR2461267} have shown that if $A$ is a Koszul algebra then the image of $\chi: {\rm HH}^*(A, A) \to {\rm Ext}_A(\Bbbk,\Bbbk)$ is the graded centre $Z_{gr}{\rm Ext}_A(\Bbbk,\Bbbk)$ of the Yoneda algebra. Since Koszulity of $A$ can be characterized by formality of ${\rm RHom}_A(\Bbbk,\Bbbk)$, and since the $A_\infty$-centre coincides with the graded centre for honest graded algebras, their theorem is generalised by the following result:
\begin{thm*}The image of $\chi: {\rm HH}^*(A, A) \to {\rm Ext}_A(\Bbbk,\Bbbk)$ is precisely $Z_\infty{\rm Ext}_A(\Bbbk,\Bbbk)$.
\end{thm*}
That something like this should hold has been speculated by Keller. It is true in fact for any dg (or $A_\infty$) algebra, with some mild connectedness assumptions. It follows immediately from our main theorem, which we state now (see theorem \ref{maintheorem} below). We will write $A^!$ to mean ${\rm RHom}_A(\Bbbk,\Bbbk)$, or indeed any $A_\infty$-algebra quasi-isomorphic to this, so that ${\rm H}(A^!)={\rm Ext}_A(\Bbbk,\Bbbk)$. With appropriate finiteness conditions, we have:
\begin{thm*}
There is an isomorphism of graded algebras fitting into a commutative diagram
\[
    \xymatrix@R=12mm{
    			 {\rm HH}^*(A,A)       \ar@{=}[r]^\sim   \ar[d]_<<<<<\Pi  \ar[dr]_<<<<<{\chi}  & {\rm HH}^*(A^!,A^!)   \ar[d]^<<<<<\Pi \ar[dl]^<<<<<{\chi} \\
			 {\rm H}(A)   & {\rm H}(A^!).
		 }
\]
%Consequently, the image of $\chi: {\rm HH}(A, A) \to {\rm Ext}_A(\Bbbk, \Bbbk)$ is the $A_{\infty}$-centre of the $A_\infty$-Yoneda algebra ${\rm H}(A^!) = {\rm Ext}_A(\Bbbk, \Bbbk)$.
\end{thm*}
 This gives a conceptual proof of the above theorem of Buchweitz, Green, Snashall and Solberg. The horizontal isomorphism was constructed by F\'elix-Menichi-Thomas \cite{FMT}, and our isomorphism is essentially theirs. However, they work in the dg situation, and with more restrictive connectedness hypotheses (e.g. our result applies when $A$ is any finite dimensional algebra with nilpotent augmentation ideal). In part, our contribution consists in doing this intrinsically on the $A_\infty$-level. In particular, taking a minimal model for $A^!$ allows us to read the image of $\chi_k$ directly.

The Hochschild cohomology ${\rm HH}^*(A,A)$ of a dg (or $A_\infty$) algebra admits two filtrations: the (well-known) cohomological (or weight) filtration, and another which roughly corresponds to the radical filtration on $A$. After establishing homotopy invariance of these filtrations, we will upgrade the main theorem above by showing that the Koszul duality isomorphism ${\rm HH}^*(A,A)\cong {\rm HH}^*(A^!,A^!)$ exchanges these two filtrations (this appears as theorem \ref{filtrationswitch}). In the Koszul case the isomorphism ${\rm HH}^*(A,A)\cong {\rm HH}^*(A^!,A^!)$ goes back to Buchweitz  \cite{BuchweitzCanberra}. Our theorem explains the observed regrading in this case.

We will discuss in detail various models for the characteristic action of ${\rm HH}^*(A,A)$ on the derived category ${\rm D}(A)$. If $M$ is a right dg module over $A$ and $B={\rm RHom}_A(M,M)$, then one consequence is the following:

\begin{thm*}
The characteristic morphism $\chi_M: {\rm HH}^*(A, A) \to {\rm Ext}_A(M, M)$ factors through ${\rm HH^*}(B,B)$ and hence lands in the the $A_\infty$-centre of $ {\rm Ext}_A(M, M)$:
\[
\xymatrix{
          {\rm HH}^*(A, A) \ar[d]^-{\chi_M} \ar@{.>}[r] & {\rm HH}^*(B, B) \ar[d]^\Pi \\
          {\rm Ext}_A(M, M) &   Z_\infty{\rm Ext}_A(M, M). 
          \ar@{_(->}[l]
          }
\]
If $M$ is homologically balanced as a $B-A$ bimodule (see section \ref{projection-shearing-section}) then the map along the top is an isomorphism, so the image of $\chi_M$ is exactly $Z_\infty{\rm Ext}_A(M, M)$.
\end{thm*}

Next, when the augmentation ideal ${\rm ker}(A \xrightarrow{\epsilon} \Bbbk)$ is nilpotent, our main result shows
\begin{cor*}
The characteristic morphism $\chi$ induces an isomorphism of graded commutative algebras
\[
{\rm HH}^*(A, A)/\mathcal{N}il \cong Z_\infty{\rm Ext}_A(\Bbbk, \Bbbk)/\mathcal{N}il
\]
where $\mathcal{N}il$ stands for the homogeneous nilradical ideal on each side.
\end{cor*}

Finite generation of ${\rm HH}^*(A, A)/\mathcal{N}il$ has been widely studied in the context of support varieties. Our result shows that this is equivalent to studying the algebra $Z_\infty{\rm Ext}_A(\Bbbk, \Bbbk)/\mathcal{N}il$. This can be tightened for various classes of algebras. For example:
\begin{thm*}
If $A$ is furthermore $d$-Koszul, then in even cohomological degrees $\chi$ surjects onto the graded centre
\[
\chi: {\rm HH}^{even}(A, A) \twoheadrightarrow Z_{gr}{\rm Ext}^{even}_A(k, k)
\]
and induces an isomorphism of graded commutative algebras
\[
{\rm HH}^*(A, A)/{\mathcal N}il \cong Z_{gr}{\rm Ext}_A(\Bbbk, \Bbbk)/\mathcal{N}il.
\]
\end{thm*}
We wish to emphasize the computational nature of our definitions: the $A_\infty$-centre is concretely definable in terms of higher commutators, and the $A_\infty$-structure on ${\rm Ext}_A(\Bbbk, \Bbbk)$ is readable from the minimal Tate resolution of $A$, which can be built algorithmically and even naturally for many classes of algebras. In the theory of support varieties, establishing the so-called {\bf Fg} condition (of e.g. \cite{MR2044054}) in various situations is an important industry, and we hope that some of the techniques of this paper will be of use in this direction.  We provide sample computations in Section 4, but wish to point out that even partial information about the $A_\infty$-structure on ${\rm Ext}_A(\Bbbk, \Bbbk)$, for example what can be gleaned from the shape of the minimal resolution of $\Bbbk$, can help describe the relation between $Z_{gr}{\rm Ext}_A(\Bbbk, \Bbbk)$ and $Z_\infty {\rm Ext}_A(\Bbbk, \Bbbk)$. A more serious study of this point will be undertaken in future papers.

% Topology!
Our results find interesting applications to topology. The characteristic morphism ${\rm HH}^*(A, A) \to {\rm Ext}_A(k, k)$ is known to model the \textit{intersection morphism} $I: {\rm H}_{*+d}(LX; k) \to {\rm H}_{*}(\Omega X; k)$, an Umkehr, or ``wrong way'' map coming from the free loop space fibration $\Omega X \to LX \to X$ of a simply-connected closed smooth oriented $d$-manifold $X$. This morphism has been studied in the context of string topology \cite{FTVP}. As a result of work of F\'elix, Thomas and Vigu\'e-Poirrier, our main theorem immediately implies
\begin{thm*} The image of $I: {\rm H}_{*+d}(LX) \to {\rm H}_*(\Omega X)$ is $Z_\infty{\rm H}_*(\Omega X)$, the $A_\infty$-centre of the Pontryagin algebra.
\end{thm*}
We finally use this theorem to study surjectivity of $I$ in characteristic zero, recovering algebraically a topological result of the previous authors.

Thanks are due to Ragnar Buchweitz for drawing our attention to the questions dealt with in this paper. Thank you also to Dan Zacharia for his hospitality in Syracuse and for sharing with us his recent joint work \cite{MR3574211}. We thank Sarah Witherspoon for drawing our attention to connections with the theory of support varieties. Thanks to Greg Stevenson for some enlightening conversations. Lastly, thank you to Bernhard Keller for taking an interest in this work and for feedback on an early draft of this paper.

\subsection*{Structure of the paper}

%....In the strongly connected case, this gives another proof of the main theorem. However, it seems that the more relaxed connectedness assumptions afforded by the  ``dualisation" proof below will warrant taking the time to give both proofs 

The philosophy of this paper is to spend time settings things up fully, indicating what structures are at work in the background. This comes at the expense of an extensive background section, which can of course be skimmed or skipped entirely on a first read. We suggest experts take a look at section \ref{models-for-HH} before skipping to \ref{projection-shearing-section}, \ref{commutativity} and section \ref{HHKoszul}.

Section \ref{tw} introduces standard constructions on twisting cochains, which guide the point of view taken in the rest of the paper. 
In section \ref{Ainfinitydef} we collect standard background material on $A_\infty$-structures. Our conventions agree with those of Lef\`evre-Hasegawa \cite{LH} whose thesis underlies much of this work. Section \ref{models-for-HH} contains an extension of \ref{tw} to the $A_\infty$-setting. We introduce the Hochschild cochain complex $C^*(A, A)$ of an $A_\infty$-algebra $A$ as a twisted Hom complex, and prove some comparison and naturality results. The main outcome of the section is that the $A_\infty$-algebra $C^*(A, A)$ can be calculated as ${\rm Hom}^\tau(C, A)$ for any acyclic twisting cochain $\tau: C \to A$. 
Results of this type have already appeared in \cite{LH}; however our construction of ${\rm Hom}^{\tau}(C, A)$ differs substantially from Lef\`evre-Hasegawa's, and our use of higher commutators in its description seems new and particularly well-adapted to study differentials on the spectral sequences coming from standard filtrations on $C$ or $A$. Following this, in section \ref{Koszuldualitysection} we collect some well-known information on Koszul duality for associative algebras, as well and for Commutative and Lie algebras, which we will need later.

In sections \ref{projection-shearing-section} and \ref{commutativity} we introduce the main subject of this paper, namely the characteristic morphism and its variants. We define the $A_\infty$-centre of an $A_\infty$-algebra and study its relation with various notions of commutativity for $A_\infty$-algebras. % This answers a question of Szymik-Neumann about surjectivity of the projection map.
Section \ref{exchange} is devoted to the proof of the main theorem, which follows readily from the setup of \ref{models-for-HH}.

Section \ref{examples} concerns the image of $\chi$ and contains sample calculations of $A_\infty$-centres of Yoneda algebras. In particular, this section contains a quick exposition of the theory of $d$-Koszul algebras via twisting cochains.

Finally, in section \ref{stringtop} we discuss connections of $\chi$ with the intersection morphism $I$ and calculate a few examples.

\subsection*{Conventions and Definitions}\label{notation}
Let $V$ be a graded vector space over a fixed field $k$, with component vector spaces $V^n$. Call $V$ locally finite if each $V^{i}$ is finite dimensional. The cohomological suspension functor $(sV)^n = V^{n+1}$ comes along with a regrading operator $s: V \rightarrow sV$ which is the identity on elements, so that $\lvert sv \rvert = \lvert v \rvert - 1$. For graded vector spaces $V, W$, ${\rm Hom}(V, W)$ and $V \otimes W$ are the graded hom and tensor product with components
\begin{align*}{\rm Hom}^n(V, W) = \prod_{i \in \Z} {\rm Hom}_k(V^i, W^{i+n}) && (V\otimes W)^n = \bigoplus_{p+q=n} V^p \otimes W^q.
\end{align*} If $V, W$ are complexes, ${\rm Hom}(V, W)$  and $V\otimes W$ inherit standard differentials 
\begin{align*}
    \partial(f) = d_W \circ f - (-1)^{\lvert f \rvert} f \circ d_V && d(v\otimes w) = d_V(v) \otimes w + (-1)^{\lvert v \rvert} v \otimes d_W(w).
\end{align*}
Implicit in this monoidal structure is the Koszul sign rule. From the twist map $T:V\otimes W\xrightarrow{\cong}W\otimes V$ we obtain a left action of the symmetric group $S_n$ on $V^{\otimes n}$. If $v = v_1\otimes ...\otimes v_n$, the Koszul sign $(-1)^{|\sigma; v|}$ makes the equality   $\sigma \cdot  (v_1\otimes ...\otimes v_n)=  (-1)^{|\sigma; v|}( v_{\sigma^{-1}(1)}\otimes ... \otimes v_{\sigma^{-1}(n)})$ hold. Tensor products of functions applied to elements produce natural signs; in particular, we have the relation $(s^{-1})^{\otimes n} \circ (s)^{\otimes n} = (-1)^{\binom{n}{2}} {\rm id}$. The differential on $sV$ is the usual $-sd_Vs^{-1}$. 

Elements of graded objects, and hence maps between such, are always taken homogeneous and all (co)limits of graded objects are taken in the graded category, meaning term wise. 

We say that $(V, d_V)$ is augmented if it is equipped with a splitting $\epsilon: (V, d_V) \leftrightarrows k: \eta$, with $k$ in degree $0$; this gives a decomposition $V = k \oplus \overline{V}$ as well as on cohomology.

% Coalgebras, coaugmentations, cocompleteness. Fix size issues.
A coaugmentation on a $k$-coalgebra $C$ is a coalgebra morphism $\epsilon: k \to C$ splitting the counit $\eta: C \to k$. Let {\small $\overline{\Delta}$}: $\overline{C} \to \overline{C}^{\otimes 2}$ be the map induced by {\small $\Delta$} on $\overline{C}$. We denote by {\small $\Delta^{(n)}$}: $C \to C^{\otimes n}$ the $n$-fold coproduct, with {\small $\Delta^{(0)}$} being the counit and {\small $\Delta^{(1)}$} being the identity, and similarly for {\small $\overline{\Delta}^{(n)}$}.
The primitive filtration on $C$ is given by $C_{[0]} = 0$, $C_{[1]} = k$ and $C_{[p]} = k\oplus {\rm ker} {\small (\overline{\Delta}^{(p)})}$ for $p \geq 2$:
\[
0=C_{[0]} \subseteq C_{[1]} \subseteq \cdots \subseteq C_{[p]} \subseteq C_{[p+1]} \subseteq ...\\
\]
The space of primitives in $C$ is $\overline{C}_{[2]}$. We say $C$ is  cocomplete (or conilpotent) if $\varinjlim C_{[p]} = C$.

% Algebras, augmentations, completeness.
An augmentation on an $k$-algebra is an algebra map $\epsilon: A \to k$ splitting the unit map $\eta: k \to A$. Augmented algebras are dually filtered by powers of their augmentation ideal, so $A^{[0]} = A$, $A^{[1]} = \overline{A}$ and $A^{[p]}$ is the image of $p$-fold multiplication $m^{(p)}: \overline{A}^{\otimes p} \to \overline{A}$ for $p \geq 2$:
\[
A = A^{[0]} \supseteq A^{[1]} \supseteq A^{[2]} \supseteq \cdots \supseteq A^{[p]} \supseteq A^{[p+1]} \supseteq ...
\]
The space of indecomposables is $\overline{A}/\overline{A}^{[2]}$, and $A$ is complete if $A = \varprojlim A/A^{[p]}$ as a graded algebra.

The tensor coalgebra $T^{co}(V) := \bigoplus_{n \geq 0} V^{\otimes n}$ has coproduct given by
\[
\Delta(v_1 \otimes ... \otimes v_n) = \sum_{i = 0}^{n} (v_1 \otimes ... \otimes v_i) \otimes (v_{i+1} \otimes ... \otimes v_n)
\]
The primitive filtration is by tensor powers $\bigoplus_{n <p}V^{\otimes n}$, we will also call it the \textbf{weight filtration}. Tensor coalgebras are cocomplete, and are in fact cofree in the category of cocomplete coaugmented coalgebras. 
% This was in the com/lie section. I think it's good to introduce the shuffle product this way through the universal envelope, if nothing else because it quickly introduces the notation  $\mathcal{U}$.
%
% "Universal properties guarantee that there is a natural isomorphism $\mathcal{U}{\rm \bf Lie}^{co}V \cong T^{co}V$, where $\mathcal{U}$ is the universal enveloping cocomplete coalgebra. Being a universal envelope, $T^{co}V$ must be commutative Hopf algebra, ...  and the map we obtain $T^{co}V\to {\rm \bf Lie}^{co}V$ is the quotient onto the space of irreducible elements with respect to this commutative product. Reversing this picture, one can present the cofree cocomplete Lie coalgebra on $V$ as $ {\rm \bf Lie}^{co}V = \overline{T^{co}V}/(\overline{T^{co}V}\sh \overline{T^{co}V})$."
%Hopf algebras
The tensor coalgebra is a Hopf algebra with shuffle product
\[
(v_1\otimes ...\otimes v_p )\sh (v_{p+1}\otimes ...\otimes v_{p+q} )=\sum_{\sigma \in {\rm sh}(p,q)} (-1)^{|\sigma; v|} v_{\sigma^{-1}(1)}\otimes ... \otimes v_{\sigma^{-1}(p+q)}
\]
where ${\rm sh}(p,q)$ is the set of permutations in $S_{p+q}$ which separately preserve the order of $\{1,...,p\}$ and of $\{p+1,...,p+q\}$. This is the unique coalgebra morphism $m: T^{co}V\otimes T^{co}V\to T^{co}V$ such that the corresponding maps  $m_{pq}: T^{co}_pV\otimes T^{co}_qV\to V$ vanish unless $p+q=1$, in which case our hand is forced by unitality. Its space of indecomposables $\overline{T^{co}(V)}/(\overline{T^{co}(V)} \sh \overline{T^{co}(V)})$ identifies with the cofree cocomplete Lie coalgebra on $V$, denoted ${\rm \bf Lie}^{co} V$.

The tensor algebra $T^a(V) = \bigoplus_{n \geq 0} V^{\otimes n}$ is the free augmented graded $k$-algebras on $V$, and is complete if and only if locally finite\footnote{Hence $T^a(V)$ is complete if and only if $V$ is locally finite and either $V = V^{\leq -1}$ or $V = V^{\geq 1}$. One sometimes works with the completion instead.}. The tensor algebra is a Hopf algebra with unshuffle coproduct, dual to the product above; its space of primitives identifies with the free Lie algebra on $V$, denoted ${\rm \bf Lie} V$.

The symmetric algebra ${\rm Sym}(V)$ is the free graded commutative algebra on $V$, and can be presented as the quotient of $T^a(V)$ by relations $(v_1 \otimes ... \otimes v_n) - \sigma \cdot (v_1 \otimes ... \otimes v_n)$ for $\sigma \in S_n$; it is a tensor product of an odd exterior and even polynomial algebra ${\rm Sym}(V) \cong \bigwedge(V^{odd}) \otimes k[V^{even}]$. Similarly ${\rm Sym}^{co}(V)$ is the cofree cocomplete cocommutative coalgebra on $V$, and can be identified with the invariant tensors $T^{co}(V)^{Sym}$ in $T^{co}(V)$.

%Quivers
We shall sometimes implicitly replace vector spaces over $k$ by bimodules over a semisimple commutative $k$-algebra. This will allow us to think of quiver path (co)algebras as (co)augmented and (co)free over a semisimple base, through the identification $T_{kQ_0}(kQ_1) = kQ$. All results of this paper immediately extend to this setting. 
% Categories / cocategories, augmentations, ...
Occasionally we will go a step further and work with small dg categories $A$ (or cocategories) augmented over a discrete $k$-linear category with the same set of objects, thinking of quiver path algebras as finite augmented categories.

A dg quiver $Q$ consists of a set of objects and a complex ${}_yQ_x$ for each pair of objects (see \cite{Kellercocat}).  The collection of dg quivers over a fixed set of objects forms a monoidal dg category by setting 
\begin{align*}
 {\rm Hom}(Q, Q') = \prod_{x,y} {\rm Hom}({}_yQ_x, {}_yQ'_x)   && {}_y(Q\otimes Q')_x = \bigoplus_z {}_y Q_z\otimes {}_z Q'_x.
\end{align*}
By forgetting composition, a dg category $A$ has an underlying dg quiver with ${}_yA_x={\rm Hom}_A(x,y)$. The discrete quiver over which $A$ is augmented will be denoted $k_A$, so ${}_x(k_A)_x = k\cdot 1_x$ and ${}_y(k_A)_x =0$ when $x\neq y$.
% $A$ is just a monoid in the monoidal category and $k_A$ is the unit object.
% The same goes for cocategories.

Throughout this paper all (co)algebras will be (co)augmented unless otherwise stated.

\section{Background and generalities}

We begin by exhibiting standard notions and lemmas which form the backbone of the study of $A_\infty$-structures.

\subsubsection*{Universal properties for tensor coalgebras}
We first recall certain coalgebraic notions. Let $C$ be a coalgebra and $M$ a bicomodule, with given coactions
\begin{align*}
\Delta_l: M \to C \otimes M  &&  \Delta_r: M \to M \otimes C
\end{align*}
and write $\Delta_{lr}: M \to C \otimes M \otimes C$ for the map $(\Delta_l \otimes 1)\Delta_r = (1 \otimes \Delta_r)\Delta_l$. Recall that $C$ is naturally a bicomodule over itself, and any two coalgebra morphisms $f,g: C' \rightrightarrows C$ define on $C'$ a bicomodule structure over $C$ via 
\begin{align*}
\Delta_l: C' \xrightarrow{(f \otimes 1)\Delta_{C'}} C \otimes C' && \Delta_r: C' \xrightarrow{(1 \otimes g)\Delta_{C'}} C' \otimes C.     
\end{align*}
A \textbf{coderivation} $D: M \to C$ is a map satisfying $\Delta D = (1 \otimes D)\Delta_l + (D \otimes 1)\Delta_r$; we write $\rm{coder}(M, C)$ for the space of coderivations. A $\textbf{codifferential}$ on $C$ is a degree $1$ coderivation $d: C \to C$ such that $d^2 = 0$.

Let $\pi_n: T^{co}(V) \to V^{\otimes n}$ denote the projection, with $\pi_0$ the counit. As in the algebra case, morphisms and coderivations are determined against $V$ (see e.g. \cite{Quillen}, Part II). For this, let $C$ be a cocomplete coalgebra and $M$ a bicomodule over $T^{co}(V)$. % Reference: Quillen, Part II, section 3.1. 

\begin{prop}\label{universalcoalg} 
Given $\alpha: \overline{C} \to V$, there is a unique morphism $\phi_{\alpha}: C \to T^{co}(V)$ such that $\pi_1 \phi_\alpha = \alpha$. Its component in $V^{\otimes n}$ is given by
\[
\pi_n \phi_\alpha = \alpha^{\otimes n} \Delta^{(n)}.
\]
\end{prop}

\begin{prop}\label{universalcoder} Given $\alpha: M \to V$, there is a unique coderivation $D_{\alpha}: M \to T^{co}(V)$ such that $\pi_1 D_\alpha = \alpha$. Its component in $V^{\otimes n}$ is given by
\[
\pi_n D_\alpha = \sum_{i=0}^{n-1} (\pi_{i} \otimes \alpha \otimes \pi_{n-1-i})\Delta_{lr}.
\]
\end{prop} 
We spell it out when $M = T^{co}(V)$. In this case, the formula reads
\[
\pi_n D_\alpha = \sum_{i=0}^{n-1} (\pi_{i} \otimes \alpha \otimes \pi_{n-1-i})\Delta^{(3)}
\]
A map $\alpha: T^{co}(V) \to V$ is given by a family of maps $\alpha_n: V^{\otimes n} \to V$ for $n \geq 0$, and $D_\alpha$ acts on $T^{co}(V)$ by
\[
D_\alpha(v_1 \otimes ... \otimes v_n) = \sum_{j=0}^{n} \sum_{i=0}^{n-j} (-1)^{\epsilon_i} v_1 \otimes ... \otimes v_i \otimes \alpha_{j}(v_{i+1} \otimes ... \otimes v_{i+j}) \otimes v_{i+j+1} \otimes ... \otimes v_n
\]

with $\alpha_{j}(v_{i+1} \otimes ... \otimes v_{i+j}) := \alpha_{0}(1)$ when $j = 0$ and the Koszul sign given by $\epsilon_i = \lvert \alpha \rvert(\lvert v_1 \rvert + ... + \lvert v_i \rvert)$. Note that if  $\alpha_{i} = 0$ for $i > 0$ and $\alpha_{0}(1) = v$, then $D_\alpha = v\ \sh \ (-) = (-)\ \sh \ v$.
% Setup: we may need to write D_\alpha = \alpha + D^{\geq 2}_\alpha, for the stuff that lives in weight \geq 2.

\subsection{Twisting cochains}\label{tw}
Throughout this section we let $C = (C, d_C, \Delta)$ be a dg coalgebra and $A = (A, d_A, \mu)$ be a dg algebra over $k$. The complex ${\rm Hom}(C, A)$ inherits a dg algebra structure through the cup product $f \smile g = \mu \circ (f \otimes g) \circ \Delta$, we call this the convolution algebra. 

Call $\tau \in {\rm Hom}^{1}(C, A)$ a \textbf{twisting cochain} if $\tau$ vanishes on the coaugmentation $k \xrightarrow{\epsilon} C$, takes $\overline{C}$ to $\overline{A}$ and if
\[
\partial(\tau) + \tau \smile \tau = 0.
\]
We denote by ${\rm Tw}(C, A)$ the set of twisting cochains. For any degree one map $C \xrightarrow{\tau} A$ one may twist the differential on ${\rm Hom}(C, A)$ by $\partial_\tau := \partial + [\tau, -]$, and $\partial_\tau ^2 = 0$ if $\tau$ is a twisting cochain; $\partial_\tau$ is again a derivation for the cup product, and we call ${\rm Hom}^{\tau}(C, A) = ({\rm Hom}(C, A), \partial_\tau)$ the twisted Hom complex, or twisted convolution algebra.

Given a right $C$-comodule $N_C$ and left $A$-module $_A M$ (resp. left comodule and right module), there is a cap product action $\tau \frown -$ defined on $N \otimes M$ as
\[
\tau \frown -: N \otimes M \xrightarrow{\Delta \otimes 1} N \otimes C \otimes M \xrightarrow{1 \otimes \tau \otimes 1} N \otimes A \otimes M \xrightarrow{1 \otimes \mu} N \otimes M
\]
and $d_\tau  := d_N \otimes 1 + 1 \otimes d_M + \tau \frown -\ $ squares to zero when $\tau$ is a twisting cochain. In this case, we call $N\otimes^{\tau}M = (N \otimes M, d_\tau)$ the twisted tensor product, and analogously for $M\otimes^{\tau}N$.

Let {\rm \bf Alg} and {\rm \bf coAlg} stand for the category of augmented dg algebras, resp. cocomplete coaugmented dg coalgebras. Twisting cochains are natural in both arguments, giving a bifunctor
\[
{\rm \bf coAlg} \times {\rm \bf Alg} \to {\rm \bf Sets}
\]
\[ (C, A) \to {\rm Tw}(C,A) \]
This functor is left-right representable, so that
\[
{\rm \bf Alg}(\Omega C, A) \cong {\rm Tw}(C, A) \cong {\rm \bf coAlg}(C, BA).
\]
The functors $\Omega$ and $B$ are the {\bf cobar} and {\bf bar} constructions, given as follows: $\Omega C := (T^a(s^{-1}\overline{C}), d_\Omega$), where $d_\Omega$ is the derivation generated by 
\[
s^{-1}\overline{C} \xrightarrow{-s^{-1}d_Cs \oplus \ -(s^{-1})^{\otimes 2} \overline{\Delta}s} s^{-1}\overline{C} \oplus (s^{-1}\overline{C})^{\otimes 2} \subset T^a(s^{-1}\overline{C})
\]
while $BA := (T^{co}(s\overline{A}), b)$, where $b$ is the coderivation cogenerated by
\[
T^{co}(s\overline{A}) \xrightarrow{\pi_1 \oplus \pi_2} s\overline{A} \oplus (s\overline{A})^{\otimes 2} \xrightarrow{-s d_A s^{-1} +\ -s \mu (s^{-1})^{\otimes 2}} s\overline{A}
\]
That $d_\Omega^2 = 0$ and $b^2 = 0$ is equivalent to (co)associativity and the Leibniz rule holding for $C$ and $A$. We denote elements of $\Omega C$ by $s^{-1} c_1\otimes ...\otimes s^{-1}c_n= \langle c_1 |\ ...\ | c_n \rangle$ and elements of $BA$ by $sa_1\otimes ...\otimes sa_n= \big[a_1 | \ ... \ | a_n \big]$. 

Given a twisting cochain $\tau: C \to A$, (co)freeness of the (co)bar constructions induces by prop. \ref{universalcoalg} a unique coalgebra map $\phi_\tau: C \to BA$, and similarly algebra map $\psi_\tau: \Omega C \to A$. The twisting cochain condition is equivalent to $\phi_\tau$ (resp. $\psi_\tau$) preserving differentials, hence the above adjunction. This becomes a Quillen adjunction given suitable Model structures on ${\rm \bf Alg}$ and ${\rm \bf coAlg}$ \cite{LH}: we shall mention only that weak equivalences in ${\rm \bf Alg}$ are given by quasi-isomorphisms, while weak equivalences $C \xrightarrow{\sim} C'$ in ${\rm \bf coAlg}$ are given by morphisms such that $\Omega C \xrightarrow{\sim} \Omega C'$ is a quasi-isomorphism. % We make note that weak equivalences of coalgebras are themselves quasi-isomorphisms of coalgebras, but delay a fuller discussion until the next section.

The universal twisting cochains $C \xrightarrow{\iota} \Omega C$ and $BA \xrightarrow{\pi} A$ correspond to the natural inclusion and projection maps, and one knows that the unit and counit maps $C \xrightarrow{\sim} B \Omega C$ and $\Omega B A \xrightarrow{\sim} A$ are weak equivalences in each categories \cite[Sect.1.3.]{LH}. 

The bar and cobar constructions are ubiquitous in algebra and topology, but (co)freeness makes for large unwieldy models for calculations, and one is lead to look for smaller models; this is mediated by twisting cochains:

\begin{thm}[Fundamental Theorem of Twisting Cochains, \cite{HMS}, \cite{LH}, \cite{LV}]\label{FT-TC}
Let $C \xrightarrow{\tau} A$ be a twisting cochain. The following are equivalent:
\begin{enumerate}[(i)]
\item $\Omega C \xrightarrow{\sim} A$ is a quasi-isomorphism,
\item $C \xrightarrow{\sim} BA$ is a weak equivalence,
\item $A \otimes^{\tau}\! C \otimes^{\tau} \! A \xrightarrow{\sim} {} _AA_A$ is a quasi-isomorphism.
\end{enumerate}
These further imply:
\begin{itemize}
\item[(iv)] $C \otimes^{\tau}\! A \xrightarrow{\sim} k_A$ is a quasi-isomorphism.
\end{itemize}
The reverse implication holds if $A$, $C$ are locally finite and either $A = k \oplus A^{\geq 2}$ and $C = C^{\geq 0}$ or $A = A^{\leq 0}$ and $C = k \oplus C^{\leq -2}$. % It's easy to cook up a qis C --> BA from (iv), but this needs to be a weak equivalence. That's basically it.
\end{thm}

Call $\tau$ acyclic if it satisfies the equivalent conditions $(i)$, $(ii)$ and $(iii)$. As seen above, the universal twisting cochains are both acyclic. The cap action $\tau \frown -$ defines a map of dg algebras ${\rm Hom}^{\tau}(C, A) \to {\rm End}_{A-A}(A\otimes^\tau\! C \otimes^\tau \! A)$ which is a quasi-isomorphism when $\tau$ is acyclic. See e.g. C. Negron's thesis \cite{MR3438933} for more details, where this was exploited to perform new computations of Hochschild cohomology rings. Hochschild cohomology for $A_\infty$-algebras will be defined through this observation. 

\begin{exmp}
Suppose $V = V^{\leq 0}$ or $V=V^{\geq 2}$ is  locally finite and let $A = Sym(V)$ and $C = Sym^{co}(sV)$. Then $\tau: Sym^{co}(sV) \twoheadrightarrow sV \to V \hookrightarrow Sym(V)$ is an acyclic twisting cochain, which gives rise to a small model $Sym^{co}(sV) \xrightarrow{\sim} BSym(V)$. The resolutions $(iii)$ and $(iv)$ are the classical Koszul resolutions. %Not necessary? fix.
\end{exmp}

\subsection{Augmented A$_{\infty}$-algebras and $A_\infty$-coalgebras}\label{Ainfinitydef}
The above example shows that one can sometimes expect small models $C \xrightarrow{\sim} BA$ and small semifree resolutions $\Omega C \xrightarrow{\sim} A$ from an appropriate acyclic twisting cochain $\tau: C \to A$. However, if one restricts themselves solely to dg algebra and coalgebra pairs, such models may fail to be minimal in an appropriate sense. The theory of $A_\infty$-algebras and coalgebras rectifies this deficiency. For now, let $A = (A, d_A)$ and $C = (C, d_C)$ denote augmented complexes over $k$.

\begin{defn} An $A_\infty$-algebra structure on $(A, d_A)$ consists of degree $2-n$ operations $m_n: A^{\otimes n} \to A$ for $n \geq 2$, satisfying the quadratic Stasheff identities 
\[
\sum_{r+s+t = n} (-1)^{rs+t+1} m_{r+1+t} \circ (1^{\otimes r} \otimes m_s \otimes 1^{\otimes t}) = \partial(m_n)
\]
%Actually we don't need this sign because we can translate between this and the usual literature with $\partial=-m_1$. -- It's better to keep d_A = m_1 because we can freely reference Lefevre. Changing this would have small repercussions here and there which are more annoying.
satisfying the strict unitality conditions $\epsilon \circ m_n = 0$ and $m_{n}(-, ..., -, 1, -, ..., -) = 0$ for $n \geq 3$ and $m_2(a \otimes 1) = m_2(1 \otimes a) = a,\ a \in A$.
\end{defn}
The first two Stasheff identities state that $d_A$ is a derivation for $m_2$ and that $m_2 \circ (m_2 \otimes 1 - 1 \otimes m_2) = \partial(m_3)$ in ${\rm Hom}(A^{\otimes 3}, A)$, hence $m_2$ is associative up to homotopy. Differential graded algebras correspond to $A_\infty$-algebras for which $m_{n} = 0$ for $n \geq 3$.
\begin{defn}
An $A_\infty$-coalgebra structure on $(C, d_C)$ consists of degree $2-n$ operations $\Delta_n: C \to C^{\otimes n}$ for $n \geq 2$, satisfying the quadratic Stasheff identities
\[
\sum_{r+s+t = n} (-1)^{r+st+n} (1^{\otimes r} \otimes \Delta_{s} \otimes 1^{\otimes t}) \circ \Delta_{r+1+t} = \partial(\Delta_n)
\]
as well as the strict counitality conditions $\Delta_n \circ \epsilon = 0$ and $(1 \otimes ... \otimes \eta \otimes ... \otimes 1) \circ \Delta_{n} = 0$ for $n \geq 3$ and $(\eta \otimes 1) \circ \Delta_2 = (1 \otimes \eta) \circ \Delta_2 = {\rm id}$. Furthermore, we require that the map $\sum_{n \geq 2}\Delta_n: C \to \prod_{n \geq 2} C^{\otimes n}$ factor through $\bigoplus_{n \geq 2} C^{\otimes n}$.
\end{defn}
As before, the Stasheff identities imply that $d_C$ is a coderivation for $\Delta_2$, that $\Delta_2$ is coassociative up to a coboundary given by $\Delta_3$, and so on.

% rewrite more symmetrically.
$A_\infty$-(co)algebras have natural (co)bar constructions. Given an $A_\infty$-coalgebra $C$, let $d_\Omega$ be the derivation on $\Omega C = T^a(s^{-1}\overline{C})$ generated by the sum of the maps
\[
s^{-1}\overline{C} \xrightarrow{d_{\Omega}^{(n)}} (s^{-1}\overline{C})^{\otimes n} \hookrightarrow \Omega C
\]
where $d_{\Omega}^{(1)} = d_{s^{-1}\overline{C}} = -s^{-1}d_Cs$ and $d_{\Omega}^{(n)} = -(s^{-1})^{\otimes n}\Delta_n s$ for $n \geq 2$. We will write $d_\Omega = d_{\Omega}^{+} + d_{\Omega}^{(1)}$ to separate the operations from the differential when convenient. Then $d_\Omega^2 = 0$ is equivalent to the Stasheff identities and we may think of the $\Delta_n$ as components of a differential on the free algebra $\Omega C$ extending $d_C$. Similarly for an $A_\infty$-algebra $A$, define $BA := (T^{co}(s\overline{A}), b)$ with $b$ the coderivation cogenerated by the sum of maps
\[
BA \xrightarrow{\pi_n} (s\overline{A})^{\otimes n} \xrightarrow{b_n} s\overline{A}
\]
where $b_1 = -sd_As^{-1}$ and $b_n = -s m_n (s^{-1})^{\otimes n}$ for $n \geq 2$. Again, we will often write $b = b_{+} + b_1$ to separate the operations and the differential. Again $b^2 = 0$ is equivalent to the Stasheff identities, and we may think of the $m_n$ as components of a codifferential on the cofree coalgebra $BA$ extending $d_A$.

\begin{defn} An $A_\infty$-morphism of $A_\infty$-algebras $\phi: A \to A'$ consists of degree $1-n$ maps $\phi_n: A^{\otimes n} \to A'$ for $n \geq 1$, satisfying quadratic identities
\[
\sum_{r+s+t=n}(-1)^{rs + t}\phi_{r+1+t} \circ (1^{\otimes r} \otimes m_s \otimes 1^{\otimes t}) \ - \sum_{i_1 + ... + i_k = n} (-1)^{u} m'_k \circ (\phi_{i_1} \otimes ... \otimes \phi_{i_k}) = \partial(\phi_n)
\]
where $u = (k-1)(i_1 - 1) + (k-2)(i_2 - 1) + \cdots + 2(i_{k-2} - 1) + (i_{k-1} - 1)$, as well as the strict unitality conditions $\phi_{n}(-, ..., -, 1, -, ..., -) = 0,\ n \geq 2$ and $\phi_1(1_A) = 1_{A'}$. Equivalently, the shifted maps $s\phi_n(s^{-1})^{\otimes n}:(s\overline{A})^{\otimes n} \to s\overline{A}'$ uniquely lift by Prop. \ref{universalcoalg} to a coalgebra morphism $\Phi: BA \to BA'$, and the Stasheff identities correspond to $b' \circ \Phi = \Phi \circ b$.
\end{defn}
Unpacking these identities shows that $\phi_1$ is a chain-map of complexes, $\phi_1$ is an algebra homomorphism for $m_2$ up to a homotopy given by $\phi_2$, and so on. Similarly, for coalgebras we have:

\begin{defn} An $A_\infty$-morphism of $A_\infty$-coalgebras $\phi: C \to C'$ consists of degree $1-n$ maps $\phi_n: C \to C'^{\otimes n}$ for $n \geq 1$ such that $\sum_{n \geq 1} \phi_n: C \to \prod_{n \geq 1} C'^{\otimes n}$ factors through $\bigoplus_{n \geq 1} C'^{\otimes n}$, and which satisfy certain quadratic identities. Equivalently, the $\phi_n$ assemble into a chain-map of dg algebras $\Phi: \Omega C \to \Omega C'$, with the quadratic identities corresponding to $d_{\Omega'} \circ \Phi = \Phi \circ d_\Omega$.
\end{defn}

We say that $\phi$ is strict if $\phi_n = 0$ for $n \geq 2$. The data of an $A_\infty$-morphism always corresponds to that of a dg (co)algebra morphism on (co)bar constructions; we will think of them as the same, but use capital letters for the later. Both the augmentation $A \xrightarrow{\epsilon} k$ and coaugmentation $k \xrightarrow{\epsilon} C$ are $A_\infty$-morphisms.

\begin{defn} A homotopy equivalence of two $A_\infty$-morphisms $f,g: A \rightrightarrows A'$ of $A_\infty$-algebras is given by a chain-homotopy $H: F \implies G$ of dg coalgebra maps on bar constructions, i.e. a degree $-1$ coderivation $H: BA \to BA'$ such that $\partial H = F - G$, and where $BA$ is a bicomodule over $BA'$ through $(F, G)$. 

A homotopy equivalence disassembles through prop. \ref{universalcoder} into a family of degree $-n$ maps $h_n: A^{\otimes n} \to A'$ satisfying quadratic identities. We mention that $h_1$ is a chain-homotopy for $f_1, g_1$ (see \cite[Sect. 1.2]{LH}).
\end{defn}
One defines homotopy equivalences of $A_\infty$-algebras in the usual way, and similarly for $A_\infty$-coalgebras.

$A_\infty$-categories are defined analogously: let $A$ be a dg quiver, augmented over $k_A$ so that $A = k_A \oplus \overline{A}$. An (augmented) $A_\infty$-category structure on $A$ consists of higher composition operations $m_n: A^{\otimes n} \to A$ of degree $2-n$ for $n \geq 2$ satisfying the Stasheff identities and the strict unitality condition. Equivalently, the $m_n$ assemble into codifferentials on the cofree cocomplete cocategory $BA := T^{co}(s\overline{A})$. As categories in nature are rarely augmented\footnote{Admitting a functor $A \to k_A$ to the discrete category $k_A$ forces isomorphic objects to be equal in $A$. Endomorphism algebras are also usually not augmented.}, one often relaxes the strict unitality condition and works with $B'\!A := T^{co}(sA)$ instead. See \cite[Sect. 5]{LH}, and section \ref{unreduced} below.

\subsubsection*{Finiteness, connectedness and dualisability}
Algebras and coalgebras are never perfectly dual; this is usually remedied by imposing standard finiteness and connectedness assumptions, whose effects we list here.

% Rename, rewrite, shrink text.
The linear dual of a $A_\infty$-coalgebra $C$ inherit an $A_\infty$-algebra structure in the natural way by dualizing $\Delta_n$, giving {\small $m_n: (C^*)^{\otimes n} \to (C^{\otimes n})^* \to C^*$}. Conversely for a locally finite $A_\infty$-algebra $A$, the canonical map is invertible and we can define coproducts on $A^*$ as {\small $\Delta_n: A^* \to (A^{\otimes n})^* \to (A^*)^{\otimes n}$} but {\small $\sum_{n \geq 2}\Delta_n: A^* \to \prod_{n \geq 2}(A^*)^{\otimes n}$} may not factor through {\small $\bigoplus_{n \geq 2} (A^*)^{\otimes n}$}. We call an $A_\infty$-algebra $A$ {\bf dualisable} if $A$ is locally finite and the above map does factor through the direct sum. 

We now list some standard finiteness assumptions:
\begin{itemize}
\item[(i)] An algebra $A$ is {\bf weakly connected} if ${\rm H}A$ is locally finite, and either negatively graded $({\rm H}A = {\rm H}A^{\leq 0})$ or simply-connected $({\rm H}A = k \oplus {\rm H}A^{\geq 2}$). This is equivalent to local finiteness of $B({\rm H}A)$.
\item[(i)'] A coalgebra $C$ is {\bf weakly connected} if ${\rm H}C$ is locally finite, and either positively graded $({\rm H}C = {\rm H}C^{\geq 0})$ or simply-coconnected $({\rm H}C = k \oplus {\rm H}C^{\leq -2})$. This is equivalent to local finiteness of $\Omega({\rm H}C)$.
\end{itemize}
Weak connectedness is not closed under taking (co)bar constructions, hence we will need the following:
\begin{itemize}
\item[(ii)] $A$ (resp. $C$) is {\bf strongly connected} if ${\rm H}A$ is locally finite, and either coconnected $({\rm H}A = k \oplus {\rm H}A^{\leq -1})$ or simply-connected $({\rm H}A = k \oplus {\rm H}A^{\geq 2})$ (resp. for ${\rm H}C$). 
\end{itemize}

As all previous statements only involve the cohomology, we call an $A_\infty$-(co)algebra \textbf{minimal} if $d = 0$, so that $A = {\rm H}A$ and $C = {\rm H}C$. We shall be concerned primarily with minimal algebras and coalgebras.
\begin{lem}\label{connectedness} The following hold for $A$ minimal:
\begin{itemize}
\item[(i)] Weakly connected $A_\infty$-algebras $A$ are dualisable, and $(BA)^* \cong \Omega(A^*)$ as dg algebras.  
\item[(ii)] Strong connectedness is preserved by applying $B$, $\Omega$, and taking linear duals.
\item[(iii)] Strongly connected $A_\infty$-algebras $A$ satisfy $(B(BA)^*)^* \cong \Omega BA$.
\end{itemize}
\end{lem}
Indeed, statement (i) follows since the natural pairing between $BA$ and $\Omega(A^*)$ is perfect whenever both spaces are locally finite (\cite[Ch. 19]{FHTbook}), and the codifferential on $BA$ gives rise to a differential on $\Omega(A^*) \cong (BA)^*$ rather than on the completion $\widehat{\Omega}(A^*)$. This differential encodes coproducts on $A^*$ as above, showing that $A$ is dualisable. The other statements are straightforward.

Strongly connected algebras are the best setting in which to study Koszul duality.\footnote{These finiteness assumptions are of course superfluous if one works with coalgebras instead, following Lef\`evre-Hasegawa.}

\subsubsection*{Quasi-isomorphisms and weak equivalences}

One is often only interested in dg algebras up to quasi-isomorphisms, and thus in the underlying homotopy type. While $A_\infty$-algebras extend the notion of dg algebras, they  share the same homotopy types, simply allowing for a more flexible presentation of the homotopy category. We review some basic notions along this direction.
\begin{defn} An $A_\infty$-morphism of $A_\infty$-algebras $\phi: A \to A'$ (resp. $A_\infty$-coalgebras $\phi: C \to C'$) is a quasi-isomorphism if $\phi_1$ is a quasi-isomorphism of complexes. An $A_\infty$-morphism of $A_\infty$-coalgebras $\phi: C \to C'$ is a weak equivalence if $\Phi: \Omega C \xrightarrow{\sim} \Omega C'$ is a quasi-isomorphism of dg algebras.
\end{defn}

It is typical to use hypotheses that insure weak equivalences coincide with quasi-isomorphisms. However, our range of applications falls outside of their scope, and we give here a more complete treatment. % rephrase.

\begin{defn} 
Let $F = F^{*}C$ be an increasing filtration on an $A_\infty$-algebra $C$. We say that $F$ is admissible if it is lower bounded, exhaustive and if the the induced filtration on $\overline{C}$ is lowered by the operations, in that $\Delta_n: F^i\overline{C} \to (F^{i-1}\overline{C})^{\otimes n}$. An $A_\infty$-coalgebra $C$ is cocomplete if it admits an admissible filtration. If $C$ is a dg coalgebra, any admissible filtration forces the primitive filtration to be exhaustive.
\end{defn}

\begin{defn} Dually, an admissible filtration $F = F^{*}A$ on an $A_\infty$-algebra $A$ is a decreasing, complete filtration such that the induced filtration on $\overline{A}$ is lowered by the operations, in that $m_n: (F^i\overline{A})^{\otimes n} \to F^{i+1}\overline{A}$. We say that $A$ is complete if it admits an admissible filtration.
\end{defn}

Any morphism $f: C \to C'$ of dg coalgebras is easily seen to preserve the primitive filtration; more generally, for $C$, $C'$ cocomplete $A_\infty$-coalgebras, consider an $A_\infty$-morphism $f: C \to C'$ such that $f_1$ preserves some admissible filtration. We say that $f$ is a filtered quasi-isomorphism if ${\rm gr}(f_1): {\rm gr}(C) \to {\rm gr}(C')$ is a quasi-isomorphism. 

Recall that for dg algebras and cocomplete dg coalgebras, the adjunction counit $A \xleftarrow{\sim} \Omega BA$ is a quasi-isomorphism and the counit $C \xrightarrow{\sim} B \Omega C$ a weak equivalence \cite[Sect. 1.3]{LH}. This is the basis for the following:
\begin{prop}[Lef\`evre-Hasegawa {\cite[Sect. 1.3]{LH}} ]\label{LHtheorem} Let $A$, $A'$ be $A_\infty$-algebras and $C$, $C'$ be $A_\infty$-coalgebras. The following hold:
\begin{enumerate}[(i)]
\item An $A_\infty$-morphism $\phi: A \to A'$ is a quasi-isomorphism if and only if  $\Phi: BA \to BA'$ is a weak equivalence of dg coalgebras if and only if $\Phi$ is in fact a filtered quasi-isomorphism. 
\item The $A_\infty$-morphism $A \to \Omega BA$ given by the adjunction unit $BA \xrightarrow{\sim} B \Omega BA$ is a quasi-isomorphism. \label{A-infinity-qis}
\item The $A_\infty$-morphism $C \leftarrow B\Omega C$ given by the adjunction counit $\Omega C \xleftarrow{\sim} \Omega B \Omega C$ is a weak equivalence, and is a quasi-isomorphism if $C$ is a cocomplete $A_\infty$-coalgebra.
\item If $C$ and $C'$ are cocomplete dg coalgebras, then any weak equivalence $C \to C'$ is a quasi-isomorphism. The converse holds if $C$ and $C'$ are locally finite, and both satisfy either $C = C^{\geq 0}$ or $C = k \oplus C^{\leq -2}$ (resp. for $C'$).
\item If $A$ and $A'$ are dg algebras, then $A$ is $A_\infty$-quasi-isomorphic to $A'$ if and only if they can be joined by a sequence of dg quasi-isomorphisms.
\end{enumerate}
\end{prop} 
\begin{proof} The key idea for $(i)$, $(ii)$ and $(iii)$ is to filter away the $A_\infty$-structure and reduce to a statement concerning tensor (co)algebras, then use the Eilenberg-Moore comparison theorem.

Part $(i)$ is proved mainly in \cite[1.3.3.5]{LH}, with some implications in 1.3.2.2 and 1.3.2.6. Part $(ii)$ is shown in 1.3.3.6 and follows from $(i)$. Part $(iii)$ is implicit in 1.3.2.3 and 1.3.3.6, but we sketch the idea. Let $C$ have an admissible filtration, and filter $\Omega C$ by total filtration index. Then $P: \Omega B \Omega C \to \Omega C$ and its strict part $\rho_1: B \Omega C \to C$ are filtered morphisms; $\rho_1$ is the projection onto $C$, and it suffices to show that the projection $gr(\rho_1): B T^a(s^{-1}V) \to k \oplus V$ is a quasi-isomorphism where $V = \overline{C}$, which is straightforward as these both calculate ${\rm Tor}^{T^a(s^{-1}V)}(k, k)$. See 1.3.2.3 for further details. 

Lastly, $(iv)$ is given 1.3.2.2 and $(v)$ follows by replacing $A \xleftarrow{\sim} \Omega BA$ and using $(i)$.
\end{proof} 

As immediate consequence of $(ii)$ and $(v)$, dg algebras and $A_\infty$-algebras share the same homotopy types. We put special emphasis on the following cases, which we shall need in the proof of the main result. Let $\rho: B \Omega C \to C$ be the $A_\infty$-morphism in $(iii)$ above:

\begin{cor}\label{HHKoszulhypotheses} The following hold for $A$ minimal:
\begin{enumerate}[(i)]
\item If $A$ is strongly connected, then $A \xrightarrow{\sim} \Omega B A$ dualises to a quasi-isomorphism $B \Omega (A^*) = (\Omega B A)^* \to A^*$ which agrees with $\rho$.
\item More generally if $A$ is a weakly connected, complete $A_\infty$-algebra, then $\rho: B \Omega (A^*) \to A^*$ is a quasi-isomorphism. This holds in particular if $A = A^0$ has nilpotent augmentation ideal, e.g. $A$ is a finite-dimensional quiver path algebra, or the group algebra of a finite $p$-group in characteristic $p$.
\end{enumerate}
\end{cor}

\subsubsection*{Homotopy Transfer and Inverse Function Theorems}
$A_{\infty}$-algebras and coalgebras satisfy the following two important theorems, to be used throughout this paper (see e.g. \cite{Huebschmann} or \cite{LV} for proofs). 

\begin{thm}[Homotopy Transfer Theorem]\label{HTT} Let $A$ be an $A_{\infty}$-algebra, and $V$ a complex over $k$ onto which $(A, d)$ deformation retracts via $\phi_1: (V, d_V) \leftrightarrows (A, d): \psi_1$. Then there is an $A_\infty$-algebra structure on $V$ such that $\phi_1, \psi_1$ are the strict part of an inverse pair of quasi-isomorphisms $\phi: V \leftrightarrows A: \psi$ with $\psi \phi = 1$.
\end{thm} % Reference?

\begin{cor}[Kadeishvili's Minimal Model Theorem]\label{MMT} Each $A_{\infty}$-algebra over a field $k$ admits a minimal model; there is a minimal $A_{\infty}$-structure on the cohomology ${\rm H}A$, along with $A_{\infty}$-quasi-isomorphisms $\phi :{\rm H}A \leftrightarrows A:\psi$ such that $\psi\phi=1$. This structure is unique up to (non-canonical) $A_{\infty}$-isomorphism. 
\end{cor}

The first part follows from the Transfer Theorem, as every complex of vector spaces over $k$ deformation retracts onto its cohomology. The uniqueness follows since then any two choices of minimal models can be joined by a straight quasi-isomorphism $({\rm H}A, m) \xrightarrow{\sim} ({\rm H}A, m')$, which is readily seen to be an isomorphism. In a similar vein, one can prove:

\begin{thm}[Inverse Function Theorem]\label{IFT}
Let $\phi: A \xrightarrow{\sim} A'$ be an $A_{\infty}$-quasi-isomorphism of $A_{\infty}$-algebras. Then $\phi$ admits a homotopy inverse $A_{\infty}$-quasi-isomorphism $\psi: A' \xrightarrow{\sim} A$. In particular $\phi_1, \psi_1$ are inverses on cohomology.
\end{thm} % Reference?

Weakly connected\footnote{As before this is to guarantee completeness of $\Omega HC$, which insures convergence of the series used in the homological perturbation lemma.} coalgebras satisfy the analogous results; in particular
\begin{thm} Let $C$ be a weakly connected dg (or $A_\infty$) coalgebra. Then there is an $A_\infty$-coalgebra structure on ${\rm H}C$ along with an inverse pair of weak-equivalences $\phi: {\rm H}C \leftrightarrows C: \psi$ such that $\phi \psi = 1$. This structure is unique up to non-canonical weak-equivalence.
\end{thm}

\subsection{Models for the Hochschild cochain complex}
\label{models-for-HH}

Let $A$ temporarily be a dg algebra. The Hochschild cochain complex of $A$ is defined as the twisted convolution algebra ${\rm Hom}^{\pi}\!(BA, A)$ whose differential $\partial_\pi(f) = \partial(f) + [\pi, f]$ is
\begin{align*}
\partial_\pi(f)[a_1 | ... | a_n] & = d_A f[a_1 | ... | a_n] \\
                                 & -(-1)^{|f|}\sum_{i=1}^{n}(-1)^{\epsilon_i} f[a_1 | ... | d_A(a_i) | ... | a_n]\\ & -(-1)^{|f|}\sum_{i=1}^{n-1}(-1)^{\epsilon_i + |a_i|} f[a_1 | ... | a_i a_{i+1} | ... | a_n]\\
                                 & + (-1)^{|f|(|a_1|+1)}a_1 f[a_2 | ... | a_n] -(-1)^{|f|}(-1)^{\epsilon_n} f[a_1 | ... | a_{n-1}] a_n.
\end{align*}
where $\epsilon_i = (i-1) + |a_1| + ... + |a_{i-1}|$. These formulas have straightforward extensions to the case of $A_\infty$-algebras. However, the definition in terms of twisted hom complexes gives additional insight and computational tools, and for this we shall impose on the reader's patience and discuss further generalities on twisting cochains.

From now on, let $C$ be a cocomplete dg coalgebra and $A$ an $A_\infty$-algebra; we will state results for the dual case at the end of the section. Define the higher cup products
\[
M_n: {\rm Hom}(C, A)^{\otimes n} \to {\rm Hom}(C, A)
\] 
for $n \geq 2$ by the formula
\[
M_n(f_1, ..., f_n) = m_n(f_1 \otimes ... \otimes f_n)\Delta^{(n)}
\] 
The $M_n$ induce on ${\rm Hom}(C, A)$ the structure of an augmented $A_\infty$-algebra \cite[lemma 8.1.1.4]{LH}\footnote{Lef\'evre-Hasegawa works with non-unital algebras and coalgebras, but the above $M_n$ are easily seen to satisfy the strict unitality conditions.}. As before, we say that $\tau \in {\rm Hom}^1(C, A)$ is a twisting cochain if $\tau$ vanishes on $k$, takes $\overline{C}$ to $\overline{A}$ and satisfies
\begin{equation}\label{highertw}
\partial(\tau) + \sum_{n \geq 2} M_n(\tau, ..., \tau) = 0
\end{equation}
The sum is locally finite since $C$ is cocomplete. The definition is of course chosen so that the following holds:
\begin{prop} Let $\tau: C \to A$ be a degree $1$ map vanishing on $k$ and taking $\overline{C}$ to $\overline{A}$. Then $\tau$ is a twisting cochain if and only if the associated coalgebra morphism $\phi_\tau: C \to BA$ is a map of dg coalgebras. In other words, 
\[
{\rm Tw}(C, A) \cong {\rm \bf coAlg}(C, BA).
\]
\end{prop}
A twisting cochain $C \xrightarrow{\tau} A$ is acyclic if $C \xrightarrow{\sim} BA$ is a weak equivalence; we always write $\pi: BA \to A$ for the universal twisting cochain corresponding to ${\rm id}_{BA}$, and note that $A_\infty$-morphisms $\phi: A \to A'$ can be thought of as twisting cochains $\tau = \pi' \Phi: BA \to A'$. Given a twisting cochain $\tau$, we shall provide two models for the twisted differential on ${\rm Hom}^{\tau}\!(C, A)$. Following the structure of Section \ref{tw}, we first introduce higher commutators on the $A_\infty$-algebra ${\rm Hom}(C, A)$.

\begin{defn} If $A$ is an $A_\infty$-algebra we define the {\bf higher commutators}  $[-;-]_{p,q}:A^{\otimes p}\otimes A^{\otimes q}\to A$ to be the degree $2-(p+q)$ maps given by
\[
[v_1,...,v_p;v_{p+1},...,v_{p+q}]_{p,q} = \sum_{\sigma \in {\rm sh}(p,q)} (-1)^{|\sigma|}(-1)^{|\sigma; v|} m_{p+q}(v_{\sigma^{-1}(1)}\otimes ... \otimes v_{\sigma^{-1}(p+q)}).
\]
Of course these can be defined using the shuffle product on $BA$ (being careful of signs). The structure of these commutators is finer than the associated $L_\infty$-algebra of section \ref{LieCom-section}; it is not clear where they stand algebraically. We are only interested in $[-;-]_{1q}$ and $[-;-]_{p1}$. Note that $[-;-]_{1,1}$ is the usual commutator for $m_2$. Note also that $[-;-]_{p,q}T=(-1)^{pq}[-;-]_{q,p}$.
%
%\begin{align*}\label{highercommutators}
%[x; a_1, ..., a_{n-1}]_n & := \sum_{i=0}^{n-1} (-1)^{i}(-1)^{|x|(|a_1| + ... + |a_{i}|)} m_n(a_1, ..., a_{i}, x, a_{i+1}, ..., a_{n-1})\\
%\text{and}\quad[a_1, ..., a_{n-1};x]_n & := \sum_{i=0}^{n-1} (-1)^{n-1-i}(-1)^{|x|(|a_{i+1}| + ... + |a_{n-1}|)} m_n(a_1, ..., a_{i}, x, a_{i+1}, ..., a_{n-1}).
%\end{align*}
%Note that $[x;a]_2 = [x, a] = m_2(x, a) - (-1)^{|x| |a|} m_2(a, x)$ is the graded commutator for $m_2$.
\end{defn}

The twisted differential $\partial_{\tau}$ on ${\rm Hom}(C, A)$ is defined using higher commutators for the convolution operations:
\begin{align*}  
\partial_{\tau} = \partial + \sum_{n \geq 1} [\tau, ..., \tau; -]_{n,1}.
%\\
                   %& = \partial(f) + \sum_{n \geq 2}\sum_{i=0}^{n-1}(-1)^{|sf|(n-1-i)} m_n(\tau^{\otimes i} \otimes f \otimes \tau^{\otimes n-1-i})\Delta^{(n)}
\end{align*}
We delay the proof that $\partial_\tau^2 = 0$, or that the above sum converges at all, until we establish a better description of this differential. Let ${\rm Hom}^\tau\!(C, A) =  \big({\rm Hom}(C, A), \partial_\tau \big)$. For now let us record a simple proposition, which follows since $\tau(1) = 0$ and $\tau(\overline{C}) \subseteq \overline{A}$.
\begin{prop} \label{seq} The subspaces ${\rm Hom}(\overline{C}, A)$, ${\rm Hom}(C, \overline{A})$ are closed under $\partial_\tau$. There are induced short exact sequences of complexes
\begin{align*}
& 0 \to  {\rm Hom}^{\tau}(\overline{C}, A) \to {\rm Hom}^{\tau}(C, A) \to  A \to 0 \\ 
& 0 \to {\rm Hom}^{\tau}(C, \overline{A}) \to {\rm Hom}^{\tau}(C, A) \to C^* \to  0.
\end{align*}
\end{prop}

Let us now establish a second model for the twisted hom complex. The twisting cochain $\tau:C \to A$ lifts to a dg coalgebra morphism $\phi = \phi_\tau: C \to BA$, giving $C$ the structure of a $BA$-bicomodule. By prop. \ref{universalcoder}, any map $g: C \to s\overline{A}$ uniquely extends to a coderivation $D_g: C \to BA$, with inverse map $D \mapsto \pi_1 D$. This identifies
\[
{\rm Hom}(C, A) = s^{-1}{\rm Hom}(C, s\overline{A}) \oplus {\rm Hom}(C, k) \cong s^{-1}{\rm coder}(C, BA) \oplus C^*.
\]
The subspace ${\rm coder}(C, BA) \subseteq {\rm Hom}(C, BA)$ is closed under the differential $\partial_{{\rm Hom}(C, BA)}$. There is a map $ad: C^* \to {\rm coder}(C, BA)$ which sends functionals $f: C \to k$ to the inner coderivation $ad_f: C \xrightarrow{\Delta} C \otimes C \xrightarrow{\phi \otimes f - f \otimes \phi} BA$. It is readily seen to be a chain-map, and upon taking its cone we get
\[
{\rm Cone}(ad) = s^{-1}{\rm coder}(C, BA) \rtimes C^*,  \partial_{ad}.
\]

The cone differential $\partial_{ad}$ is given by
\[
\partial_{ad}(s^{-1}D + f) = ad_f - s^{-1}\partial_{{\rm Hom}(C, BA)}(D) + \partial_{C^*}(f).
\]

\begin{prop}\label{ident} The following hold:
\begin{itemize}
\item[(i)] The isomorphism ${\rm Hom}(C, A) \cong s^{-1}{\rm coder}(C, BA) \oplus C^*$ identifies $\partial_\tau$ with $\partial_{ad}$.
\item[(ii)] $\partial_\tau^2 = 0$.
\item[(iii)] The cone short exact sequences identifies with the short exact sequence \ref{seq}:
\[
\xymatrixrowsep{1pc}\xymatrix@1{0 \ar[r] & {\rm Hom}^{\tau}(C, \overline{A}) \ar@{=}[d] \ar[r] & {\rm Hom}^{\tau}(C, A) \ar@{=}[d] \ar[r] & C^* \ar@{=}[d] \ar[r] & 0 \\
0 \ar[r] & s^{-1}{\rm coder}(C, BA) \ar[r] & s^{-1}{\rm coder}(C, BA) \rtimes C^* \ar[r] & C^* \ar[r] & 0.}
\]
\end{itemize}
\end{prop}
\begin{proof} (i) First, assume that $f \in {\rm Hom}(C, k) \subset {\rm Hom}(C, A)$. The strict unitality condition $m_n(..., 1, ...) = 0$ means $ [ \tau, ..., \tau;f]_{n-1, 1} = 0$ for $n \geq 3$. The remaining part $[\tau;f]_{1,1} = [\tau, f]$ agrees with $\pi ad_f$ since $\pi \phi_\tau = \tau$, and so the result holds for $f \in {\rm Hom}(C, k)$. 

Now take $sf$ in the suspension $s{\rm Hom}(C, \overline{A}) = {\rm Hom}(C, s\overline{A}) \cong {\rm coder}(C, BA)$, with corresponding $D_{sf}$. We compute the projection of the suspended differential $-s\partial_{ad}s^{-1}$:
\begin{align*}\pi_1(-s\partial_{ad}s^{-1})(D_{sf}) & = \pi_1\partial_{{\rm Hom}(C, BA)}(D_{sf})\\ 
                                          & = \pi_1 b \circ D_{sf} - (-1)^{|sf|} sf \circ d_C\\
                                          & = -sd_As^{-1} \circ sf - (-1)^{|sf|} sf \circ d_C + \pi_1 b_{+}\circ D_{sf}\\
                                          & = -s(d_A \circ f - (-1)^{|f|} f \circ d_C) + \pi_1 b_{+}\circ D_{sf}
\end{align*}
We are left to calculate the ``twisted part'' $\pi_1 b_{+}\circ D_{sf}$. Recall by prop. \ref{universalcoalg}, \ref{universalcoder} that $\phi_\tau$ and $D_{sf}$ have components
\begin{align*}
 \pi_m \phi_\tau & = (s\tau)^{\otimes m} \Delta^{(m)}\\
 \pi_n D_{sf}    & = \sum_{i = 0}^{n-1} (\pi_{i} \otimes sf \otimes \pi_{n-1-i})(\phi_\tau \otimes id \otimes \phi_\tau) \Delta^{(3)}\\
                & = \sum_{i=0}^{n-1} (s\tau^{\otimes i} \otimes sf \otimes s\tau^{\otimes n-1-i}) \Delta^{(n)}\\
                & = \sum_{i=0}^{n-1} (-1)^{\epsilon_i}s^{\otimes n}(\tau^{\otimes i}\otimes f \otimes \tau^{n-1-i}) \Delta^{(n)}
\end{align*}
where $\epsilon_i = |f|(n-1-i) + i + 1 + 2 + ... + n-2$. Applying $\pi_1 b_n = -sm_n(s^{-1})^{\otimes n}$ to $D_{sf}$ and collecting signs then gives                
\begin{align*}
&  = -s\sum_{i=0}^{n-1}(-1)^{n-1-i} (-1)^{|f|(n-1-i)}m_n (\tau^{\otimes i} \otimes f \otimes \tau^{\otimes n-1-i}) \Delta^{(n)}
\end{align*}
which is easily seen to be $-s[ \tau, ..., \tau; f]_{n-1,1}$. This shows that $\partial_{ad}$ corresponds to $\partial_\tau$. Since $\partial_{ad}$ is the cone differential, (ii) follows and (iii) is an immediate consequence of the above identifications.
\end{proof}

With all this, we define the Hochschild cochain complex of an $A_\infty$-algebra $A$ with coefficients in $A'$, where $\phi: A \to A'$ is an $A_\infty$-morphism. Recall that this corresponds to a twisting cochain $\tau: BA \to A'$.
\begin{defn} The Hochschild cochain complex of $A$ is defined as $C^*(A, A) = {\rm Hom}^\pi(BA, A)$, and more generally $ C^*(A, A') = {\rm Hom}^\tau(BA, A')$ with coefficients in  $A'$. % When necessary, we will write $ C^{*}_\phi(A, A')$. Actually, whenever this subtlety arises I will just use the twisted hom notation.
Its cohomology is the Hochschild cohomology ${\rm HH}^*(A, A')$.\footnote{We shall not treat the case of general $A_\infty$-bimodule coefficients.}
\end{defn}

As in the differential graded case, one should be able to compute $C^*(A, A)$ from any model $C \xrightarrow{\sim} BA$. This essentially amounts to the naturality of twisted complexes. We need some preliminary work.

The twisted complex ${\rm Hom}^\tau(C, A)$ inherits a decreasing filtration $F^{*}_{co}$ from the primitive filtration on $C$:
\[
F_{co}^p{\rm Hom}^{\tau}(C, A) = {\rm Hom}^{\tau}(C, A)^{\geq p} = \{\varphi: C \to A \mid \varphi(C_{[p]}) = 0 \}.
\]
This filtration is complete since $C$ is cocomplete. Note that $F^{1}_{co}{\rm Hom}^\tau(C, A) = {\rm Hom}^\tau(\overline{C}, A)$, and so further terms live inside ${\rm Hom}^\tau(\overline{C}, A)$. In the case of the Hochschild cochain complex this will be called the {\bf weight filtration}, and it will be denoted
\[
F_\Pi^nC^*(A,A)={\rm Hom}^\pi\!(B_{\geq n}A,A)
\]
and similarly for other coefficients. The induced weight filtration on Hochschild cohomology is also denoted $F_\Pi^n {\rm HH}^*(A,A)$.

\begin{lem}\label{twlowers} The twisted part $\partial_\tau - \partial$ of the differential lowers the filtration:
\[
(\partial_\tau - \partial)F_{co}^p{\rm Hom}^{\tau}(C, A) \subseteq F_{co}^{p+1}{\rm Hom}^{\tau}(C, A).
\]
\end{lem}
\begin{proof}
Recall that $\partial_\tau - \partial$ sends $f$ to $\sum_{n \geq 1}[\tau, ..., \tau; f]_{n,1} = \sum_{n \geq 1}\sum_{i=0}^{n} \pm m_{n+1}(\tau^{\otimes i} \otimes f \otimes \tau^{\otimes n-i})\Delta^{(n+1)}$. For $p \geq 1$, take $f \in {\rm Hom}^\tau(C, A)^{\geq p} \subseteq {\rm Hom}^\tau(\overline{C}, A)$ and note that $f(1) = 0 = \tau(1)$ implies that
\[
m_{n+1}(\tau^{\otimes i} \otimes f \otimes \tau^{\otimes n-i})\Delta^{(n+1)} = m_{n+1}(\tau^{\otimes i} \otimes f \otimes \tau^{\otimes n-i}) \overline{\Delta}^{(n+1)}.
\]
It then follows from $\overline{\Delta}^{(n+1)} \big( \overline{C}_{[p]} \big) \subseteq \overline{C}_{[p-n]}^{\otimes n+1}$ that $[\tau, ..., \tau; -]_{n,1}$ sends $F_{co}^{p}$ to $F_{co}^{p+n}$. The case $p = 0$ follows directly from $\tau(1) = 0$.
\end{proof}
\begin{lem}\label{coalgspectralsequence} $F^{*}_{co}{\rm Hom}^\tau(C, A)$ induces a cohomology spectral sequence
\[
{\rm E}_1^{p,q} = {\rm H}^{p+q}{\rm Hom}(gr^p(C), A) \implies {\rm H}^{p+q}{\rm Hom}^{\tau}(C, A)
\]
\end{lem}
\begin{proof} The first page of the spectral sequence is given ${\rm H}^{p+q}\big({\rm Hom}^\tau(C, A)^{\geq p}/{\rm Hom}^\tau(C, A)^{\geq p+1}\big)$, and since $\partial_\tau - \partial$ decreases the filtration, $E_1^{p,q}$ becomes ${\rm H}^{p+q}\big({\rm Hom}(C, A)^{\geq p}/{\rm Hom}(C, A)^{\geq p+1}\big) \cong {\rm H}^{p+q}{\rm Hom}(gr^p(C), A)$.
\end{proof}

%\begin{rem}
%Convergence holds in certain situations, in which case the higher commutators give us an explicit understanding of the differentials. This will be investigated later.
%\end{rem}

As in the dg case, twisting cochains $C \xrightarrow{\tau} A$ are natural in $C$ and $A$: any dg coalgebra morphism $C' \xrightarrow{\phi} C$ yields a twisting cochain $C' \xrightarrow{\tau\phi} A$, while an $A_\infty$-morphism $A \xrightarrow{\psi} A'$ gives a twisting cochain $C \xrightarrow{\psi\tau} A'$ by composing
\[
C \xrightarrow{\phi_\tau} BA \xrightarrow{\Psi} BA' \xrightarrow{\pi} A'.
\]
We now show that twisted Hom complexes are natural in both arguments. Define $\phi^*: {\rm Hom}^\tau(C, A) \to {\rm Hom}^{\tau\phi}(C', A)$ as ${\rm Hom}(\phi, A)$ in the usual way, but we will define $\psi_*$ on each piece of ${\rm Hom}(C, A) = {\rm Hom}(C, \overline{A}) \oplus {\rm Hom}(C, k)$.

First, define $\psi_*: {\rm Hom}^\tau(C, \overline{A}) \to {\rm Hom}^{\psi\tau}(C, \overline{A}')$ by $f \mapsto \pi \Psi D_{sf}$, i.e. lifting to a coderivation $D_{sf}: C \to BA$ and pushing through $BA \xrightarrow{\Psi} BA' \xrightarrow{\pi'} A'$. Extend $\psi_*$ to ${\rm Hom}(C, k)$ as the identity. Assuming that $\psi_*$ is a chain-map, note that this is compatible with the decompositions \ref{seq}:
\[
\xymatrixrowsep{1pc}\xymatrix@1{ 0 \ar[r] & {\rm Hom}^{\tau}(C, \overline{A}) \ar[r] \ar[d]^{\psi_*} & {\rm Hom}^{\tau}(C, A) \ar[r] \ar[d]^{\psi_*} & C^* \ar[r] \ar@{=}[d] &  0\\
 0 \ar[r] & {\rm Hom}^{\psi\tau}(C, \overline{A}') \ar[r] & {\rm Hom}^{\psi\tau}(C, A') \ar[r] & C^* \ar[r] &  0.}
\]

\begin{prop}\label{cocompletecomparison} The morphisms $\phi^*$, $\psi_*$ are chain-maps, and there is a commutative diagram
\[
\xymatrixrowsep{1pc}\xymatrixcolsep{6pc}\xymatrix@1{{\rm Hom}^{\tau\phi}(C', A) \ar@{=}[d] & \ar[l]_{\phi^*} {\rm Hom}^{\tau}(C, A) \ar[r]^{\psi_*} \ar@{=}[d] & {\rm Hom}^{\psi\tau}(C, A') \ar@{=}[d]\\
s^{-1}{\rm coder}(C', BA) \rtimes (C')^* & \ar[l]_{\ \ \ {\rm Hom}(\phi, BA) \oplus \phi^*} s^{-1}{\rm coder}(C, BA) \rtimes C^* \ar[r]^{{\rm Hom}(C, \Psi) \oplus id} & s^{-1}{\rm coder}(C, BA')\rtimes C^*.}
\]
Furthermore, $\phi^*$ is a quasi-isomorphism whenever $\phi: C' \xrightarrow{\sim} C$ is a filtered quasi-isomorphism and $\psi_*$ is a quasi-isomorphism whenever $\psi: A \xrightarrow{\sim} A'$ is an $A_\infty$-quasi-isomorphism.
\end{prop}
\begin{proof}
The definitions of $\phi^*, \psi_*$ were chosen as to make the diagram commute, and it isn't difficult to see from $\partial_{ad}$ that the bottom row consist of chain-maps.

If $\phi: C \xrightarrow{\sim} C'$ is a filtered quasi-isomorphism, then the induced map on $E_1$ of the spectral sequence \ref{coalgspectralsequence} is the isomorphism ${\rm Hom}(gr(\phi), A): {\rm H}^*{\rm Hom}(gr(C), A) \xrightarrow{\sim} {\rm H}^*{\rm Hom}(gr(C'), A)$, and since the filtrations used are complete the Eilenberg-Moore comparison theorem applies and $\phi_*$ is a quasi-isomorphism.

If $\psi: A \xrightarrow{\sim} A'$ is an $A_\infty$-quasi-isomorphism we proceed similarly. By the five lemma and the above remark it suffices to check ${\rm Hom}^{\tau}(C, \overline{A}) \xrightarrow{\psi_*} {\rm Hom}^{\psi\tau}(C, \overline{A}')$. Write $\psi_*(f) = \pi \Psi D_{sf} = \psi_1 \circ f + \pi \Psi D^{\geq 2}_{sf}$. A similar argument as in \ref{twlowers} shows that the map $f \mapsto \pi \Psi D^{\geq 2}_{sf}$ lowers the $F_{co}^{*}$ filtration, and the remaining part $\psi_1 \circ (-)$ induces a quasi-isomorphism on $E_1$. By the Eilenberg-Moore comparison theorem we are done.
\end{proof}

Suppose that $\phi :A\to A'$ is a quasi-isomorphism of $A_\infty$-algebras. An immediate corollary of the proposition is that we have a canonical chain of quasi-isomorphisms $C^*(A',A')\xrightarrow{\phi^*}C^*(A,A')\xleftarrow{\phi_*}C^*(A,A)$, which we will make use of often. The formal composition $(\phi_*)^{-1}\phi^*$ will be denoted $C(\phi)$. It is an isomorphism in the homotopy category of complexes (and in fact of $A_\infty$-algebras, as we will see soon). In particular, $C(\phi)$ induces an isomorphism ${\rm HH}^*(A',A')\xrightarrow{\cong}{\rm HH}^*(A,A)$ which will be denoted ${\rm HH}(\phi)$. If $A\xrightarrow{\psi}A'\xrightarrow{\phi}A''$ are two quasi-isomorphisms then $C(\phi\psi)$ coincides with $C(\psi)C(\phi)$ in the homotopy category of chain complexes.

% Should I record this discussion as a corollary?
%\begin{cor} If $\phi:A\to A'$ is a quasi-isomorphism of $A_infty$-algebras then there is a canonical chain of...
%\end{cor}
% nahh

% A_\infty-coalgebras to dg algebras.
We now discuss the dual statements for twisting cochains $\tau: C \to A$ from $A_{\infty}$-coalgebras into complete dg algebras. Their demonstration being very similar to above, we shall be brief. Under this setup, ${\rm Hom}(C, A)$ is an $A_\infty$-algebra with higher cup products
\[
M_n(f_1 \otimes ... \otimes f_n) = m^{(n)} (f_1 \otimes ... \otimes f_n) \Delta_n
\]
and we define twisting cochains $\tau: C \to A$ by the same formula (\ref{highertw}). This is equivalent to the induced map $\phi_\tau: \Omega C \to A$ preserving differentials, and so the functor ${\rm Tw}(C, -)$ is corepresentable:
\[
{\rm \bf Alg}(\Omega C, A) \cong {\rm Tw}(C, A).
\]
We can take higher commutators against $\tau$ in the the $A_\infty$-algebra ${\rm Hom}(C, A)$ to form the twisted differential $\partial_{\tau} = \partial + \sum_{n \geq 1} [\tau, ..., \tau; -]_{n,1}$ just as before.
%\begin{align*}
%\partial_{\tau}(f) & = \partial(f) + \sum_{n \geq 1} [\tau, ..., \tau; f]_{n,1}
%                   & = \partial(f) + \sum_{n \geq 2}\sum_{i=0}^{n-1}(-1)^{|sf|(n-1-i)} m^{(n)}(\tau^{\otimes i} \otimes f \otimes \tau^{\otimes n-1-i})\Delta_n
%\end{align*}
The twisted complex ${\rm Hom}^{\tau}(C, A)$ admits a model in terms of derivations through the identification ${\rm Hom}(s^{-1}\overline{C}, A) \cong {\rm der}(\Omega C, A)$. This extends to an isomorphism ${\rm Hom}^{\tau}(C, A) \cong s^{-1}{\rm der}(\Omega C, A) \rtimes A$ where $s^{-1}{\rm der}(\Omega C, A) \rtimes A$ is defined to be the cone of $ad: A \to {\rm der}(\Omega C, A)$, $a \mapsto [-, a]$.\footnote{This unusual choice is forced by the previous conventions.}
%with cone differential $\partial'_\tau$ given by
%\[
%\partial'_\tau(s^{-1} D + x) = ad_x -s^{-1}\partial_{{\rm Hom}(\Omega C, A)}(D) + d_A(x)
%\]
%with $ad: A \to {\rm der}(\Omega C, A)$ sending $x$ to the inner derivation $ad_x$.

Completeness of $A$ induces a complete decreasing filtration $F^{p}_{alg}{\rm Hom}^{\tau}(C, A) = {\rm Hom}^{\tau}(C, A^{[p]})$. This filtration is lowered by $[\tau, ..., \tau; -]_{n,1}$, which by the standard Eilenberg-Moore comparison theorem proves:
\begin{prop}\label{completecomparison} Let $\phi: C' \to C$ be an $A_\infty$-morphism of $A_\infty$-coalgebras and $\psi: A \to A'$ a morphism of complete dg algebras. The twisted complex ${\rm Hom}^\tau(C, A)$ is natural in both arguments, with commutative diagram
\[
\xymatrixrowsep{1pc}\xymatrixcolsep{6pc}\xymatrix@1{{\rm Hom}^{\tau\phi}(C', A) \ar@{=}[d] & \ar[l]_{\phi^*} {\rm Hom}^{\tau}(C, A) \ar[r]^{\psi_*} \ar@{=}[d] & {\rm Hom}^{\psi\tau}(C, A') \ar@{=}[d]\\
s^{-1}{\rm der}(\Omega C', A) \rtimes A & \ar[l]_-{{\rm Hom}(\Phi, A) \oplus id} s^{-1}{\rm der}(\Omega C, A) \rtimes A \ar[r]^-{{\rm Hom}(\Omega C, \psi) \oplus \psi} & s^{-1}{\rm der}(\Omega C, A')\rtimes A'.}
\]
The map $\phi^*$ is a quasi-isomorphism whenever $\phi$ is a quasi-isomorphism, while $\psi_*$ is a quasi-isomorphism whenever $\psi$ is a filtered quasi-isomorphism.
\end{prop}

Finally, when $C$ is a cocomplete dg coalgebra and $A$ a dg algebra, the twisted complex ${\rm Hom}^\tau(C, A)$ is the twisted convolution algebra of Section \ref{tw}, and the results of this section give the ``adjunction''
\begin{equation}\label{adjunctionlemma}
s^{-1}{\rm der}(\Omega C, A)\rtimes A \ \cong \ {\rm Hom}^\tau(C, A) \ \cong \ s^{-1}{\rm coder}(C,BA)\rtimes C^*.
\end{equation}
The constructions presented in this section being entirely dual, we record down a natural consequence which is to be used Section $3$.% (Flagged for issue: A needs minimality for $(BA)^* = \Omega (A^*)$).
\begin{lem}\label{Dlemma} 
Let $C$ be a cocomplete dg coalgebra and $A$ a minimal weakly connected $A_\infty$-algebra, with a twisting cochain $\tau:C\to A$. When $C$ and $A$ are locally finite, dualizing gives a well-defined isomorphism
\[
{\bf D}:{\rm Hom}^\tau(C,A)\xrightarrow{\cong} {\rm Hom}^{\tau^*}\!\!(A^*,C^*).
\]
On the level of derivations, ${\bf D}$ is the natural dualization
\[
{\bf D}:s^{-1}{\rm coder}(C, BA)\rtimes C^* \to  s^{-1}{\rm der}(\Omega A^*,C^*)\rtimes C^*.
\]
In the case $C=BA$ the isomorphism ${\rm coder}(BA, BA)\to{\rm der}(\Omega A^*,\Omega A^*)$ is one of Lie algebras.
\end{lem}

\subsubsection*{Two Filtrations on Hochschild Cohomology}\label{filtrationsection}

We have already seen the weight filtration $F_\Pi^nC^*(A,A)={\rm Hom}^\pi\!(B_{\geq n}A,A)$. Let's start by proving

\begin{prop}\label{weightnaturality} If $\phi: A \xrightarrow{\sim} A'$ is a quasi-isomorphism of  $A_\infty$-algebras  
then  ${\rm HH}(\phi)$ preserves the weight filtration. Hence ${\rm HH}^*(A,A)$ and $ {\rm HH}^*(A',A')$ are isomorphic as (weight) filtered graded vector spaces.
\end{prop}

\begin{proof}
Any such quasi-isomorphism factors as a sequence of homotopy retracts (this follows from \ref{MMT}, or indeed is true in any model category, knowing that $BA$ and $BA'$ are cofibrant in the model structure of \cite{LH}), so we may assume we have $\psi:A'\to A$ such that $\psi\phi=1$. 
We get the following commutative diagram
\[
    \xymatrix{   
    			    & {\rm Hom}^{\pi}\!(BA,A)  \ar@{=}[d] \ar[dr]^-{\psi^*} \ar[dl]_-{\phi_*}  & \\ 
    			    {\rm Hom}^{\phi\pi}\!(BA,A')\ar[r]^-{\psi_*} &   {\rm Hom}^{\pi}\!(BA,A)  &  {\rm Hom}^{\Psi\pi}\!(BA',A) \ar[l]_-{\phi^*}\\
    			    & {\rm Hom}^{\pi}\!(BA',A').   \ar[ur]_-{\psi_*} \ar[ul]^-{\phi^*} & 
		 }
\]
Since all the maps are quasi-isomorphisms it follows that $C(\phi)=(\phi_*)^{-1}\phi^*$ and  $C(\psi)=(\psi_*)^{-1}\psi^*$ are inverse isomorphisms in the homotopy category of chain complexes. In particular ${\rm HH}(\phi)$ and  ${\rm HH}(\psi)$ are inverse isomorphisms on the level of cohomology. Hence the second statement of the proposition will follow from the first.

Certainly $\phi^*:C^*(A',A')\to C^*(A,A')$ preserves the weight filtration. The reason $(\phi_*)^{-1}$ preserves the weight filtration on the level of cohomology is that we have for each $n$ a commutative diagram
\[
    \xymatrix@R=4mm@C=15mm{   
    			    F^n_{\Pi}C^*(A,A) \ar[d] \ar[r]^{ \phi_*|_{F^n_\Pi} } & F^n_{\Pi}C^*(A,A')\ar[d] \\
    			    C^*(A,A) \ar[r]^{\phi_*} & C^*(A,A')
		 }
\]
and the restriction $\phi_*|_{F^n_\Pi}$ is a quasi-isomorphism by exactly the same argument as in proposition \ref{cocompletecomparison}.
\end{proof}

The reason using the primitive filtration on $BA$ works this well on Hochschild cohomology is essentially that $BA$ is cofree as a graded coalgebra. If we want another filtration coming from the radical filtration on $A$, we first need to replace $A$ with something nicer. Thus, we say that a dg algebra $\mathcal{A}$ is {\bf semi-free} if it is of the form $(TV,d)$ where $V$ is a vector space with an exhaustive filtration $0=V^{(0)}\subseteq V^{(1)}\subseteq ...$ such that $d(V^{(n)})\subseteq T(V^{(n-1)})$ for all $n\geq 1$. Equivalently, $\mathcal{A}=\Omega C$ for some cocomplete $A_\infty$-coalgebra $C$. 
Every $A_\infty$-algebra is quasi-isomorphic to a semi-free dg algebra, indeed, we can use $\Omega BA$. Importantly these dg algebras enjoy a lifting property (they are cofibrant in the model structure of Hinich \cite{Hinich}) 
 which means that a chain of (possibly non-strict) quasi-isomorphisms between two semi-free dg algebras can be replaced by a sequence of homotopy retracts (consisting of strict dg algebra maps). This allows us to use the same argument as above to establish well-definedness of the following filtration.

\begin{defn}
Let $A$ be an $A_\infty$-algebra, choose a (possibly non-strict) quasi-isomorphism $\mathcal{A}\xrightarrow{\sim}A$ from a semi-free dg algebra. The {\bf shearing filtration} $F^*_\chi {\rm HH}^*(A,A)$ on the Hochschild cohomology of $A$ is defined through the isomorphism ${\rm HH}(\phi):{\rm HH}^*(A,A)\cong {\rm HH}^*(\mathcal{A},\mathcal{A})$. The shearing filtration on ${\rm HH}^*(\mathcal{A},\mathcal{A})$ by definition descends from the filtration
\[
F_\chi^nC^*(\mathcal{A},\mathcal{A})={\rm Hom}^\pi\!(B\mathcal{A},\mathcal{A}^{[n]}).
\]
\end{defn}
The filtration could also be computed just from ${\rm Hom}^\pi\!(BA,\mathcal{A}^{[n]})$. We will also make use of the relative shearing filtration  $F_\chi^nC^*(\mathcal{A},\mathcal{A}')={\rm Hom}^\pi\!(B\mathcal{A},\mathcal{A}'^{[n]})$ for a given dg algebra map $\mathcal{A}\to \mathcal{A}'$.

\begin{prop}\label{shearingnaturality} The shearing filtration on ${\rm HH}^*(A,A)$ is independent of choice of $\mathcal{A}\xrightarrow{\sim} A$. If $\phi: A \xrightarrow{\sim} A'$ is a quasi-isomorphism of  $A_\infty$-algebras then ${\rm HH}(\phi):{\rm HH}^*(A',A')\xrightarrow{\cong}{\rm HH}^*(A,A)$ is an isomorphism of (shearing) filtered graded vector spaces.
\end{prop}

% I could just say "the proof is similar".
\begin{proof}
Two quasi-isomorphic semi-free dg algebras $\mathcal{A}$ and $\mathcal{A}'$ are connected by a sequence of strict dg algebra homotopy retractions. So assume we have $\phi:\mathcal{A}\leftrightarrows\mathcal{A}':\psi$ such that $\psi\phi=1$. By the same argument as in \ref{weightnaturality}, $C(\phi)$ and $C(\psi)$ are inverse isomorphisms in the homotopy category. Now note that $C(\phi)^{-1}=(\phi^*)^{-1}\phi_*$ preserves the shearing filtration. This is clear for $\phi_*$, and for $(\phi^*)^{-1}$ we look to the commuting diagram
\[
    \xymatrix@R=4mm@C=15mm{   
    			    F^n_{\chi}C^*(\mathcal{A},\mathcal{A}') \ar[d] \ar[r]^{ \phi^*|_{F^n_\chi} } & F^n_{\chi}C^*(\mathcal{A}',\mathcal{A}')\ar[d] \\
    			    C^*(\mathcal{A},\mathcal{A}') \ar[r]^{\phi^*} & C^*(\mathcal{A}',\mathcal{A}'),
		 }
\]
the restriction $\phi^*|_{F^n_\chi}$ being a quasi-isomorphism by argument in  proposition \ref{cocompletecomparison}. The same applies to $C(\psi)^{-1}$, thus the first claim is established.

For the second claim, one checks that ${\rm HH}(\phi):{\rm HH}^*(A',A')\xrightarrow{\cong}{\rm HH}^*(A,A)$ can be computed using ${\rm HH}^*(\Omega B\phi):{\rm HH}^*(\Omega BA',\Omega BA')\xrightarrow{\cong}{\rm HH}^*(\Omega BA,\Omega BA)$. Since $\Omega B \phi$ factors as  a sequence of homotopy retractions, it induces an isomorphism between the shearing filtrations on cohomology.
\end{proof}

We will prove later that these two filtrations are Koszul dual to each other, in a precise sense.

\subsubsection*{Mitchell-Hochschild cochain complexes}
\label{unreduced}

Hochschild cohomology is also defined for small dg and $A_\infty$-categories. If $A$ is augmented, the definition $C^*(A, A) = {\rm Hom}^\pi\!(BA, A)$ continues to make perfect sense. % This is a complex and we define ${ HH}^*(A, A)$ as its cohomology.
% It is more natural to consider non-augmented dg (or $A_\infty$) categories, 
However, the non-augmented situation will arise by force, since module categories are never augmented. So, we collect here some definitions for later use, restricting to the dg case for simplicity.

Fix a small dg category $A$, and consider its unreduced bar construction $B'\!A = T^{co} sA$, the cofree graded cocategory on the graded quiver $sA$. To avoid talking about curvature, we will not discuss the differential structure on $B'\!A$ (see the work of Positselski \cite{MR2830562} for this story). But we can still make sense of the differential on the $A$-bimodule $A\otimes^\pi\! B'\!A\otimes^\pi \!A$, which has the usual formula: it is the sum of $d_{A\otimes T^{co}sA\otimes A}$  and $\  g\otimes [f_1|...|f_n]\otimes h\  \mapsto $
\[
 (-1)^{|g|+1}gf_1\otimes [f_2|...|f_n]\otimes h 
 +
 \sum (-1)^{\epsilon_i+|f_i|}
 g\otimes [f_1|...|f_i f_{i+1}|...|f_n]\otimes h
 +
 (-1)^{\epsilon_n} g\otimes [f_1|...|f_{n-1}]\otimes f_nh
\]
where $\epsilon_i=(i-1)+|f_1|+...+|f_{i-1}|$. The $A$-bilinear map $A\otimes^\pi\! B'\!A\otimes^\pi \! A\xrightarrow{\sim} {}_AA_A$ is a quasi-isomorphism. If $A$ were augmented we could define $B'\!A=k_A\otimes_A(A\otimes^\pi\! B'\!A\otimes^\pi \!A)\otimes_A k_A$, in which case the inclusion $BA\hookrightarrow B'\!A$ is a quasi-isomorphism. % In general the situation is slightly more complicated.

As a graded vector space the unreduced Hochschild cochain complex $C^*(A,A)'$ is ${\rm Hom}(B'\!A,A)$. The differential may be defined through the identification $C^*(A,A)' \cong {\rm Hom}_{A-A}(A\otimes^\pi\! B'\!A\otimes^\pi \!A,A)$. Again, if $A$ were augmented then the restriction $C^*(A,A)'\to C^*(A,A)$ would be a quasi-isomorphism\footnote{If $A$ is an augmented dg algebra (or category), then $B'\!A$ is a dg coalgebra. It is tempting to identify $C^*(A,A)'={\rm Hom}(B'\!A,A)$ with $s^{-1}{\rm coder}(B'\!A,B'\!A)$ (they are isomorphic as graded vector spaces). However, the differentials do not coincide, and ${\rm coder}(B'\!A,B'\!A)$ cannot be used to compute Hochschild cohomology. In fact, ${\rm coder}(B'\!A,B'\!A)$ is quasi-isomorphic to ${\rm coder}(BA,BA)$, which has homology sitting strictly inside the shifted Hochschild cohomology of $A$.}.

As Keller remarks in \cite{KellerDIH} if $C\to D$ is a fully faithful embedding of dg categories then one has a restriction morphism $C^*(D,D)'\to C^*(C,C)'$. This fact is visible from the above formulas. This functoriality will be useful to us later.

\subsubsection*{The $A_\infty$-algebra structure on ${\rm Hom}^{\tau}(C, A)$}

\label{infty-twisted-convolution}
If $C$ is a cocomplete dg coalgebra and $A$ is an $A_\infty$-algebra the complex ${\rm Hom}(C,A)$ inherits an $A_\infty$-structure through the convolution operations $M_n$. 
%\[
%M_n(x_1,...,x_n) = m_n\circ(x_1\otimes ...\otimes x_n)\circ\Delta^{(n)},
%\]
%for $n\geq 2$, and the usual differential.
Given a twisting cochain $\tau:C\to A$ Lef{\`e}vre-Hasegawa \cite{LH} explains how to twist this structure, in a way generalising the construction of the Hochschild cochain complex. We will briefly indicate a more direct way to obtain the same twisted $A_\infty$-algebra, more in the spirit of Getlzer \cite{MR1261901} (see also \cite{MR1321701} and \cite{getzler-jones}%Gerstenhaber-Voronov and Getlzer-Jones.
). 
These authors only treat the universal twisting cochain, and also work in the simpler non-unital case.

The coalgebra $B{\rm Hom}(BA,\overline{A})=T^{co}{\rm Hom}(BA,s\overline{A})$ has a natural product, more honestly expressed before applying $B$ by saying that ${\rm Hom}(BA,\overline{A})$ is a brace algebra (see \cite{MR1321701}). %Maybe add Justin Young?
It is the unique coalgebra map $(B{\rm Hom}(BA,\overline{A}))^{\otimes 2}\to B{\rm Hom}(BA,\overline{A})$ whose projection onto ${\rm Hom}(BA,s\overline{A})$ is given by
\[
(x ,\  y_1\otimes ... \otimes y_n) \ \mapsto \  \sum_{i_0,...,i_n} x\circ (\pi_{ i_0}\otimes y_1\otimes  \pi_{ i_1}\otimes \cdots  \otimes \pi_{ i_{n-1}}\otimes y_n\otimes  \pi_{ i_n})\circ  \Delta^{(2n+1)},
\]
the map vanishing if the left input lives in weight more than one. Similarly, if $\tau:C\to s\overline{A}$ is a twisting cochain then $B{\rm Hom}(C,\overline{A})=T^{co}{\rm Hom}(C,s\overline{A})$ has a left module structure $B{\rm Hom}(BA,\overline{A})\otimes B{\rm Hom}(C,\overline{A})\to B{\rm Hom}(C,\overline{A})$ whose projection onto ${\rm Hom}(C,s\overline{A})$ is given by
\begin{equation}\label{braceaction}
(x ,\  y_1\otimes ... \otimes y_n) \ \mapsto \  \sum_{i_0,...,i_n} x\circ (\tau _{ i_0}\otimes y_1\otimes  \tau_{ i_1}\otimes \cdots  \otimes \tau_{i_{n-1}}\otimes y_n\otimes  \tau_{ i_n})\circ  \Delta^{(2n+1)},
\tag{$\ast$}
\end{equation}
lifted as before using the coalgebra structure on $B{\rm Hom}(C,\overline{A})$. Here $\tau_i$ is the map $\pi_i\phi_\tau: C\to (s\overline{A})^{\otimes i}$. Desuspending,  this is specified by sequence of maps
\[
{\rm Hom}(BA,\overline{A})\otimes {\rm Hom}(C,\overline{A})^{\otimes n}\to {\rm Hom}(C,\overline{A}),
\quad x \ ,\ y_1,  ... , y_n \ \mapsto \ x\{y_1,...,y_n\}
\]
of respective degree $-n$, making ${\rm Hom}(C,\overline{A})$ into a ``brace module" over ${\rm Hom}(BA,\overline{A})$.

The $A_\infty$-structure on $A$ corresponds to a degree $2$ element $m=\pi b_+$ in ${\rm Hom}(BA,\overline{A})$ such that  $\partial([m])+[m]^2=0$ in $B{\rm Hom}(BA,\overline{A})$. 
With this,
the $A_\infty$-structure on ${\rm Hom}(C,\overline{A})$ is specified by the differential $m\cdot(-)+\partial$ on $B{\rm Hom}(C,\overline{A})$. On the brace level, $M_n^\tau(x_1,...,x_n)= m\{x_1,...,x_n\}$ for $n$ at least $2$, and the twisted differential is  $\partial_\tau(x)= \partial(x)+m\{x\}$. This $A_\infty$-algebra is denoted ${\rm Hom}^\tau\!(C,\overline{A})$. Explicitly, if $x_1,...,x_n$ are in ${\rm Hom}^\tau\!(C,\overline{A})$ then $M^\tau_n(x_1,...,x_n) : C\to \overline{A}$ is
\[
M^\tau_n(x_1,...,x_n) =\sum_{i_0,...,i_n} (-1)^\epsilon m_{i_0+...+i_n+n}\circ( \tau^{\otimes i_0}\otimes x_1\otimes  \tau^{\otimes i_1}\otimes \cdots  \otimes \tau^{\otimes i_{n-1}}\otimes x_n\otimes  \tau^{\otimes i_n})\circ  \Delta^{(i_0+...+i_n+n)}.
\]
%Should I work out this sign? Or maybe delete this line?
Observe that if $A$ is simply a dg algebra then these are the usual convolution operations: only the differential of ${\rm Hom}^\tau\!(C,A)$ is twisted in this case.

Now we need to extend this structure to the whole of  ${\rm Hom}(C,A)= {\rm Hom}(C,\overline{A})\oplus C^*$. The complex ${\rm Hom}(C,\overline{A})$ is naturally a bimodule over the dg algebra $C^*$. We get an extension  ${\rm Hom}^\tau\!(C,A)=  {\rm Hom}^\tau\!(C,\overline{A})\rtimes C^*$ 
defined as follows. The component $C^*\to s{\rm Hom}(C,\overline{A})$ of the differential is $f\mapsto [\tau,f]$ making use of the bimodule structure (this we have already seen). The $A_\infty$-structure is extended to the middle by setting, for $f$ in $C^*$,
\[
M^\tau_2(f,x) = fx,\quad M^\tau_2(x,f)= xf, \quad \text{and}\quad M^\tau_i(...,f,...)=0 \quad \text{when } i>2.
\]
As for functoriality, suppose we have two $A_\infty$-algebras $A$ and $A'$ and a twisting cochain $\tau:C\to A$. 
A morphism $\phi:A\to A'$ is specified by a degree zero element $\sigma=\pi_1\Phi$ in ${\rm Hom}(BA,s\overline{A'})$. 
Then there is a coalgebra map\footnote{
The product on $B{\rm Hom}(BA,\overline{A})$ makes ${\rm Hom}(BA,\overline{A})$ a $B_\infty$-algebra (this goes back to  \cite{MR1261901}, \cite{getzler-jones}). Actually, $A_\infty$-algebras form a $B_\infty$- (or brace, or $E_2$-) category: the $B{\rm Hom}(B-,-)$ are the hom-spaces in a dg category whose objects are $A_\infty$-algebras. Here, we are describing a left action of this dg category on $B{\rm Hom}(C,-)$.} 
$B{\rm Hom}(BA,\overline{A'})\otimes B{\rm Hom}(C,\overline{A})\to B{\rm Hom}(C,\overline{A'})$ determined as usual by its  projection to ${\rm Hom}(C,s\overline{A'})$, where it is given by a formula identical to (\ref{braceaction}). 
The induced map of (nonunital) $A_\infty$-algebras ${\rm Hom}^\tau\!(C,\overline{A})\to {\rm Hom}^{\phi\tau}\!(C,\overline{A'})$ is then specified by the degree zero map
\[
B{\rm Hom}^\tau\!(C,\overline{A})\xrightarrow{\ \sigma\cdot (-)\ } B{\rm Hom}^{\phi\tau}\!(C,\overline{A'})\xrightarrow{\ \pi_1\ }{\rm Hom}^{\phi\tau}\!(C,s\overline{A'}).
\footnote{
The corresponding dg coalgebra map $B{\rm Hom}^\tau\!(C,\overline{A})\to B{\rm Hom}^{\phi\tau}\!(C,\overline{A'})$ is given by the action of the group-like element ${\rm exp}(\sigma)=\sum_{n\geq 0}\sigma^{\otimes n}\in \widehat{B}{\rm Hom}(BA,\overline{A'})$, the exponential being taken with respect to the shuffle product, using its divided power structure.
}
\]
This extends, by the identity on $C^*$, to a map ${\rm Hom}^\tau\!(C,A)\to {\rm Hom}^{\phi\tau}\!(C,A')$ of $A_\infty$-algebras. 
Functoriality in the dg coalgebra argument $C'\xrightarrow{\phi} C$ is simpler. The map $\phi^*:{\rm Hom}^{\tau}(C, A)\to {\rm Hom}^{\tau\phi}(C', A)$ of proposition \ref{cocompletecomparison} is already a strict map of $A_\infty$-algebras. Summing up: 

\begin{prop}\label{twistednaturality}
Let $\tau: C\to A$ be a twisting cochain from a dg coalgebra into an $A_\infty$-algebra. Then the complex ${\rm Hom}^\tau\!(C,A)$ canonically has the structure of an $A_\infty$-algebra.

If $\phi:C'\to C$ is a dg coalgebra map and $\psi :A\to A'$ is possibly non strict, then  $\phi^*: {\rm Hom}^{\tau}\!(C, A)\to {\rm Hom}^{\tau\phi}\!(C', A)$ is a strict map of $A_\infty$-algebras and the map $\psi_*:{\rm Hom}^\tau\!(C,A)\to {\rm Hom}^{\psi\tau}\!(C,A')$ from \ref{cocompletecomparison} extends canonically to a map of $A_\infty$-algebras.

Dually, if $\tau: C\to A$ is a twisting cochain from an $A_\infty$-coalgebra into a complete dg algebra, then ${\rm Hom}^\tau\!(C,A)$ is canonically an $A_\infty$-algebra. In this situation, if $\psi:A\to A'$ is a dg algebra map and $\phi:C'\to C$  is possibly non strict, then $\psi_*$ is a strict map of $A_\infty$-algebras and $\phi^*$ extends canonically to a map of $A_\infty$-algebras.

Finally, the isomorphism from lemma \ref{Dlemma} is one of $A_\infty$-algebras.
\end{prop}

%Attribute to LH?

\subsection{Koszul Duality}%%%%%%%%%%%%%%%%%%%%%%%%%%%%%%%%%%%%%%%%%%%%%%%%%%%%%%%%%%%%%%%%%%%%%%%%%%%%%%

\label{Koszuldualitysection}
%PLAN OF SECTION
%
% 1 Recall bar cobar adjunction.
%   it extends to A_infty co/algebras by definition
% 2 Relation to Ext and Tor, Koszul complex
%    Consequently we define BA=Kozul dual coalg and (BA)*=Koszul dual alg (even in A_infty setting).
% 3 Prove that A^!!=A in weakly connected case (note that finiteness issues occur when we try to get algebras on both sides).
% 4 The equivalence of perf A and thick_{A!} k.
% 5 Kozsul algebras. Basic facts and references for classical theory.

In this section we discuss some aspects of Koszul Duality for associative algebras. Of course, much more can be said, and we will have to completely elide many important aspects of this theory. Everything in this section is well known.

%To begin with we work only with dg algebras and coalgebras.

%1
Recall that the bar and cobar functors together form a Quillen equivalence of model categories
\[
    \xymatrix{
    			 B: {\rm \bf Alg} \ar@<2.5pt>[r]&  \ar@<2.5pt>[l]{\rm \bf coAlg}:\Omega.
		}
\]
%This has already been said. Cut down. (Sorry I rewrote it above you - V).
This fact sums up much of the force behind Koszul duality. %It can be rephrased by saying that the universal twisting cochains, out of $B$ and into $\Omega$, are acyclic.
Classical Koszul duality concerns what can be said about formality on either side of this adjunction.
%ADDREF
%Goes back to Quillen, Rational homotopy theory, Ann. of Math.
% Should refer to Hinich, who got model structures for dg algebras over a field over lots of operads.
% Also see Keller, A-infinity algebras, modules and functor categories

%The functors $B$ and $\Omega$ extend to $A_\infty$-algebras and coalgebras, essentially tautologically.

%2

If $A$ is an $A_\infty$-algebra, the {\bf Koszul dual dg coalgebra} to $A$ is any choice of cocomplete augmented dg coalgebra $A^\kos$ equipped with an acyclic twisting cochain $A^\kos\to A$. More generally, $A^\kos$ can be modeled by any $A_\infty$-coalgebra weakly equivalent to $BA$. %For concreteness, we usually take $A^\kos = BA$.

Let $\tau:C\to A$ be a twisting cochain. Note that the twisted part of the differential on $C\otimes^\tau \!\!A$ decreases the primitive filtration on $C$. By cocompleteness of $C$, this makes $C\otimes^\tau \!\!A$ a semi-free dg module, so in particular it is $h$-projective\footnote{A dg module $P$ is called $h$-projective if for any surjective quasi-isomorphism $\pi:M\to N$ of dg modules, every map $P\to N$ of dg modules factors through $\pi$. Equivalently, ${\rm Hom}_A(P,-)$ preserves quasi-isomorphisms. 
A dg module $M$ over $A$ is semi-free if it admits a filtration $0\subseteq F^1M\subseteq F^2M\subseteq ...$ by submodules such that $\bigcup F^iM=M$ and such that $F^iM/F^{i-1}M$ is free for all $i$. Actually, this just means $M$ is of the form $V\otimes^{\pi}A$ for a dg comodule $V$ over $BA$. Semi-free modules are $h$-projective. (See e.g. \cite{MR3291613} for details.)%Give reference for semi-free vs h-projective. 
}.
  Thus, if $\tau$ is acyclic, $C\otimes^\tau \!\!A\to k_A$ is semi-free resolution by theorem \ref{FT-TC}. This makes $C = (C \otimes^\tau A) \otimes_A k$ a model for $k\otimes^{{\rm L}}_A k$. The dg algebra $C^*$ acts canonically on the left of $C\otimes^\tau \!\!A$. This action gives a map $C^*\to {\rm Hom}_A(C\otimes^\tau \!\!A,C\otimes^\tau \!\!A)={\rm RHom}_A(k,k)$ of dg algebras which is a quasi-isomorphism by theorem \ref{FT-TC}. In particular, $(BA)^*$ functorially provides a model for ${\rm RHom}_A(k,k)$.

%Can (sometimes) take a minimal model $C\xrightarrow{\sim} BA$ and get a quasi-isomorphism $A \otimes^\tau C \otimes^\tau A \to A$ giving functional resolutions $M \otimes^\tau C \otimes^\tau A \to M$. %in some circumstances this is a functorial free resolution by finitely generated modules.

So, we define the Koszul dual algebra $A^!$ to be any choice of dg or $A_\infty$-algebra with a fixed choice of quasi-isomorphism $A^!\xrightarrow{\sim}(BA)^*$. This definition works properly only  when $A$ satisfies additional finiteness conditions. This is the cost of working on the algebra side.  %More fundamental is the Koszul dual dg coalgebra.
%Most of the time we can just take $A^!=(BA)^*$, but often we will need $A^!$ to be a minimal model of $(BA)^*$, making $A^!={\rm Ext}_A(k,k)$ as graded associative algebras.
For emphasis, let us record the following version of the above Quillen equivalence. %The cost of working on the algebra side is that finiteness conditions are necessitated by the repeated dualization.

\begin{thm}\label{doubledual} If $A$ is a strongly connected $A_\infty$-algebra then $A^{!!}$ is canonically quasi-isomorphic to $A$. Thus, $A$ can be recovered from $A^!$ up to quasi-isomorphism.
\end{thm}

Both dualizing and $B$ preserve quasi-isomorphisms, so whichever model for $A^!$ we use there is a quasi-isomorphism $A^{!!}\simeq (B(BA)^*)^*$. Thus, taking a minimal model for $A$, this follows from lemma \ref{connectedness} and \ref{LHtheorem} (\ref{A-infinity-qis}).

Koszul duality yields derived equivalences (going back to \cite{BGS}%ALSO PRIDDY?
). The following is perhaps the most basic such equivalence. It follows readily from the above discussion and the techniques of Keller \cite{MR1258406}.

%Some notation for the following theorem, and discussion.
%We write $D^{\rm dg}(A)$ for the pre-triangulated dg category of $h$-projective right dg modules over $A$. Its homotopy category is the derived category $D(A)$. If $M$ is an object (or set of objects) in $D(A)$ then ${\rm thick}_A M$ is the smallest triangulated subcategory of $D(A)$ containing $M$ which is closed under retracts. Its pretriangulated dg enhancement, as a full dg subcategory of $D^{\rm dg}(A)$, is denoted ${\rm thick}_A^{\rm dg} \!M$. We also write ${\rm perf}A={\rm thick}_{A}A$ and  ${\rm perf}^{\rm dg} \!A={\rm thick}_A^{\rm dg} \!A$.

\begin{thm}\label{small-equiv}
Let $A$ be a dg algebra and $A^!$ a dg algebra model for the Koszul dual. We have natural equivalences of triangulated categories
\[
{\rm thick}_{D(A)}k\simeq D_{\rm perf} (A^!)
\quad\quad
\text{and when $A$ is strongly connected}
\quad\quad
 D_{\rm perf} (A) \simeq {\rm thick}_{D(A^!)}k.
\]
\end{thm}

%\begin{proof}
%We need an $h$-projective replacement for $k$, so we take $\mathbb{K}=BA\otimes^\pi \!A \xrightarrow{\sim}k$. Following Keller \cite{MR1258406}%Deriving DG categories
%we work with the dg enhancement ${\rm thick}_A^{\rm dg}k$. Since $k$ is generator for ${\rm thick}_Ak$ the dg functor
%\[
%{\rm RHom}_A(k,-)={\rm Hom}_A(\mathbb{K},-): {\rm thick}_A^{\rm dg}k\to {\rm perf}^{\rm dg} \,{\rm Hom}_A(\mathbb{K},\mathbb{K})
%\]
%is a quasi-equivalence, by the five-lemma. Since $A^!$ is quasi-isomorphic to ${\rm Hom}_A(\mathbb{K},\mathbb{K})$, we have a quasi-equivalence ${\rm perf}^{\rm dg}\, {\rm Hom}_A(\mathbb{K},\mathbb{K})\xrightarrow{\sim}{\rm perf}^{\rm dg} A^!$. The second statement then follows using theorem \ref{doubledual}.
%\end{proof}

With appropriate definitions, this extends to the $A_\infty$ setting, worked out by Lu, Palmieri, Wu and Zhang in \cite{LPWZ}. Lef{\`e}vre-Hasegawa establishes a much larger and more generally applicable derived equivalence, relating $A$ and $A^\kos$, in \cite{LH}.

Another consequence of \ref{doubledual}, using Theorem 4.6 of Keller's \cite{KellerDIH}, is that for a strongly connected dg algebra $A$, there is an isomorphism $C^*(A,A)\simeq C^*(A^{!},A^{!})$ in the homotopy category of $B_\infty$-algebras (see \emph{loc. cit.} for this terminology). In order to check that our theorem \ref{maintheorem} holds, we will give a more concrete isomorphism in section \ref{HHKoszul}, which is more in the spirit of \cite{FMT}. This way, one can  slightly relax the needed finiteness conditions.

Starting with a graded algebra $A$, it would be desirable to be able to replace $A^!={\rm RHom}_A(k,k)$ with the simpler graded algebra ${\rm Ext}_A(k,k)$ in theorem \ref{doubledual} and theorem \ref{small-equiv}. With this in mind, we define $A$ to be {\bf Koszul} if $A$ is strongly connected and $A^!$ is formal. Let us explain how this relates to the more classical picture of Koszul duality, which concerns small resolutions.

%From one perspective the Koszul property is about finding small resolutions for modules. But lets start with a more homotopical definition: a graded algebra $A$ is called Koszul if $A^!$ is formal. In other words, the Koszul dual dg algebra $(BA)^*$ should be quasi-isomorphic to the associative algebra ${\rm Ext}_A(k,k)$. Or more fundamentally, $BA$ should be quasi-isomorphic (weak eq?) to ${\rm Tor}^A(k,k)$.

The famous bar resolution $M\otimes^\pi BA\otimes^\pi A$ functorially resolves any module, but it grows extremely fast. In the strongly connected situation we can replace $BA$ with a minimal model, that is, an acyclic twisting cochain $\tau:A^\kos \to A$ where $A^\kos={\rm Tor}^A(k,k)$ with higher structure. We obtain  a much smaller resolution $M\otimes^\tau A^\kos \otimes^\tau A$, still functorial in $M$. The problem now is understanding the higher operations on $A^\kos$. % they are difficult to write down in general, and their presence can make the differential on $M\otimes^\tau A^\kos \otimes^\tau A$ extremely complicated.
The best possible situation is that there are no higher operations: if $A$ is Koszul then we may use the graded coalgebra ${\rm Tor}^A(k,k)$. %The only objection remaining is we need to find a quasi-isomorphism ${\rm Tor}^A(k,k)\xrightarrow{\sim} BA$, or equivalently, an acyclic twisting cochain $\tau:{\rm Tor}^A(k,k)\to A$.
Miraculously, this assumption that $A$ is Koszul leads (with the data of a presentation of $A$) to a simple, canonical quasi-isomorphism ${\rm Tor}^A(k,k)\xrightarrow{\sim} BA$.

So, let $A$ be an augmented graded algebra and suppose we have a graded vector space $V$ with a map $V\to A$ such that $T V\to A$ is surjective, and assume that $V$ is minimal, meaning that the kernel of $T^a V\to A$ lies in $T^a_{\geq 2}V$. %Certainly ${\rm Tor}^A(k,k)$ contains $V$, because the minimal resolution of $k_A$ begins $\cdots \to V\otimes A\to A$. %Very roughly, $A$ is Koszul if and only if ${\rm Tor}^A(k,k)\twoheadrightarrow V\hookrightarrow A$ is an acyclic twisting cochain.
Choose also a minimal space of relations $R\subseteq T^a_{\geq 2}V$, so that $A=T^{a}V/\langle R\rangle$.\footnote{ Minimality means here that $\langle V\cdot R+ R\cdot V\rangle \cap R = 0$. We can interpret the first step $T(V\oplus sR)$ in the Tate construction of $A$ (see section \ref{examples}) as the first step in the construction of the minimal model $\Omega A^\kos =\Omega {\rm Tor}^A(k,k)$. If $A$ is Koszul, we can construct $A^\kos$ with no higher coproducts. Examining the Tate construction, this means we must be able to choose $R$ quadratic.
} 
A semi-free dg module $M=(U\otimes A, d)$ is called {\bf linear} if $d(U)\subseteq U\otimes V$. The minimal semi-free resolution $k_A$ begins
\[
\cdots\to s^2R\otimes A\to sV\otimes A\to A\to 0.
\]
So ${\rm Tor}^A(k,k)$ will begin $(k\oplus sV\oplus s^2 R\oplus ...)$. Note that for this resolution to be linear $R$ must be quadratic.

%CUT OUT OR CUT DOWN: In the first two steps of building the Tate-model $\Omega A^\kos \to A$, we construct the $A_\infty$-coalgebra $C_2=s V\oplus s^2 R$ and define the differential on $\Omega C_2=T (V\oplus sR)$ to be zero on  $V$ and $d(sr)=r\in T^{a}V$ on $sR$. This simply means that in the $A_\infty$-coalgebra $C_2$ we have $\Delta_n(r)=r_n\in T_nV\subseteq C_2^n$, where $r_n$ is the tensor-length $n$ part of  $r$. Since $C_2$ will turn out to be a be a strict sub-colagebra of $A^\kos$, if $A$ is Koszul we can make these choices so that $C_2$ has no higher coproducts. This means we can chose $R$ to lie in $T_2 V$. In other words, being Koszul forces $A$ to have a quadratic presentation. 

 Aside from the algebra $A=T^{a}V/\langle R\rangle$, there is an equally natural way to define a coalgebra from  data of $V$ and $R\subseteq V^{\otimes 2}$: we can take the complete, augmented coalgebra $C(V,R)$ which is cofree on $V$ subject to the constraint that $C(V,R)\to C(V,R)\otimes C(V,R)\to V\otimes V$ factors through $R\hookrightarrow V\otimes V$. Concretely, $C(V,R)$ is the sub-coalgebra of $T^{co}V$ whose weight $n$ part is $C_n =  \bigcap_{p+2+q=n} V^{\otimes p}\otimes R \otimes V^{\otimes q}$. Now taking $C=C(sV,s^2 R)$, the twisting cochain $\tau:C\twoheadrightarrow sV\hookrightarrow sA$ factors through the universal one via the inclusion $C\hookrightarrow T^{co}sV\hookrightarrow BA$ of dg coalgebras.

%The point is that the quadratic algebra $A$ is Koszul if and only if $C\hookrightarrow BA$ is a quasi-isomorphism (weak-equivalence?). Hence $C$ is a minimal model for the coalgebra Koszul dual to $A$, and we can call $C\otimes^{\tau} A$ is the Koszul complex, it is just an explicit model of $A\otimes^\tau A^{\rm upside down !}$. Thus we obtain a linear resolution of $k$.

%Equivalently ${\rm Ext}^*_A(k,k)$ is generated in weight $1$ (might not even be weight graded, people usually take bigraded dual instead and get something smaller than full Ext).
%Theorem(?): the minimal model ${\rm Ext}^*_A(k,k)$ (equipped with its higher structure) for $A^!$ is always generated in weight $1$ as an $A_\infty$-algebra. Due to May? Need reference.

\begin{thm}\label{classicalKoszul} Let $A = T^{a}V/\langle R\rangle$ be as above. Assume $A$ is strongly connected. The following are equivalent:
\begin{enumerate}[(i)]
\item the dg algebra $A^!={\rm RHom}_A(k,k)$ is formal,\label{item-formal}
\item the augmentation module $k_A$ admits a linear resolution, \label{item-clkos}
%\item ${\rm Ext}_A(k,k)$ is generated in cohomological degree $1$ as an algebra,\label{item-priddy}% Delete this one?
\item $R$ can be chosen quadratic and $C\hookrightarrow BA$ is a quasi-isomorphism, so $C\cong {\rm Tor}^A(k,k)$.\label{item-qis}
\end{enumerate}
\end{thm}
%Needs fixing (and shortening). Unless there is a short way to talk about $C$, it's probably not worth going in to here.

The equivalence of  (\ref{item-qis}) with (\ref{item-clkos}) is in \cite{BGS},  the obtained resolution $C\otimes^\tau\! A\xrightarrow{\sim} k$ is called the Koszul complex by Beilinson, Ginzburg and Soergel. The connection to (\ref{item-formal}) is in Keller's \cite{MR2067371}, but see also the work of Gugenheim and May \cite{MR0394720}. Many of these ideas go back to Priddy \cite{MR0265437}.

It follows that we can compute Hochschild cohomology using ${\rm Hom}^\tau\!(C,A)$,  the dg algebra with convolution product and differential $[\tau,-]$. The weight and shearing filtrations split into gradings, so they don't need to be computed on semi-(co)free models. The weight filtration comes from the grading ${\rm Hom}^\tau\!({\rm gr}^nC_{[*]},A)$ and the shearing filtration comes from ${\rm Hom}^\tau\!(C,{\rm gr}^nA^{[*]})$, the differential is homogeneous of degree $1$ for both these gradings. 

\subsubsection*{Koszul Duality for Commutative and Lie Algebras}\label{LieCom-section}%%%%%%%%%%%%%%%%%%%%%%%%%%%%%%%%%%%%%%%%%%%%%%%%%%%%%%%%%%%%%%%%%%%%%%%%%%%%%%

In the this section we will briefly indicate how the bar and cobar construction generalise to the Lie and commutative settings, allowing us to define $L_\infty$- and $C_\infty$-algebras. This section  should serve as context for the discussion of commutativity in section \ref{commutativity}, as well as being a source of examples later on. We will also need to record some facts about universal envelopes. The definitions are best motivated through the Koszul duality between the Lie and commutative operads, for which the reference \cite{LV} is highly recommended. $C_\infty$-algebras were introduced by Kadeishvili  in \cite{MR1029003}.   %Kadeishvili The structure of the A()-algebra, and the Hochschild and Harrison cohomologies
%...even before the theory of Koszul duality for operads had been worked out. 
The functors $\mathcal{C}$ and $\mathcal{L}$ go back to Chevalley-Eilenberg and Quillen respectively, and are explained in \cite{FHTbook}. %rational homotopy theory book %ADDREF
It is important for this section that we assume our base field $k$ has characteristic zero.% Since the main results of this paper do not rely heavily on this section, it can safely be skipped.

%Aside from associative and commutative coalgebras we will use in this paper the notion of Lie coalgebra. Let us remind the reader what these are now. A dg Lie coalgebra is a complex $L$ with chain map $\Delta: L\to L\otimes L$ which is anti-commutative in the sense that $T\Delta=-\Delta$, and which satisfies the dual Jacobi identity
%\[
%(1\otimes \Delta)\Delta =(\Delta\otimes 1)\Delta +(T\otimes 1)(1\otimes \Delta)\Delta.
%\]
%Here we are using $T$ to denote the usual twist map $U\otimes V\to V\otimes U$ which follows the Koszul sign rule. We call $\Delta$ the cobracket of $L$. Further, $L$ is called complete if for any particular element $l$ of $L$, any sufficiently large iteration $L\to L^{\otimes n}$ of the cobracket (with whatever choice of association) vanishes on $l$. The notion of coderivation is the same in this context as before: it is a graded map $p:L\to L$ of some degree for which $\Delta p = (1\otimes p+p\otimes 1)\Delta=0$.
% Am I allowed to delete this?

%The free Lie algebra on a complex $V$ will be denoted ${\rm \bf Lie} V$, while the cofree complete Lie coalgebra on $V$ will be denoted ${\rm \bf Lie}^{co} V$. We use the usual notation ${\rm Sym} V$ and ${\rm Sym}^{co} V$ for the free augmented commutative algebra and cofree complete augmented commutative coalgebra on $V$, respectively.

%Cofreeness gives ${\rm \bf Lie}^{co} V$ a ``cobraket length" grading, with $V$ in degree $1$.  This grading As with the bar construction we refer to this grading as weight, and denote the weight $n$ part by ${\rm \bf Lie}^{co}_n V$. OR 

The famous Chevalley-Eilenberg complex $\mathcal{C}^{co}L$ of a Lie algebra is a dg commutative coalgebra, conversely Quillen's functor from rational homotopy theory takes a commutative coalgebra $A$ to a dg Lie algebra $\mathcal{L}A$:
\[
    \xymatrix{
    			 \mathcal{L}: {\rm \bf coCom} \ar@<2.5pt>[r]&  \ar@<2.5pt>[l]{\rm \bf Lie}:\mathcal{C}^{co}.
		}
\]
%Quillen, Rational homotopy theory, Ann. of Math.
%
Here ${\rm \bf coCom}$ is the category of cocomplete coaugmented dg commutative coalgebras, and ${\rm \bf Lie}$ is the category of dg Lie algebras. %LV or one of the operad papers below? Or just Ho-theory book.
With suitable model structures (and, possibly, additional finiteness conditions) these form a Quillen equivalence of model categories \cite{FHTbook}. This fundamental fact in rational homotopy theory is an aspect of the Koszul duality between commutative algebras and Lie algebras. Equally, we have a dual adjunction $\mathcal{L}^{co}: {\rm \bf Com} \leftrightarrows {\rm \bf coLie}:\mathcal{C}$.\footnote{
A definition for Lie coalgebras can be found in \cite{MR814187}.%Stasheff and Schlessinger
}
%...between the category of augmented dg commutative algebras and the category of complete dg Lie algebras.

%First we explain how to construct the functor $\mathcal{L}^{co}$. For details we refer the reader to \cite{FHTbook}.%rational homotopy theory book

As for associative algebras there is a natural notion of a twisting cochain $\tau:C\to L$ from a commutative coalgebra to a Lie algebra (or $\tau:L\to A$ from a Lie coalgebra to a commutative algebra) and one may interpret Lie algebra cohomology (or Harrison cohomology) from this persective. See \cite{LV} for details.

As before, this twisting cochain functor is representable in both arguments. $\mathcal{C},\mathcal{L}$ and their duals  $\mathcal{C}^{co},\mathcal{L}^{co}$ are the representing functors. In short, the bracket on a Lie algebra $L$ is encoded by a weight $1$ codifferential on $\mathcal{C}^{co}L= {\rm Sym}^{co}({s} L)$, while the product on a commutative algebra $A$ is encoded by a weight $1$ codifferential  $\mathcal{L}^{co}\!A={\rm \bf Lie}^{co}({s}\overline{A})$. 
Now $L_\infty$- and $C_{\infty}$-algebras are defined through Koszul duality, just as $A_{\infty}$-algebras were:

\begin{defn}
An $L_{\infty}$-algebra structure on $L$ is a codifferential $b$ on ${\rm Sym}^{co}({s} L)$ vanishing on $k \hookrightarrow Sym^{co}({s} L)$; the Koszul dual dg commutative coalgebra $\mathcal{C}^{co}L$ is  by definition ${\rm Sym}^{co}({s} L)$ equipped with $b$. A $C_{\infty}$-algebra structure on an augmented complex $A$ is a codifferential $b$ on ${\rm \bf Lie}^{co}({s}\overline{A})$; the Koszul dual dg Lie coalgebra $\mathcal{L}^{co}\!A$ is ${\rm \bf Lie}^{co}({s}\overline{A})$ equipped with $b$.
\end{defn}

% By codifferential we mean a square zero coderivation of total degree $1$, augmented means here that $b$ vanishes on the augmentation $k\to \mathcal{C}^{co}L$. 

% A commutative dg algebra $A$ is readily recovered from $\mathcal{L}^{co}\!A$, so commutative dg algebras  are the same thing as $C_{\infty}$-algebras whose differential $b$ decreases weight by no more than $1$.
%Again, through $\mathcal{C}^{co}$ one readily identifies dg Lie algebras with  $L_{\infty}$-algebras for which $b$ decreases weight by no more than $1$. We will also refer to $\mathcal{L}^{co}$ and $\mathcal{C}^{co}$ as bar constructions.

A morphism $A\to A'$ of $C_\infty$-algebras is by definition a morphism of dg Lie coalgebras $\mathcal{L}^{co}\!A\to \mathcal{L}^{co}\!A'$, a strict such morphism being one which preserves weight. A morphism $A\to A'$ is a quasi-isomorphism if its weight $1$ part is a quasi-isomorphism of chain complexes $\overline{A}\to \overline{A'}$. Equally, morphisms of $L_\infty$-algebras are defined through the commutative bar construction $\mathcal{C}^{co}$.

One can just as well use the dual cobar constructions $\mathcal{C}$ and $\mathcal{L}$ to make the analogous coalgebra definitions. In brief: a $C_{\infty}$-coalgebra $C$ is  the data of a differential on $\mathcal{L}C={\rm \bf Lie}({s^{-1}}\overline{C})$, while an $L_{\infty}$-coalgebra $L$ is an augmented differential on $\mathcal{C}L={\rm Sym}({s^{-1}} L )$.% Dually, $L_\infty$- and $C_\infty$-coalgebras are defined through their cobar constructions $\mathcal{C}$ and $\mathcal{L}$.

%The Koszul duals defined above are really only one possible choice. In this paper any $L_{\infty}$-coalgebra $L$ equipped with a quasi-isomorphism (weak eq?) $L\to \mathcal{L}^{co}\! A$ will be called a (or `the') Koszul dual $L_{\infty}$-coalgebra to $A$. A Koszul dual $\mathcal{L}_\infty$-algebra is one quasi-isomorphic (weakly eq?) to $(\mathcal{L}^{co}\! A)^*$. An important such choice is the minimal model of $(\mathcal{L}^{co}\! A)^*$, which is unique up to isomorphism.

%\begin{defn}
%A $C_{\infty}$-coalgebra $C$ is a differential $b$ on ${\rm \bf Lie}({s^{-1}}\overline{C})$. The Koszul dual dg Lie algebra $\mathcal{L}C$ is by definition ${\rm \bf Lie}({s^{-1}}\overline{C})$ equipped with $b$.
%\end{defn}

%\begin{defn}
%An $L_{\infty}$-coalgebra $L$ is an augmented differential $b$ on ${\rm Sym}({s^{-1}} L )$. The Koszul dual dg commutative algebra $\mathcal{C}L$ is by definition ${\rm Sym}({s^{-1}} L )$ equipped with $b$.
%\end{defn}

%Let $L$ be an $L_\infty$-algebra.
As in the associative case, one can extract from the components $\overline{b}_n:\mathcal{C}^{co}_nL\to L$ a sequence of anti-linear maps $l_n:L^{\otimes n}\to L$, using that $k$ has characteristic zero. The condition that $b^2=0$ becomes a sequence of Stasheff-like quadratic identities in the $l_n$ generalising the Jacobi identity. These can be found in \cite{MR1327129}. % Strongly homotopy Lie algebras by Tom Lada and Martin Markl
%In this way an $L_\infty$-algebra can be thought of as a Lie algebra with a sequence of higher brackets $l_n$, such that $l_2$ satisfies the Jacobi identity up to a sequence of `higher homotopies' $l_3,l_4,...$. %BLEH
Of course there is a similar story for $C_\infty$-algebras.% most easily interpreted through the alternative description of $C_\infty$-algebras below.

These objects (in characteristic zero) are just as well behaved as their associative counterparts. A more precise statement of the following well known theorem can be found in \cite{LV}, %  LV section 10.3 and 10.4
or \cite{MR2640648}, 
% Benoit Fresse, Props in model categories and homotopy invariance of structures
\cite{MR2016697}.
% C. Berger and I. Moerdijk, Axiomatic homotopy theory for operads
% Also possibly HTT: PRE-LIE DEFORMATION THEORY VLADIMIR DOTSENKO, SERGEY SHADRIN, AND BRUNO VALLETTE
%
\begin{thm}\label{Lie-com-homotopy}
Both  $C_\infty$- and $L_\infty$-algebras satisfy the inverse function theorem: quasi-isomorphisms have inverses up to taking homology. A $C_\infty$ or $L_\infty$ structure can be transferred along a homotopy retraction $i:V\hookrightarrow A$, from $A$ to $V$, with $i$ extending to a quasi-isomorphism of algebras. In particular, $C_\infty$- and $L_\infty$-algebras have minimal models, unique up to isomorphism.
\end{thm}

Any associative algebra $A$ has an underlying Lie algebra $A^{\rm Lie}$, and any Lie algebra $L$ has a universal envelope $\mathcal{U}L$. This adjunction extends to algebras (or coalgebras) with higher operations, and allows one to give another definition of $C_{\infty}$-algebra.

Firstly, there is a natural embedding ${\rm Sym}^{co}\hookrightarrow T^{co}$ constructed through obvious universal properties. More concretely, ${\rm Sym}^{co}V$ may always be presented as the sub-coalgebra of symmetric tensors in $T^{co}V$. For any $A_{\infty}$-algebra $A$ it is true that the bar differential $b$ restricts to ${\rm Sym}^{co}(s\overline{A})$ inside $BA$, and the restriction of $b$ is always an augmented codifferential on ${\rm Sym}^{co}(s\overline{A})$. Thus, we define $A^{\rm Lie}$ to be the $L_{\infty}$-algebra whose bar construction is $\mathcal{C}^{co}A^{\rm Lie}={\rm Sym}^{co}(s\overline{A})$ with this codifferential. In particular the underlying complex of $A^{\rm Lie}$ is  $\overline{A}$. In the case of a dg algebra this coincides with the usual associated Lie algebra. 
Since any morphism of $A_\infty$-algebras $BA\to BA'$ automatically restricts to symmetric elements $\mathcal{C}^{co}A^{\rm Lie}\to \mathcal{C}^{co}A'^{\rm Lie}$, this assignment is functorial in all (strict or non-strict) $A_\infty$-morphisms. 
Similarly, there is a well known presentation $T^{a}\twoheadrightarrow {\rm Sym}$, and for any $A_{\infty}$-coalgebra $C$ the cobar differential on $\Omega C$ descends to ${\rm Sym}\, s^{-1} \overline{C}$, allowing us define the associated $L_{\infty}$-coalgebra. %We define the associated $L_{\infty}$-coalgebra $C^{\rm Lie}$ by declaring $\mathcal{C}C^{\rm Lie}$ to be  ${\rm Sym}\,s^{-1} \overline{C}$ with the differential induced from $\Omega C$.

The universal twisting cochains for these $L_{\infty}$-structures factor through their associative counterparts:
\[
\mathcal{C}^{co}A^{\rm Lie} \to BA\to A^{\rm Lie}\quad\quad \text{and} \quad\quad C^{\rm Lie}\to \Omega C \to \mathcal{C}C^{\rm Lie}.
\]
If $A$ is a $C_{\infty}$-algebra there is a canonical isomorphism $\mathcal{U} \mathcal{L}^{co}\! A \cong T^{co}{s}\overline{A}$. The codifferential thus obtained on $T^{co}{s}\overline{A}$ equips $A$ with the structure of an $A_{\infty}$-algebra, % Note that this codifferential is automatically a derivation for the shuffle product (as well as a coderivation in the usual sense). 
extending the usual forgetful functor from commutative algebras to associative algebras. 
%If $A$ were a commutative dg algebra then we could forget that $A$ was commutative and obtain an associative dg algebra. The bar construction $BA$ of this algebra is the same as the one obtained in the previous paragraph. This is simply because the map $\overline{b}:BA\to A$ factors through $\mathcal{L}^{co}\! A$, the identification $\mathcal{L}^{co}_2\! A= \bigwedge^2 { s}\overline{A}$ being consistent with the obvious map $B_2A\to \bigwedge^2 { s}\overline{A}$.
Conversely, if $A$ is an $A_\infty$-algebra  whose bar differential $b$ is a derivation for the shuffle product, %on $BA=T^{co}s\overline{A}$, 
then $b$ descends to ${\rm \bf Lie}^{co}({s}\overline{A})=\overline{BA}/(\overline{BA}\sh\overline{BA})$, equipping $A$ with the structure of a $C_\infty$-algebra. %In this way,  a $C_\infty$-algebra can be thought of as an $A_\infty$-algebra such that the operations $m_n$  vanish on all nontrivial shuffles $[a_1|...|a_p]\sh[a_{p+1}|...|a_{n}]$ with $ 0<p<n$. %As Kadeishvilli does

According to Lada and Markl \cite{MR1327129} a universal universal envelope exists for $L_\infty$-algebras: there is a left adjoint to the functor $(-)^{\rm Lie}$ as long as one restricts to strict morphisms. %Strongly homotopy Lie algebras by Tom Lada and Martin Markl
This also follows from the general principles developed in \cite{LV}. 
For us the more concrete approach of Baranovsky is useful. Baranovsky constructs an explicit model for the universal envelope of an $L_\infty$-algebra in \cite{Baranovsky}, and establishes the following properties.

\begin{thm}[Baranovsky]
There exists a functor $\mathcal{U}$ from $L_\infty$-algebras and strict morphisms to $A_\infty$-algebras and strict morphisms, such that
\begin{itemize}
\item $\mathcal{U}$ extends the universal envelope for dg Lie algebras,
\item $\mathcal{U}$ preserves quasi-isomorphisms,
\item there is a canonical, strict morphism of $L_\infty$-algebras $L\to (\mathcal{U}L)^{\rm Lie}$,
\item Poincar\'{e}-Birkhoff-Witt holds in the sense that $L\to (\mathcal{U}L)^{\rm Lie}$ extends to an isomorphism $ {\rm Sym}^{co}L\cong \mathcal{U}L$.
\end{itemize}
Further, if $\phi:L\to L'$ is possibly non-strict, there is a naturally defined morphism $\mathcal{U}(\phi): \mathcal{U}L\to\mathcal{U}L'$ of $A_\infty$-algebras, which is a quasi-isomorphism if $\phi$ is.
\end{thm} 

While Baranovsky's construction does not extend functorially to non-strict morphisms, he does prove that the assignment $\phi\mapsto \mathcal{U}(\phi)$ is functorial up to (canonical) homotopy, in a precise sense. Baranovsky's construction on strict morphisms is not the adjunction mentioned above. More likely his construction is some variety of weak 2-adjunction, but this is another story. %should probably delete this
What we need from this theorem is the following lemma, which follows readily.

\begin{lem}\label{envelope}
If $C$ is $C_{\infty}$-coalgebra and $L$ is a minimal model for $\mathcal{L}C$ then $\mathcal{U} L$ is a minimal model for $\Omega C$, and $L$ is a strict subalgebra of $(\mathcal{U} L)^{\rm Lie}$.
\end{lem}

Finally, let's introduce some terminology that will be useful in section \ref{stringtop}.

\begin{defn} An $L_{\infty}$-algebra or coalgebra $L$ will be called completely abelian if all of its operations (not including the differential) vanish. Further, $L$ is called quasi-abelian if it admits a completely abelian model.
\end{defn}

\section{The characteristic morphism and Hochschild cohomology} % feel free to change names here.

\subsection{The characteristic action of Hochschild cohomology and the shearing and projection morphisms }\label{projection-shearing-section}%%%%%%%%%%%%%%%%%%%%%%%%%%%%%%%%%%%%%%%%%%%%%%%%%%%%%%%%%%%%%%%%%%%%%%%%%%%%%%

% What versions of the theorem of this section have already been proven by others?
% Decide how general the last theorem is if possible

On the Hochschild cohomology of any dg or $A_{\infty}$-algebra $A$ we have two well known maps
\[
\Pi:{\rm HH}^*(A,A)\longrightarrow {\rm H}(A)\quad \text{ and }\quad
\chi:{\rm HH}^*(A,A)\longrightarrow {\rm H}(A^!) ={\rm Ext}_A(k,k)
\]
which we call the {\bf projection} and {\bf shearing} morphisms respectively. On the chain level the definitions are very simple: $\Pi:{\rm Hom}^\pi(BA,A)\to{\rm Hom }(k,A)=A$ is defined by composing with the coaugmentation $k\to BA$ while  $\chi:{\rm Hom}^\pi(BA,A)\to{\rm Hom }(BA,k)=A^!$ is defined by composing with the augmentation $A\to k$. Because $k\to BA$ is a coalgebra map and $A\to k$ is an $A_{\infty}$-algebra map, proposition \ref{twistednaturality} says that $\Pi$ and $\chi$ are morphisms of dg or $A_\infty$-algebras. In particular they descend to homology.

Note that the projection and shearing maps are, respectively, the quotient by the positive part of the weight and shearing filtrations.

For ordinary graded algebras, the image of the projection map ${\rm HH}^*(A,A)\to A$ is exactly the graded centre $Z_{\rm gr}A$. The next section will be devoted to discussing the meaning of the image in general.

Shamir \cite{Sha} 
discusses the shearing map in detail. In terms of coderivations the shearing map is the obvious quotient
\[
C^*(A, A) = s^{-1} {\rm coder}(BA,BA) \rtimes (BA)^*\to (BA)^*,
\]
and there is a similar interpretation of the projection map (see the next section). 
 
As a special case of naturality of the twisted hom spaces ${\rm Hom}^\tau(C,A)$, we get naturality of the projection and shearing maps for quasi-isomorphisms of $A_\infty$-algebras. That is, if $\phi:A\to A'$ is a quasi-isomorphism of $A_\infty$-algebras, then we have a diagram
\[
    \xymatrix{
                    (BA)^*     &   C^*(A,A) \ar[r]^-{\Pi\ \ }  \ar[l]_-\chi  \ar@{}[d]|-(.5){\bullet}="a"^{C (\phi)} \ar;"a" &    A   \ar[d]_{\wr}^{\phi} \\
                    (BA')^*  \ar[u]^{\Phi^*}_\wr            &   C^*(A',A') \ar[r]^-{\Pi\ \ }  \ar;"a" \ar[l]_-\chi   &     A' 
		    }
\]
which commutes in the homotopy category of $A_\infty$-algebras.

The same definitions work for an augmented $A_\infty$-category $A$. That is, the projection and shearing morphisms come from the coaugmentation $k_A\to BA$ and the augmentation $A\to k_A$ respectively. Further, if $A$ is a not necessarily augmented dg algebra (or category) there is a similarly defined projection $\Pi:C^*(A,A)'\to A$, which, if $A$ happens to be augmented, factors through the above one via the quasi-isomorphism $C^*(A,A)'\to C^*(A,A)$. Of course there is no shearing map in this situation.

Recall that any a graded $k$-category $C$ has a graded centre $Z_{\rm gr}C$, defined by setting $Z^n_{\rm gr}C$ to be the collection of degree $n$ natural transformations $1_C\to 1_C$, in the graded sense. If the graded category has an internal suspension $s$, so that $C$ gets its grading from ${\rm Hom}_C^n(X,Y)={\rm Hom}_C^0(X,s^nY)$, %and $s$ extends from the ungraded category compatibly with these identifications
then $Z^n_{\rm gr}C$ may alternatively be described as the space of natural transformations $\xi:1_C\to s^n$ (on the underlying ungraded category) which additionally satisfy $s\xi=\xi s$.\footnote{
To be more precise, this means $1_s \circ \xi=  \xi \circ 1_s $, where $1_s:s\to s$ is the identity transformation and $\circ$ denotes the horizontal composition of natural transformations. There is no sign here because both $\xi$ and $1_s$ have degree zero. Signs appear when $\xi$ and $s$ are extended to the graded category.
} In particular, every triangulated category is naturally a graded category and has a graded centre. Krause and Ye discuss graded centres of triangulated categories in \cite{MR2794666}.

It will be simpler to define the characteristic map ${\rm HH}^*(A,A)\to Z_{\rm gr} (DA)$ in the case that $A$ is a dg algebra. For the general case we would need to develop the theory of $A_{\infty}$-algebras and modules more than is necessary in this paper. 
%is defined through the identification ${\rm HH}^*(A,A)={\rm Ext}_{A^{e}}^*(A,A)$.
Thus, assume for the rest of this section that $A$ is a dg algebra. Recall that through the $A$-bimodule resolution $A \otimes^\pi BA \otimes^\pi A \to A$ we obtain a quasi-isomorphism $C^*(A,A)\simeq {\rm RHom}_{A^{\rm e}}(A,A)$. Given a morphism $\zeta : A\to s^nA$ in $D(A^{e})$ we get a family of morphisms $\{ M\otimes^L_A \zeta : M\to s^nM\}$ in $D(A)$ which define a natural transformation $1_{D(A)}\to s^n$. This determines a map ${\rm Ext}_{A^{e}}^*(A,A)\to  Z_{\rm gr} (DA)$.

Let us be more explicit. A Hochschild cocycle $\xi:BA\to A$ canonically lifts to a bilinear $_A\xi_A:A\otimes^\pi BA\otimes^\pi A\to A$, and for any dg module $M$ we get an $A$-linear map
\[
_M\xi_A\ =M\otimes_A {}_A\xi_A\  : \ M\otimes^\pi BA\otimes^\pi A\longrightarrow M.
\]
Since $ M\otimes^\pi BA\otimes^\pi A$ is a semi-free resolution of $M$, this ${}_M\xi_A$ represents an element of ${\rm Ext}_A(M,M)$. Naturality up to homotopy is easy to see, so the family $\{{}_M\xi_A\}$ defines an element of $Z_{\rm gr} D(A)$. 
In this way we obtain a morphism of graded rings 
\[
{\rm Char}:{\rm HH}^*(A,A)\longrightarrow Z_{\rm gr} (DA),
\]
which we call the {\bf characteristic action} of ${\rm HH}^*(A,A)$ on $D(A)$. % Who should we refer to for this?
Note that restricting to $A$ and to $k$ we recover the projection and shearing morphisms. That is, one can check from the definitions that we have two commuting triangles
\[
    \xymatrix{
    			       & {\rm HH}^*(A,A)  \ar[d] \ar[dl]_\Pi \ar[dr]^\chi & \\ 
    	 		{\rm H}(A) & Z_{\rm gr}D(A)   \ar[l]_{\rm res}  \ar[r]^{\rm res} &   {\rm Ext}_A(k,k).
		 }
\]
The above description suggests a further enhancement of the characteristic action. First, note that $ {\rm Hom}^\pi \!(M\otimes^\pi \! BA,M)\cong {\rm Hom}_A(M\otimes^\pi \!BA\otimes^\pi\! A,M)$ has a `convolution product' given by
\[
\phi\smile\psi\ :\  M\otimes BA\to M\otimes BA\otimes BA\xrightarrow{\ \psi\otimes BA\ } M\otimes BA\xrightarrow{\ \phi\ } M,
\]
for $\phi,\psi:M\otimes BA\to M$. With this product the natural map ${\rm Hom}_A(M,M)\to {\rm Hom}^\pi\!(M\otimes^\pi\! BA,M) $ becomes a quasi-isomorphism of dg algebras (for $h$-projective $M$). In fact, this extends to a `convolution enhancement' $D^{\rm conv}(A)$ of the derived category $D(A)$, with ${}_ND^{\rm conv}(A)_M={\rm Hom}^\pi\!(M\otimes^\pi\! BA,N)$. The natural functor $D^{\rm dg}(A)\to D^{\rm conv}(A)$ is then a quasi-equivalence of pre-triangulated dg categories, where $D^{\rm dg}(A)$ is  the usual  enhancement consisting of $h$-projective dg modules and $A$-linear maps.

The point is that the characteristic action now lifts to a map of dg algebras
\[
C^*(A,A)\to {\rm Hom}^\pi\!(M\otimes^\pi\! BA,M)\quad\quad  \xi:BA\to A\quad \mapsto \quad {}_M\xi :M\otimes BA\xrightarrow{\ M\otimes \xi\ } M\otimes A\to M.
\]
In fact it lifts further, all the way to the Hochschild cochain complex of the endomorphism algebra of $M$, as long as one is willing to work with two enhancements at once. More generally, one can do this for any set of objects together. 

\begin{prop}\label{charlift} Let $M$ be a set of objects in $D(A)$ and let $R=REnd_A(M)$ be the full dg subcategory of $D^{\rm dg}(A)$ on $M$. Denote also by $R^{\rm conv}$ the corresponding full dg subcategory of $D^{\rm conv}(A)$. The characteristic action of ${\rm HH}^*(A,A)$ on $H(R)$ lifts canonically to map $\widetilde{Char}$ of dg algebras, as in the following commutative diagram:
\[
    \xymatrix@C=15mm{
			  {\rm HH}^*(A,A)\ar[dr] & C^*(A,A)'  \ar[r]^-{\ \widetilde{\rm Char}\ }        \ar[dr]_{\rm Char} &  C^*(R,R^{\rm conv})' \ar[d]^{\Pi} &   C^*(R,R)'   \ar[l]_-\sim \ar[d]^-{\Pi}   \\
			  & Z_{\rm gr} {\rm H}(R) &  \widetilde{Z} (R^{\rm conv}) & \widetilde{Z} (R). \ar[l]_-\sim
		 }
\]
\end{prop}
In the proposition $\widetilde{Z} (R)$ denotes the `pre-centre' $\prod_{m\in M} {\rm REnd}_{A}(m)$, and similarly for $R^{\rm conv}$. Also recall that $C^*(-,- )'$ denotes the unreduced Hochschild cochain complex for non-augmented dg categories, as explained in section \ref{unreduced}.

The lift $\widetilde{\rm Char}$ takes a cochain $\xi :BA\to A$ to the composition $BR\xrightarrow{\eta} k_M\xrightarrow{\alpha} R^{\rm conv}$, where $\alpha$ takes $1_m$ to $(-)\cdot \xi$ in ${\rm Hom}^\pi(m\otimes^\pi BA,m)$. It is not difficult to verify that this is an anti-homomorphism of dg algebras (i.e. reversing the order of composition, but in particular respecting the differential), but we can make things clearer by writing them in a more symmetrical way. We can think of $M$ as an $R-A$ bimodule (abusively denoted using the same symbol). Just rewriting things through the tensor-hom adjunction, the top row in the diagram of proposition \ref{charlift} becomes\footnote{
The twist on ${\rm Hom}^\pi(BR\otimes^\pi M \otimes^\pi BA, M)$ is the one which makes it isomorphic to ${\rm Hom}_{R-A}(R\otimes^\pi BR\otimes^\pi M \otimes^\pi BA\otimes^\pi A, M)$.
}
\[
{\rm Hom}^\pi(BA,A)\to {\rm Hom}^\pi(BR\otimes^\pi M \otimes^\pi BA, M)\leftarrow {\rm Hom}^\pi(BR,R)
\]
\[
\xi\ \mapsto\  \eta\cdot(-)\cdot \xi \quad\quad\quad\quad  \zeta\cdot(-)\cdot \eta  \ \mapsfrom\  \zeta
\]
where $\eta$ is being used for the counits of both $BA$ and $BR$. This symmetry suggests the following definition: a bimodule ${}_DM_C$ over two dg categories $C$ and $D$  is called {\bf homologically balanced} if the natural maps $D\to {\rm REnd}_C(M)$ and $C^{\rm op}\to {\rm REnd}_{D^{\rm op}}(M)$ are both quasi-equivalences\footnote{
These derived homs are enriched as dg categories, which means computed object-wise.
}, this is essentially the condition Keller considers in \cite{KellerDIH}. In the context ${}_RM_A$ we have been discussing, the map $R\to {\rm REnd}_A(M)$ is a quasi-equivalence by definition.

%need reference for this terminology.

\begin{thm}\label{hbalanced}
Let $M$ be a set of objects in $D(A)$ and let $R=REnd_A(M)$ be the full dg subcategory of $D^{\rm dg}(A)$ on $M$. If the bimodule ${}_RM_A$ is homologically balanced then $\widetilde{\rm Char} : C^*(A,A)'\to C^*(R,R^{\rm conv})'$ is a quasi-isomorphism. Hence there is a canonical chain of quasi-isomorphisms 
\[
C^*(A,A) \xrightarrow{\ \sim\ } C^*(R,R)'.
\]
In the case that $M$ contains the free module $A$, the restriction map  $C^*(R,R)'\rightarrow C^*(A,A)'$ is a quasi-isomorphism, inverse in cohomology to $\widetilde{\rm Char}$.
\end{thm}

The advantage of the last statement is that the restriction map is one of $B_\infty$-algebras (as Keller points out in \cite{KellerDIH}), and also that it respects the weight filtration.

\begin{proof}
In light of the above discussion we just need to show that
\[
\Phi:{\rm Hom}^\pi(BA,A)\to {\rm Hom}^\pi(BR\otimes^\pi\! M \otimes^\pi\! BA, M)
\]
is a quasi-isomorphism. But using the tensor-hom adjunction on the other side we have $ {\rm Hom}^\pi(BR\otimes^\pi\! M \otimes^\pi \!BA, M) \cong {\rm Hom}^\pi( BA, {\rm Hom}^\pi\!( BR\otimes^\pi\! M, M))$ (as usual checking that the various twists match up). Note though that ${\rm Hom}^\pi\!( BR\otimes^\pi\! M, M)$ is simply a `convolution' model for ${\rm REnd}_{R^{\rm op}}(M)$. By assumption, the map $A^{\rm op}\to {\rm Hom}^\pi\!( BR\otimes^\pi\! M ,M)$ is a quasi-isomorphism, and by proposition \ref{cocompletecomparison} this makes $\Phi$ a quasi-isomophism as well.

To make the last statement more precise, we mean that the upper triangle in the following diagram commutes after taking cohomology
\[
    \xymatrix{
			   &{\rm Hom}^\pi\!(BR,R) \ar[d]_{\rm res} \ar[r]^-\sim     &     {\rm Hom}^\pi\!(BR,R^{\rm conv})\ar[d]_{\rm res} &\\
			  &{\rm Hom}^\pi\!(BA,A)   \ar[r]^-\sim \ar[ur]^{\widetilde{\rm Char}}     &   {\rm Hom}^\pi\!(BA,A^{\rm conv}) \ar@{=}[r]^-{\sim} & {\rm Hom}^\pi\!(BA\otimes^\pi\!A\otimes^\pi\!BA,A).
		 }
\]
The outer rectangle clearly commutes. The lower triangle commutes in cohomology because the two paths are equalised by a quasi-isomorphism, as follows:
\[
{\rm Hom}^\pi\!(BA,A)\rightrightarrows{\rm Hom}^\pi\!(BA\otimes^\pi\!A\otimes^\pi\!BA,A)\xrightarrow{\sim} {\rm Hom}^\pi\!(BA,A).
\]
The two maps on the left are the left and right actions $\xi\mapsto  \xi\cdot(-)\cdot\eta$ and $ \eta\cdot(-)\cdot\xi $ respectively. The final map is obtained by pulling back along the natural bicomodule map $BA\to BA\otimes^\pi\!A\otimes^\pi\!BA$. Either composition from left to right is the identity. The statement follows.
\end{proof}

\begin{rem} Since $\widetilde{\rm Char}$ is an anti-homomorphism, this actually results in a strange proof that Hochschild cohomology is commutative. \end{rem}

It is well known that the restriction $C^*(D^{\rm dg}_{\rm perf}(A),D^{\rm dg}_{\rm perf}(A))'\to C^*(A,A)$ is a quasi-isomorphism (this follows from  Keller's theorem \cite{MR1258406}, for example). Less well known, it follows from the work of To\"en \cite{MR2276263} that the restriction $C^*(D^{\rm dg}(A),D^{\rm dg}(A))'\to C^*(A,A)$ is a quasi-isomorphism (see also \cite[section 4.3]{KellerICM}). We'll deduce this by taking $M$ to be the collection of all dg modules in $D(A)$. Those uneasy about the size of $D(A)$ can just use a sufficiently large set instead, but this subtlety will not be important here.

\begin{thm}\label{big-morita-invariance} The restriction map $C^*(D^{\rm dg}(A),D^{\rm dg}(A))'\to C^*(A,A)$ is a quasi-isomorphism. The projection map from ${\rm HH}^*(D^{\rm dg}A,D^{\rm dg}A)$ and the characteristic action of ${\rm HH}^*(A,A)$ coincide. That is, the following diagram commutes
\[
    \xymatrix@C=0mm{
			  {\rm HH}^*(D^{\rm dg}A,D^{\rm dg}A)   \ar[rr]^-{\cong}_-{\rm res}\ar[dr]_-{\Pi} & & {\rm HH}^*(A,A)          \ar[dl]^{\rm Char}   \\
			 &  Z_{\rm gr} D (A). &
		 }
\]
\end{thm}

\begin{proof}
Because of theorem \ref{hbalanced}, we just need to show that $D=D^{\rm dg}A$ is homologically balanced as a $D-A$ bimodule. We have a sequence of quasi-fully-faithful functors 
\[
    \xymatrix{
			  D^{\rm dg}_{\rm perf}(A^{\rm op}) \ar[rr]^{{\rm RHom}_A(-,A)}_\sim  \ar@/_1.5pc/[rrrr]_{D\otimes^L_A-} & & (D^{\rm dg}_{\rm perf}A)^{\rm op}  \ar@{^(->}[r] &   D^{\rm op}   \ar[r]^-{\bf y} & D^{\rm dg}_{\rm perf}(D^{\rm op}),
		 }
\]
where the Yoneda embedding ${\bf y}$ is quasi-fully-faithful by a dg version of the Yoneda lemma (note that dg modules in the image of ${\bf y}$ are automatically perfect and semi-free). One can check that the diagram commutes (up to a natural quasi-isomorphism). It then follows from the fact that $D\otimes^L_A-$ is quasi-fully-faithful that $A^{\rm op}\to {\rm REnd}_{D^{\rm op}}(D)$ is a quasi-isomorphism.
\end{proof}

If $A$ is an augmented dg algebra which is strongly connected, then according to theorem \ref{doubledual} (and the fact that $B(A^{\rm op})\cong (BA)^{\rm op}$) the augmentation module $k$ is homologically balanced as an $A^!-A$ bimodule. Taking $M$ in \ref{charlift} and \ref{hbalanced} to consist only of the augmentation module, we obtain:

\begin{cor}\label{almost-main-theorem} If $A$ is a strongly connected dg algebra then there is a canonical isomorphism making the following diagram commute:
\[
    \xymatrix@C=0mm{
			  {\rm HH}^*(A,A)   \ar[rr]_-{\cong}^-{\widetilde{\rm Char}} \ar[dr]_-\chi & & {\rm HH}^*(A^!,A^!)          \ar[dl]^-{\Pi}   \\
			 &  Z_{\rm gr} ({\rm H}A^!). &
		 }
\]
\end{cor}

This is close to how Keller builds his isomorphism in \cite{KellerDIH}, where he recasts things in terms of the restriction maps (as in theorem \ref{hbalanced}) to see that the isomorphism lifts to one in the homotopy category of $B_\infty$-algebras. However, our approach makes it clear that the isomorphism $\widetilde{\rm Char}$ interacts well with the characteristic action.

It follows that the image of the shearing map ${\rm HH}^*(A,A)\to {\rm Ext}_A(k,k)$ is the same as the image of the projection ${\rm HH}^*(A^!,A^!)\to {\rm H}(A^!)$. In the next section we interpret this image in terms of higher structure on ${\rm Ext}_A(k,k)$.

These connectedness hypotheses can be relatively restrictive. To catch examples such as the group algebra of a finite $p$-group (with ${\rm char} k = p$), we will over the next few sections give a more careful (and independent) proof of this corollary, making use of the background developed at the beginning of the paper. Another defect of this approach is that it is not clear in corollary \ref{almost-main-theorem} what is happening with the two filtrations from section \ref{filtrationsection}. This will also be repaired in the coming sections.

\subsection{Commutativity for $A_{\infty}$-algebras and the $A_{\infty}$-centre}\label{commutativity}%%%%%%%%%%%%%%%%%%%%%%%%%%%%%%%%%%%%%%%%%%%%%%%%%%%%%%%%%%%%%%%%%%%%%%%%%%%%%%

% TO DO:
% Examples
% write out homotopy equation for p with signs

% MAYBE TO DO:
% discuss strict centre? Probably do this in the example section when needed. [I will do it. -V]
% Possibly include other consequences of definition, if I think of any good ones.
% Try to prove that (or decide whether) the A-infinity centre is a strict-subalgebra
% Could it be true that A_\infty-commuative is equivalent to A^Lie completely abelian? Probably not.
% Could it be equivalent to $E_2$?

There are several possible notions of centre, and of commutativity, for $A_\infty$-algebras, all generalising the usual concepts for graded algebras. A common answer is that a commutative $A_\infty$-algebra is a $C_\infty$-algebra. In characteristic zero certain aspects of commutative algebra generalise well to $C_\infty$-algebras.\footnote{Ultimately, $E_\infty$-algebras (in various settings) encapsulate the right notion of commutative algebras  up to homotopy, as far as generalising commutative algebra goes. Commutativity at this level is an extra structure on, rather than a property of, an $A_\infty$-algebra (or some other kind of algebra). The commutativity property we introduce below is really a completely different (and much weaker) notion.
}
 However, this does not obviously give rise to a notion of centre for $A_\infty$-algebras. A more immediate disadvantage is that this property is not even invariant under isomorphisms of $A_\infty$-algebras.

Another possibility is to assert that an $A_\infty$-algebra $A$ is commutative if $A^{\rm Lie}$ is completely abelian. %Pros and cons. 
This is much weaker than being a $C_\infty$-algebra (or isomorphic to one). We will discuss an intermediate notion in this paper.

We would like a notion of centre which is a quasi-isomorphism invariant, and for this we pass to minimal models (note that the centre of a dg algebra is not a quasi-isomorphism invariant). Being somewhat rigid (quasi-isomorphisms coincide with isomorphisms) a good notion of centre is perhaps a reasonable hope for minimal $A_\infty$-algebras.

Let $A$ be an $A_\infty$-algebra. The space of {\bf homotopy derivations} of $A$ is by definition the positive weight part of the suspended Hochschild cochain complex
\[
{\rm hoder}(A,A) =s{\rm Hom}^\pi(\overline{BA},A) = sF^1_{\Pi}C^*(A,A).
\]
%There is an obvious relative notion ${\rm hoder}(A,A')=s{\rm Hom}^\tau(\overline{BA},A')\subseteq s C^*(A,A')$ for a map of $A_\infty$-algebras $A\to A'$.
It is a sub-Lie algebra of $C^*(A,A)$. The reason for this notation is the following proposition, which goes back to Stasheff and Schlessinger \cite{MR814187}, and Quillen. %ADDREF?
A proof is in \cite{Hinich}, or more explicitly in \cite[Lemma 4.2]{FMT}.

\begin{prop}
${\rm hoder}(A,A)$ is quasi-isomorphism invariant, and there is a canonical chain of quasi-isomorphisms
\[
{\rm hoder}(A,A) \xrightarrow{\ \sim\ } {\rm der}(\mathcal{A}, \mathcal{A})
\]
of dg Lie algebras, where $\mathcal{A}=\Omega BA$. Or more generally, $\mathcal{A}$ can be any semi-free dg algebra quasi-isomorphic to $A$.
\end{prop}

Quasi-isomorphism invariance is established in the same way as for the full Hochschild cochain complex. Let's sketch the proof for the case $\mathcal{A}=\Omega BA$. There is a canonical quasi-isomorphism of $A_\infty$-algebras $A \to \Omega BA$.  So we have quasi-isomorphisms 
${\rm Hom}^\pi\!(\overline{BA},A)  \xrightarrow{\sim} {\rm Hom}^\iota\!(\overline{BA},\Omega BA)$ by  proposition \ref{cocompletecomparison} as usual, and there is a canonical isomorphism ${\rm Hom}^\iota\!(\overline{BA},\Omega BA) \cong {\rm der}(\mathcal{A}, \mathcal{A})$ by the adjunction formula (\ref{adjunctionlemma}) of section \ref{models-for-HH}. 
F\'{e}lix, Menichi and Thomas observe \cite{FMT} that the Lie structure is respected because the map ${\rm der}(\mathcal{A}, \mathcal{A}) \to {\rm hoder}(\mathcal{A},\mathcal{A}),\ $ $p\mapsto (-1)^{|p|}sp\pi$ is one of dg Lie algebras (knowing already that the Gerstenhaber Lie algebra structure on ${\rm hoder}$ is respected by quasi-isomorphisms, by a standard Mayer-Vietoris argument).

Now, from the short exact sequence
\[
0\to s^{-1}{\rm hoder}(A,A)\to C^*(A,A)\to A\to 0
\]
we get a connecting homomorphism ${\rm H}(A)\xrightarrow{{\rm ad}} {\rm H}({\rm hoder}(A,A))$.  In fact, this canonically lifts to the chain level: the assignment
\[
{\rm ad}:A\longrightarrow{\rm hoder}(A,A) \quad\quad  a\mapsto \pi b_+([a]\sh -)
\]
is a chain map giving rise to the above connecting homomorphism in homology. Here $\pi b_+([a]\sh -)\in{\rm Hom}^\pi(\overline{BA},A)$ should be interpreted in terms of the sequence of higher commutators against $a$ from section \ref{models-for-HH}. In terms of the model $\mathcal{A}=\Omega BA$ this is the classical
\[
{\rm ad}:\mathcal{A}\longrightarrow {\rm der}(\mathcal{A},\mathcal{A}) \quad a\mapsto [a,-]. % Explain this?
\]
\begin{defn}
With this is mind, we define the $A_\infty$-centre of a minimal $A_\infty$-algebra $A$ to be
%\[
%Z_\infty(A) = \left\{\ a\in A \ : \ {\rm ad}(a) \text{ is zero in } {\rm H}({\rm hoder}(A,A))\ \right\}.
%\]
\[
Z_\infty(A)  = {\rm ker} \left(A \xrightarrow{\ {\rm ad}\ }  {\rm H}\,{\rm hoder}(A,A)\right) = {\rm im}\left({\rm HH}^*(A,A)\xrightarrow{\ \Pi\ } A\right). 
\]
\end{defn}
So, $a$ is $A_\infty$-central if there exists $ p\in {\rm hoder}(A,A) $ with ${\rm ad}(a)=\partial_\pi(p)$. We can interpret the projection morphism as the obvious map
\[
\Pi\ :\ C^*(A,A) = s^{-1}{\rm hoder}(A,A) \rtimes A\longrightarrow A,
\]
just as we saw shearing morphism in terms of coderivations.

Writing this out explicitly in terms of the higher multiplications, this means that an element $a$ is central if for $n\geq 2$ the higher commutators $[a;-]_{1,n}:A^{\otimes n}\to A$ vanish together up to a sequence of `homotopies' $p_i:A^{\otimes i}\to A$ of degree $|a|-i$ for $ i\geq 1$, meaning precisely that
\[
[a; -]_{1,n}=\sum_{r+s+t=n} (-1)^{r(|a|+s)+t(|a|+1)}m_{r+1+s}(1^{\otimes r} \otimes p_s\otimes 1^{\otimes t} ) - (-1)^{|a|} (-1)^{rs+t}p_{r+1+t}(1^{\otimes r} \otimes m_s\otimes 1^{\otimes t} ).
\]
In particular the usual commutator $[a;-]_{1,1} = 0$. The intuition is that keeping track of the homotopy $p$ for which ${\rm ad}(a) = \partial_\pi(p)$ gives the homotopy (or derived) centre, 
%\[
%{\rm HH}^*(A,A) =
%\left\{
%\ (a,p)\ :\  {\rm ad}(a)=d(p) \text{ where }a\in A \text{ and } p\in {\rm hoder}(A,A)/({\rm boundaries}) \ 
%\right\},
%\]
and forgetting $p$ is exactly the projection morphism.

The reason for introducing $Z_\infty(A)$ in terms of ${\rm hoder}(A,A)$ is that this perspective will be computationally useful. In particular, examples show that these homotopies for the higher commutators can sometimes be ignored.

A simple but important observation is that the $A_\infty$-centre is contained in the graded centre:
\[
Z_\infty(A)\subseteq Z_{\rm gr }(A)
\]
for any minimal $A_\infty$-algebra. If $A$ is formal they coincide, but there are natural examples where the $A_\infty$-centre is much smaller than the graded centre. Forgetting for a moment the higher structure, the next observation one makes is that $Z_\infty(A)$, being the image of a map of graded algebras, is itself a graded commutative subalgebra of $A$. 

Recall from \ref{projection-shearing-section} that a quasi-isomorphism of $A_\infty$-algebras $\phi:A\to A'$ induces a commutative diagram
\[
    \xymatrix{
                    C^*(A,A) \ar[r]  \ar[d]^{\wr}_{\phi^*}   &    A   \ar[d]^{\wr}_{\phi} \\
                    C^*(A',A') \ar[r]     &     A'.
		    }
\]
Taking $A$ and $A'$ to be isomorphic minimal $A_\infty$-algebras, it follows from this that $Z_\infty A\cong Z_\infty A'$ as graded rings. In particular, if $A$ is a dg algebra, then the $A_\infty$-centre $Z_\infty ({\rm H}A)$ of a minimal model is an invariant of the quasi-isomorphism type of $A$. In fact $Z_\infty({\rm H}A)$ is just the image of $\Pi:{\rm HH}^*(A,A)\to {\rm H}(A)$.

This definition may seem overly strict from a homotopy theoretic point of view. There are a few reasons to consider it. Firstly, since $Z_\infty(A)=Z_{\rm gr}(A)$ in the formal case, we will be able to conceptually recover some classical results for Koszul algebras (see section \ref{exchange}).  Secondly, the assumption that $A=Z_\infty (A)$ actually has interesting consequences, and is satisfied in many natural examples. Finally, the image of the projection ${\rm HH}^*(A,A)\to {\rm H}(A)$ is a graded commutative algebra which is an invariant of the homotopy type of $A$, and so is likely worth investigating. Interpreting it in terms of the higher structure on a minimal model at least seems to be computationally and conceptually useful. Regardless, the true homotopy centre of $A$ remains $C^*(A,A)$.

If $C$ is an augmented minimal $A_\infty$-category the above discussion applies, and we define the $A_\infty$ centre of $A$ to be the image of ${\rm HH}^*(C,C)\to Z_{\rm gr} C$. In the not necessarily augmented case this discussion does not apply directly, but we can make the same definition. We simply declare the $A_\infty$-centre of non-augmented minimal $A_\infty$-category (or algebra) to be the image of ${\rm HH}^*(C,C)\to Z_{\rm gr} C$. Taking a minimal model one sees $Z_\infty {\rm H}C ={\rm HH}^*(C,C)/F^1_\Pi$. With these definitions in place let us record some consequences of the previous section.

\begin{cor}\label{char-infinity-cor-i} The image of the characteristic action ${\rm HH}^*(A, A) \to  Z_{\rm gr}(D A)$ is the $A_{\infty}$-centre of $D (A)$, after it is enhanced to a triangulated $A_\infty$-category quasi-equivalent to $D^{\rm dg}(A)$.

\end{cor}

\begin{cor}\label{char-infinity-cor-ii} For every dg module $M$, the characteristic morphism ${\rm HH}^*(A, A) \to {\rm Ext}_A(M, M)$ lands in the $A_{\infty}$-centre of the Yoneda algebra ${\rm Ext}_A(M, M)$. If $M$ is homologically balanced as an ${\rm REnd}_A(M)-A$ bimodule then the characteristic morphism surjects onto the $A_\infty$-centre of ${\rm Ext}_A(M,M)$.
\end{cor}

 Examples of homologically balanced modules include any generator for $D_{\rm perf}(A)$, and when $A$ is strongly connected, the augmentation module $k$. The first corollary follows from theorem \ref{big-morita-invariance}, the second from \ref{charlift} and \ref{hbalanced}.

% NOT NEEDED:

%\begin{proof} Denote $R={\rm REnd}_A(M)$, so ${\rm H}(R)={\rm Ext}_A(M,M) $. Using the fully faithful embedding $R\hookrightarrow \underline{A}$ of dg categories, the first statement is visible from the following commutative diagram
%\[
%    \xymatrix{
%                    {\rm HH}^*(\underline{A},\underline{A}) \ar[r]^{\Pi \,=\,{\rm Char}}  \ar[d]^{\text{res}}   &     Z_\infty (D_{\rm perf} A)   \ar[d] \\
%                    {\rm HH}^*(R, R) \ar[r]^-{\Pi}     &     Z_\infty {\rm Ext}_A(M,M)  .
%		    }
%\]
%If we assume $M$ is a generator, then ${\rm RHom}_A(M,-):\underline{A}\to \underline{R}$ is a quasi-equivalence of dg categories, in which case the statement is visible from the following diagram
%\[
%    \xymatrix{
%                    {\rm HH}^*(\underline{R},\underline{R}) \ar[r]^{\cong}  \ar[d]^{\cong}   &    {\rm HH}^*(\underline{A},\underline{A})   \ar[d]^{\rm Char} \\
%                    {\rm HH}^*(R, R) \ar@{->>}[r]^-{\Pi}     &     Z_\infty {\rm Ext}_A(M,M)
%		    }
%\]
%which commutes according to theorem \ref{moritaprojection}.
%\end{proof}

\noindent
\begin{rem}Classically, a first order deformation of a graded algebra $A$ is classified by a weight $2$ element $\xi$ of ${\rm HH}^*(A,A)$. The obstruction to continuing a module $M$ along such a deformation is ${\rm Char}(\xi)=[{}_M\xi_A]\in {\rm Ext}_A(M,M)$. The corollary tells us that the obstruction must actually be in $Z_\infty {\rm Ext}_A(M,M)$, where ${\rm Ext}_A(M,M)$ has been enhanced to be a minimal model of ${\rm RHom}_A(M,M)$.
\end{rem}

Another consequence of the previous section is that the $A_\infty$-centre as defined here is unfortunately not derived Morita invariant, see example \ref{BGG} below (this question was raised by Keller after seeing an earlier draft of the paper). 

\begin{defn}
We say that a minimal $A_\infty$-algebra $A$ is $A_\infty$-commutative if $Z_\infty A = A$, or equivalently if the projection map ${\rm HH}^*(A,A)\to A$ is surjective.% In general, an algebra (dg or $A_\infty$) is called homotopy commutative if its minimal model is $A_\infty$-commutative.
\end{defn}

In section \ref{stringtop} we will need the following fact.

\begin{prop}\label{completelyabelian}
Assume $k$ has characteristic zero. If $A$ is $A_{\infty}$-commutative then $A^{\rm Lie}$ is completely abelian.
\end{prop}

\begin{proof} % Proof needs standardized notation. Is p a cochain or its corresponding coderivation? a coderivation.
By induction on weight we show that the bar differential $b$ vanishes on the symmetric tensors $\mathcal{C}^{co}\!A^{\rm Lie}\subseteq  BA$.

Since $\mathbb{Q}\subseteq k$, any weight $n+1$ symmetric element can be written as a linear combination of elements of the form $[a] \sh x$, where $a$ is in $A$ and $x$ is symmetric of weight $n$. Since $a$ is in the image of the projection map there is a coderivation $p$ which (weakly) decreases weight such that $[b, [a]\sh -] = [b,p]$. Thus
\[
 b([a] \sh x)= bp(x) +(-1)^{|a|} pb(x) -(-1)^{|a|}[a] \sh b(x).
\]
Since $p$ decreases weight and preserves $\mathcal{C}^{co}\!A^{\rm Lie}$ this formula and the inductive hypotheses on $b$ means  $b([a] \sh x)=0$.
\end{proof}

We can rephrase this entirely in terms of dg algebras.

\begin{cor}\label{formalassociatedlie}
Assume $k$ has characteristic zero. Let $A$ be a dg algebra such that ${\rm HH}^*(A,A)\to {\rm H}(A)$ is surjective. %(that is, $A$ has an $A_\infty$-commutative minimal model).
Then the associated dg Lie algebra $A^{\rm Lie}$ is quasi-abelian (that is, formal and quasi-isomorphic to an abelian Lie algebra). %(that is, $A^{\rm Lie}$ is homotopy abelian).
\end{cor}

Such an $A$ need not be quasi-isomorphic to a commutative algebra, so corollary \ref{formalassociatedlie} is a small example of what subtler forms of commutativity can be deduced from $A_\infty$-commutativity of the minimal model ${\rm H}(A)$. In the proof one only needs to remember that $(-)^{\rm Lie}$ is functorial even for non-strict morphisms, and preserves quasi-isomorphisms.

In section \ref{examples} we will give examples of interesting $A_\infty$-centres, explaining how one can compute them using the philosophy of this section. For now, let us note down a few examples of $A_\infty$-commutative algebras. % enumerate?

Any algebra with the homotopy type of a commutative dg algebra (so, any $C_\infty$-algebra) will be $A_\infty$-commutative, but our condition is much weaker than this.

The algebra $C^*(X;k)$ of cochains on a space $X$ is always $A_\infty$-commutative (more generally, $E_\infty$-algebras satisfy this condition). % see the last section for a geometric reason when $X$ is simply connected.

The shearing map for is split surjective for any Hopf algebra. By theorem \ref{maintheorem} this means the Koszul dual to a Hopf algebra (with additional finiteness assumptions) is always $A_\infty$-commutative.

One can see (e.g \cite[theorem 7]{MR1321701}) from the defining formulas that for any algebra over the ``brace operad" (denoted $\mathcal{S}_2$ in \cite{MR1969208}) the projection morphism is split surjective (alternatively it is shown in \cite{BBCD} that the bar construction of such an algebra is a Hopf algebra). It is proven in \cite{MR1890736} that this brace operad $\mathcal{S}_2$ is an $E_2$ operad (quasi-isomorphic to chains on the little $2$-disk operad), see also \cite{MR1969208}. Since this operad is $\Sigma$-split, by \cite{Hinich} any $E_2$-algebra has a model which is a brace algebra, and hence any $E_2$-algebra is $A_\infty$-commutative.

%Even iace andlhence any $E_2$-algebra is $A_\infty$-commutative.g ebran , acharacteristic zero there are examples of dg algebras with surjective shearing map which are not quasi-isomorphic to a commutative dg algebra (it is the same to produce a minimal $A_\infty$-algebra which is $A_\infty$-commutative but not isomorphic to a $C_\infty$-algebra, which can be done by hand from the defining relations).

\subsection{Hochschild cohomology and Koszul duality}%%%%%%%%%%%%%%%%%%%%%%%%%%%%%%%%%%%%%%%%%%%%%%%%%%%%%%%%%%%%%%%%%%%%%%%%%%%%%%

 \label{HHKoszul}

%Remark: As already mentioned (we should mention this) everything we will do can be done while carrying around an Adams grading (by any Abelian group). If we assume $A$ is Adams connected (meaning Adams graded by $\mathbb{Z}$ and the homology is locally finite and connected in the Adams direction) then $A$ is weakly connected. However, it is important to realise that $(BA)^*$ should then be interpreted as the bigraded dual of $BA$. Thus, it is not true that if $A$ admits an Adams connected grading then $A$ is weakly connected (in the non-Adams sense), because the true Koszul dual may be much larger than the bigraded version. That Buenos Aires guy fails to mention this.

Buchweitz proved in 2003 that if $A$ is a Koszul algebra then ${\rm HH}^*(A,A)\cong {\rm HH}^*(A^!,A^!)$ as algebras \cite{BuchweitzCanberra}. Later, F\'{e}lix, Menichi and Thomas proved for any simply connected dg algebra that ${\rm HH}^*(A,A)\cong {\rm HH}^*(A^{!}, A^{!})$ as Gersenhaber algebras \cite{FMT}. Their theorem can be deduced from the results of section \ref{projection-shearing-section}. In order to fix the construction for use below, we give a proof of their theorem now which works in greater generality (although in this section we will not address the Lie structure).

Assume $A$ is dualisable. According to proposition \ref{LHtheorem} we have a weak equivalence $\rho:B\Omega A^*\to A^*$. This is a quasi-isomorphism whenever $A$ complete, by \ref{HHKoszulhypotheses}.  For example, this includes any strongly connected $A_\infty$-algebra, as well as any finite dimensional algebra whose augmentation ideal is nilpotent (such as the group algebra of a finite $p$-group, in characteristic $p$).

\begin{thm} \label{HHKoszulTHM}
If $A$ is a weakly connected $A_\infty$-algebra which (possibly after taking a minimal model) is complete, then there is a canonical isomorphism
\[
C^*(A, A) \xrightarrow{\ \sim\ } C^*(A^{!}, A^{!})
\]
in the homotopy category of $A_\infty$-algebras.
\end{thm}

In the statement $A^!$ is any $A_\infty$-algebra with fixed choice of quasi-isomorphism $A^!\xrightarrow{\sim} (BA)^*$. %Surprisingly, this doesn't need to come from a weak equivalence $BA\to A^\kos$.
The isomorphism is essentially that of F\'{e}lix, Menichi and Thomas (but simplified, thanks our background on twisting cochains). They assume their algebra is simply connected (in particular, strongly connected). Later we will frequently use the identification ${\rm HH}^*(A, A) \cong {\rm HH}^*(A^{!}, A^{!})$, and this should always be interpreted as the isomorphism constructed below.

\begin{proof}
Let's first assume that $A$ is minimal, so $BA$ is locally finite and $(BA)^*=\Omega(A^*)$. We have maps
\[
{\rm Hom}^\pi(BA, A)  \xrightarrow{\bf \ D\ } {\rm Hom}^{\pi^*}\!(A^*, (BA)^*)\xrightarrow{\ \rho^*\ } {\rm Hom}^\pi(B(( B A)^*), ( B A)^*).
\]
Lemma \ref{Dlemma} says that ${\bf D}$ is an isomorphism of $A_\infty$-algebras, while $\rho^*$ is a quasi-isomorphism of $A_\infty$-algebras by lemma \ref{completecomparison}, using completeness of the algebra $(BA)^*$. 

We are supposed to interpret $A^{!}$ in the theorem as an $A_\infty$-algebra with a fixed choice of quasi-isomorphism $\phi:A^!\xrightarrow{\sim} (BA)^*$.  So we finish by using proposition \ref{cocompletecomparison} to get a chain of quasi-isomorphisms
$
C^*(\phi):C^*(( B A)^*, ( B A)^*) \xrightarrow{\sim} \!\! \cdot \!\!\xleftarrow{\sim} C^*(A^!, A^!)
$ 
in the usual way. According to proposition \ref{twistednaturality} these are quasi-isomorphisms of $A_\infty$-algebras.
% obtaining an isomorphism $C^*(A,A)\xrightarrow{\cong} C^*(A^!,A^!)$ in the homotopy category of $A_\infty$-algebras, and shifted Lie algebras. 

In general we need to make a choice of minimal model $\phi :A\xrightarrow{\sim} A'$. Since any two choices for $\phi$ are isomorphic it is easy to check that the homotopy class of the zig-zag
\[
C^*(A,A)   \xrightarrow{\ \sim\ } \! \cdot \!\xleftarrow{\ \sim\ } C^*(A',A')   \xrightarrow{\ \rho^*{\bf D}\ } C^*( B((BA')^*), (BA')^*)    \xrightarrow{\ \sim\ } \!  \cdot  \!\xleftarrow{\ \sim\ }  C^*( B((BA)^*), (BA)^*)
\]%add sim
does not depend on $\phi$.
\end{proof}

%Maybe this lemma is a bit over the top. (V asks: do we need it anywhere?)
%\begin{lem}\label{HHKoszulNAT}
%The isomorphism just constructed is natural in the sense that the following square of isomorphisms commutes
%\[
%    \xymatrix{
%    			  {\rm HH}^*(A,A)   \ar[r]^-{\cong}            \ar[d]_-{\cong} &  {\rm HH}^*(A^!,A^!) \ar[d]^-\cong \\
%			   {\rm HH}^*(A',A') \ar[r]^-\cong &  {\rm HH}^*(A'^!,A'^!)
%		 }
%\]
%for any quasi-isomorphism $A\to A'$ of weakly connected $A_\infty$-algebras.
%\end{lem}

%remarks on general version without bar-finiteness using $\widehat{\Omega} (A^*)$?

\subsection{Koszul duality exchanges the shearing and projection morphisms}%%%%%%%%%%%%%%%%%%%%%%%%%%%%%%%%%%%%%%%%%%%%%%%%%%%%%%%%%%%%%%
\label{exchange}
% TO DO:
% Add references

From now on we assume $A$ is an algebra (dg or $A_\infty$) which satisfies the hypothesis of theorem \ref{HHKoszulTHM}. Because the proof of  \ref{HHKoszulTHM} is quite simple, it is straightforward to check the results of this section. However, the main theorem generalises much of what is known about the image of the shearing map, and seems to be extremely useful in computations (for example, to check whether one has a good theory of support varieties).% Give ref for this? We will already have discussed this in the introduction, and will probably talk a bit about support theory in the example section.

The Hochschild cohomology of $A$ has a great deal of structure. %It is a Gerstenhaber algebra, and ${\rm HH}^*(A,A)\cong {\rm HH}^*(A^{!}, A^{!})$ as Gerstenhaber algebras, according to \cite{FMT}. % There is even more structure the Hochschild cochain complex, greatly extending the Gerstenhaber structure (although, this is still just an $E_2$ structure), which may be summed up by saying that $C^*(A,A)$ is a $B_{\infty}$-algebra (we will refrain here from explaining what this means). % except to say that a $B_\infty$-algebra structure on $A$ is the data of a product on $BA$. 
Notably, $C^*(A,A)$ forms a $B_\infty$-algebra, and Keller has established (in the strongly connected situation) that $C^*(A,A)$ and $C^*(A^!,A^!)$ are isomorphic in the homotopy category of $B_\infty$-algebras (see \cite{KellerDIH} for an explanation of this terminology).  
Despite this picture, we discuss in this section a structure on ${\rm HH}^*(A,A)$ not preserved by Koszul duality. In short, ${\rm HH}^*(A,A)$ admits two filtrations, and Koszul duality exchanges rather than preserves them.

%The point is that proposition \ref{shearingflitration} may be rephrased as 

\begin{thm}\label{maintheorem}
Koszul duality exchanges the shearing and projection morphisms:
\[
    \xymatrix@R=12mm{
    			 {\rm HH}^*(A,A)       \ar@{=}[r]^\sim   \ar[d]_<<<<<\Pi  \ar[dr]_<<<<<\chi  & {\rm HH}^*(A^!,A^!)   \ar[d]^<<<<<\Pi \ar[dl]^<<<<<\chi \\
			 {\rm H}(A)   & {\rm H}(A^!).
		 }
\]
Consequently, the image of the shearing map is the $A_{\infty}$-centre of the minimal model ${\rm H}(A^!)$.

\end{thm}

\begin{proof} First assume $A$ is minimal, so the constructions in the proof of theorem \ref{HHKoszulTHM} apply directly. The verification is simple, but we'll do it explicitly. To see that the left hand triangle commutes we only need notice that the following diagram commutes
\[
    \xymatrix@R=5mm{   
    			{\rm Hom}^\pi(BA, A) \ar[d]_\Pi       \ar[r]^-{\bf D} &  {\rm Hom}^{\pi^*}\!(A^*, (BA)^*) \ar[d]  \ar[r]^-{\rho^*}& {\rm Hom}^\pi(B(( B A)^*), ( B A)^*) \ar[d]^\chi \\
			    A \ar@{=}[r] & {\rm Hom}(A^{*},k) \ar[r]^-{{\rm Hom}(\rho,k)} &  {\rm Hom}(B(( B A)^*),k).
		 }
\]
Similarly, for the right hand triangle we check that the following diagram commutes
\[
    \xymatrix@R=5mm{
    			{\rm Hom}^\pi(BA, A)  \ar[d]_\chi         \ar[r]^-{\bf D} &  {\rm Hom}^{\pi^*}\!(A^*, (BA)^*)   \ar[r]^-{\rho^*} \ar[d] & {\rm Hom}^\pi(B(( B A)^*), ( B A)^*) \ar[d]^\Pi \\
			    {\rm Hom}(BA,k) \ar@{=}[r] & (BA)^*  \ar@{=}[r]  &  ( B A)^*.
		 }
\]
In general we need to take a minimal model $A\xrightarrow{\sim} A'$. %
The projection and shearing morphisms are obviously compatible with the chain of quasi-isomorphisms $ C(\phi):C^*(A',A')\rightarrow \!\! \cdot \!\!\leftarrow C^*(A,A)$. The same goes for the choice of Koszul dual $(BA)^*\xrightarrow{\sim}A^!$, and the result follows.
\end{proof}

This theorem directly generalises the following result of Buchweitz, Green, Snashall and Solberg \cite{MR2461267}.

\begin{cor}\label{BGSS}
If $A$ is a Koszul algebra then the image of the shearing map $\chi: {\rm HH}^*(A,A)\to {\rm Ext}^*_A(k,k)$ is the graded centre of ${\rm Ext}^*_A(k,k)$.
\end{cor}

We will see that there are plenty of non-formal minimal algebras whose $A_\infty$-centre coincides with their graded centre, including examples which are the Koszul duals of honest graded algebras. Thus, corollary \ref{BGSS} is true for many more algebras than Koszul algebras.
%"Speculation of Keller"

\begin{exmp} 
\label{BGG}
A simple but important example is the Koszul duality between two symmetric algebras, which lies behind the so called BGG correspondence \cite{BGG}. 
Let $V$ be a  graded vector space and set $S={\rm Sym}V$. Then $S$ is Koszul with acyclic twisting cochain $\tau:\Lambda^{co}={\rm Sym}^{co}sV\twoheadrightarrow V\hookrightarrow S$, so the Koszul dual algebra is $\Lambda = {\rm Sym}(s^{-1}V^*)$. To avoid extra finiteness conditions we should interpret Hochschild cohomology as bigraded in the usual way for algebras. The differential $[\tau,-]$ on ${\rm Hom}^{\tau}\!(\Lambda^{co},S)$ vanishes because $S$ is commutative and $\Lambda^{co}$ is  cocommutative, so we obtain  the Hochschild Kostant Rosenberg theorem %for connected graded algebras
$C^*(S,S)\simeq {\rm Hom}(\Lambda^{co},S) = \Lambda\widehat{\otimes} S$. Similarly for $C^*(\Lambda,\Lambda)$. The situation of theorem \ref{maintheorem} is then
\[
    \xymatrix@R=3mm{
    			  {\rm HH}^*(S,S)  \ar@{=}[d]   & {\rm HH}^*(\Lambda, \Lambda)  \ar@{=}[d]  \\  
    			    \Lambda\widehat{\otimes} S \ar[dd]_<<<<<\Pi  \ar[ddr]_<<<<<\chi \ar@{=}[r]^\sim  & S \widehat{\otimes} \Lambda   \ar[dd]^<<<<<\Pi \ar[ddl]^<<<<<\chi \\
			 \\
			 S   & \Lambda.
		 }
\]%This looks terrible, I have messed this up.
The maps are the obvious ones, and the triangles clearly commute. While the situation involves no higher structure, this example quite transparently illustrates theorem \ref{maintheorem}. %When is it it okay to use tensor product instead of Hom? Is it worth it?

In fact, this is a surprisingly general example: in section \ref{stringtop} we'll see that if one has a strongly connected commutative dg algebra for which the shearing map is surjective, then one is in the BGG situation above (as long as $\mathbb{Q}\subseteq k$).

Taking $M$ to be the right dg module $S\oplus k$ over $S$, and $R={\rm REnd}_S(M)$, one can deduce from the results of section \ref{projection-shearing-section} that there are isomorphisms
\[
{\rm HH}^*(S,S)\xleftarrow{\ \cong\ } {\rm HH}^*(R,R) \xrightarrow{\ \cong\ }{\rm HH}^*(\Lambda,\Lambda)
 \]
both respecting the weight filtrations. It follows that $Z_\infty({\rm H}R)$ maps surjectively onto both $S$ and $\Lambda$. Since $R$ is derived Morita equivalent to $S$, this means that the $A_\infty$-centre is not derived Morita invariant.
\end{exmp}

Now we go back to assuming $A$ is weakly connected with a complete minimal model. We have already remarked that the projection map is 
$
\Pi: {\rm HH}^*(A,A)\twoheadrightarrow {\rm HH}^*(A,A)/F_\Pi^1 \hookrightarrow {\rm H}(A)
$, 
so ${\rm HH}^*(A,A)/F_\Pi^1 $ is isomorphic to the $A_\infty$-centre of the minimal model ${\rm H}(A)$. Similarly, the shearing map is $
\chi: {\rm HH}^*(A,A)\twoheadrightarrow {\rm HH}^*(A,A)/F_\chi^1 \hookrightarrow {\rm H}(A^!)
$. Thus, the following theorem directly upgrades \ref{maintheorem}.

\begin{thm}\label{filtrationswitch}
The weight filtration and the shearing filtration are exchanged by the Koszul duality isomorphism ${\rm HH}^*(A,A)\cong {\rm HH}^*(A^!,A^!)$. That is, this restricts to an isomorphism
\[
F_\Pi^n{\rm HH}^*(A,A)\cong F_\chi^n{\rm HH}^*(A^!,A^!).
\]
If $A$ is strongly connected connected it restricts to an isomorphism $
F_\chi^n{\rm HH}^*(A,A)\cong F_\Pi^n{\rm HH}^*(A^!,A^!)$ as well.
\end{thm}

\begin{proof}
By propositions \ref{weightnaturality} and \ref{shearingnaturality} we may assume that $A$ is minimal that $A^!=(BA)^*$, and we may compute the shearing filtration on ${\rm HH}^*(A^!,A^!)$ using ${\rm Hom}^\pi(B(( B A)^*), ( B_{\geq n} A)^*)$, since $(BA)^*=\Omega A^*$ is semi-free and $(B_{\geq n}A)^*=\Omega (A^*)^{[n]}$. % Need A^* to be cocomplete 

Clearly ${\bf D}$ is restricts to an isomorphism between the weight filtration on $C^*(A,A)$ and the filtration $ {\rm Hom}^{\pi^*}\!(A^*, ( B_{\geq n} A)^*)$. Note that $(B_{\geq n}A)^*$ has a complete filtration for each $n$, so we can use the argument of \ref{completecomparison} to see that $\rho$ restricts to a quasi-isomorphism  $ {\rm Hom}^{\pi^*}\!(A^*, ( B_{\geq n} A)^*)\xrightarrow{\sim} {\rm Hom}^\pi(B(( B A)^*), ( B_{\geq n} A)^*)$.

The last statement follows by exchanging the roles of $A$ and $A^!$, after checking that in the strongly connected case the isomorphism ${\rm HH}^*(A,A)\xrightarrow{\cong} {\rm HH}^*(A^!,A^!)$ of theorem \ref{HHKoszulTHM} is inverse to the isomorphism ${\rm HH}^*(A^!,A^!)\xrightarrow{\cong} {\rm HH}^*(A^{!!},A^{!!})\cong {\rm HH}^*(A,A)$.
\end{proof}

So far we have not mentioned the fact that in the classical theory of Koszul duality, algebras are usually assumed connected and generated in degree $1$. From our point of view, this is simply a technical condition which is convenient because it makes the radical filtration coincide with the grading filtration. But let us adopt this convention and compare the above theorem to the classical situation discovered by Buchweitz \cite{BuchweitzCanberra}.

Thus, suppose $V$ is a bigraded vector space concentrated in degree $(0,1)$,
and $R\subseteq V^{\otimes 2}$ is such that $A=T^{a}V/\langle R\rangle$ is Koszul. Let the suspension $s$ have degree $(-1,0)$\footnote{
One can totalise $|(p,q)| := p-q$ to make $A$ strongly connected, so that the rest of the paper applies directly.
}. These conventions give ${\rm HH}^{*,*}(A,A)= {\rm Ext}^{*,*}_{A^e}(A,A)$ its classical bigrading by cohomological degree (weight) and internal degree. Let $C=C(sV,s^2R)$. The weight grading on ${\rm Hom}^\tau\!(C,A)^{m,n}= {\rm Hom}^\tau\!(C^{(-m,m)},A^{(0,m+n)})$ is by $m$ and the shearing grading is by $m+n$ (this is at least partially responsible for the term ``shearing"). The Koszul dual $A^!=C^*$ is now generated in degree $(1,-1)$ and $C^*(A^!,A^!)\simeq {\rm Hom}^{\tau^*}\!(A^*,C^*)$. We see easily that $
{\rm HH}^{m,n}(A,A)\cong {\rm HH}^{m,n}(A^!,A^!)$. 
Classically, one might regrade $A^!$ to also be generated in degree $(0,1)$. If we call this regraded algebra $A^!_{cl}$, then we have instead
\[
{\rm HH}^{m,n}(A,A)\cong {\rm HH}^{m+n,-n}(A_{cl}^!,A_{cl}^!),
\]
which illustrates theorem \ref{filtrationswitch}. One can also compare this with the picture obtained by Green, Snashall, Solberg and Zacharia \cite{MR3574211} in which the authors identify the diagonal subalgebra $\bigoplus_{m}{\rm HH}^{m,-m}(A,A)$ as the graded centre $\bigoplus_{n}{\rm HH}^{0,n}(A^!,A^!)$ of the Koszul dual. % And they also....

\section{Calculations of $A_{\infty}$-centres}
\label{examples}

% Fix coalgebra construction for d-koszul algebras.
% Mention Tate construction for Com-Lie koszul duality.

In this section we will more explicitly work over a commutative semisimple $k$-algebra base $\Bbbk$, e.g. $\Bbbk = k$ or $\Bbbk = kQ_0$ for some finite quiver $Q$. Let $(-)^*$ be the $\Bbbk$-linear dual, $\otimes$ = $\otimes_{\Bbbk}$ and $A$ denote an algebra of the form $A = T_{\Bbbk}V/\langle R \rangle = TV/\langle R \rangle$ with $R \hookrightarrow T_{\geq 2}V$. We always assume $A$ locally finite and weakly connected in that either $V = V^{\leq 0}$ or $V = V^{\geq 2}$; $R$ is always homogeneous with respect to the grading on $V$, but not necessarily with the tensor weight grading on $T_{\Bbbk}V$. For instance, $A$ can be any finite dimensional path algebra $A = A^0 = kQ/I$.

We have up to now used dg notation for ${\rm Ext}^n_A(\Bbbk, \Bbbk)$, where the $n$ implicitly meant total grading. If $A$ is a graded algebra with the above hypotheses, thought of as a dg algebra with trivial differential, we shall separate\footnote{This wouldn't be correct if A weren't weakly connected, since the bigraded dual of ${\rm Tor}^A(\Bbbk, \Bbbk)$ might then be smaller than ${\rm Ext}_A(\Bbbk, \Bbbk)$.} the cohomological grading from the internal one, and write ${\rm Ext}^n_A(\Bbbk, \Bbbk) = \bigoplus_{p+q=n} {\rm Ext}^{p, q}_A(\Bbbk, \Bbbk)$ where $p$ denotes cohomological degree, and similarly write ${\rm HH}^{*,*}(A, A)$ and ${\rm Tor}_{*, *}^A(\Bbbk, \Bbbk)$.

We begin by computing the $A_\infty$-structure on ${\rm Ext}_A(\Bbbk, \Bbbk)$. The algebra $A$ admits a minimal semi-free resolution $(TW, d) \xrightarrow{\sim} A$, the minimal Tate resolution, where minimal means that $d(W) \subseteq T_{\geq 2}W$. This can be constructed iteratively by setting $W_0 = V$, $W_1 = sR$ with differential $d: W_1 = sR \to T_{\geq 2}V = T_{\geq 2}W_0$, and iteratively adding higher homological degree\footnote{These can be taken compatible with the internal grading on $V$.} cycles $W_{n+1}$ in a minimal way to kill cohomology classes in the kernel of $(TW_{\leq n}, d) \twoheadrightarrow A$. See \cite{BL} for details on the construction. % Minimal models in homotopy theory, by Baus-Lemaire.

The data of $(TW, d)$ defines a cocomplete\footnote{The filtration on W given by the Tate construction is an admissible filtration by construction.} $A_\infty$-coalgebra structure on $C = k \oplus \overline{C}$ where $s^{-1}\overline{C} = W$, so that $\Omega C = (TW, d)$ tautologically. Minimality of $(TW, d)$ corresponds to minimality of $C$, and this defines an acyclic twisting cochain $\tau: C \to A$ through $\phi_\tau: \Omega C = (TW, d) \xrightarrow{\sim} A$. As $C$ is cocomplete, taking bar constructions and using Prop. \ref{LHtheorem} yields $C \xleftarrow{\sim} B\Omega C \xrightarrow{\sim} BA$, exhibiting $C$ as a minimal model for ${\rm Tor}^{A}(\Bbbk, \Bbbk)$; in particular, if $gldim(A) = n \leq \infty$, then the Tate resolution stops at the $(n-1
)^{{\rm th}}$ stage with $W_{n-1} \cong {\rm Tor}_{n, *}^A(\Bbbk, \Bbbk)$. The graded dual $C^* = {\rm Ext}_A(\Bbbk, \Bbbk)$ is a minimal model for $A^!$, and the dual of the differential on $(TW, d)$ encodes its $A_\infty$-structure. Note that ${\rm Ext}^{*, *}_A(\Bbbk, \Bbbk)$ is strictly unital and augmented over ${\rm Ext}^{0, *}_A(\Bbbk, \Bbbk) = \Bbbk$.

Since ${\rm Tor}_{1, *}^A(\Bbbk, \Bbbk) = sV$ and ${\rm Tor}_{2, *}^A(\Bbbk, \Bbbk) = s^{2}R$, we have ${\rm Ext}^{1, *}_A(\Bbbk, \Bbbk) = s^{-1}V^*$ and ${\rm Ext}^{2, *}_A(\Bbbk, \Bbbk) = s^{-2}R^*$. Let $\iota_n: R \hookrightarrow T_{\geq 2}(V) \xrightarrow{\pi_n} V^{\otimes n}$. This map appears as a component of the differential in the first stage of the Tate resolution, which implies that the $m_n: {\rm Ext}^{1, *}_A(\Bbbk, \Bbbk)^{\otimes n} \to {\rm Ext}^{2, *}_A(\Bbbk, \Bbbk)$ are dual to the relations: 

\begin{prop}[Keller, \cite{MR2067371}]\label{relations}The following diagram commutes:
\[
\xymatrixcolsep{6pc}\xymatrix@1{ {\rm Ext}^{1, *}_A(\Bbbk, \Bbbk)^{\otimes n}\ar@{=}[d] \ar[r]^{m_n} & {\rm Ext}^{2, *}_A(\Bbbk, \Bbbk) \ar@{=}[d]\\
(s^{-1}V^*)^{\otimes n} \ar[r]^{-(-1)^{\binom{n}{2}}s^{-2}(\iota_n)^*s^{\otimes n}} & s^{-2}R^*.}
\]

In particular, $m_n(s^{-1}v_1 \otimes ... \otimes s^{-1}v_n) = -s^{-2}\iota_n^*(v_1 \otimes ... \otimes v_n)$ whenever $|v_1| = ... = |v_n| = 0$.
\end{prop}

$ $
\begin{rem}
If $A$ is a commutative algebra (in characteristic zero), one can similarly read the $L_\infty$-structure on the minimal Koszul dual Lie coalgebra $L$ from any minimal semifree resolution $\left( {\rm Sym}(V'), d \right) \xrightarrow{\sim} A$. 
In particular, if $A = {\rm Sym}(V)/\langle W \rangle$ is a complete intersection with $W \hookrightarrow {\rm Sym}^{\geq 2}(V)$ a minimal subspace of relations, the Koszul resolution $\left({\rm Sym}(V \oplus sW), d\right) \xrightarrow{\sim} A$ presents a very short minimal resolution of $A$, and this gives a simple presentation of $L = sV \oplus s^2W$ by setting $\left({\rm Sym}(V \oplus sW), d\right) = \mathcal{C}(L)$. Then $L^* = s^{-1}V^* \oplus s^{-2}W^*$ is a minimal $L_\infty$-algebra whose higher brackets are encoded by the Koszul differential. Baranovsky \cite{Baranovsky} has shown that ${\rm Ext}_A(k, k) = \mathcal{U}(L^*)$, see section \ref{LieCom-section}.  In a future paper we will explicitly describe the full $A_\infty$-structure on ${\rm Ext}_A(k, k) = \mathcal{U}(L^*)$ directly from the Koszul differential, by using a different construction of Baranovsky's $\mathcal{U}(-)$.

\end{rem}

Now, assume that gldim$(A) = g < \infty$. The conditions defining a top degree element $\xi \in Z_\infty{\rm Ext}^{g, *}_A(\Bbbk, \Bbbk)$ are vacuous aside from $[\xi, s] = 0$ for $s \in {\rm Ext}^{0, *}_A(\Bbbk, \Bbbk) = \Bbbk$, since the non-trivial higher products against $\xi$ strictly increment cohomological degree as is easily verified. The condition $[\xi, s] = 0$ cuts out the symmetric part of the $\Bbbk$-bimodule ${\rm Ext}^{g, *}_A(\Bbbk, \Bbbk)$, which we denote ${\rm Ext}^{g, *}_A(\Bbbk, \Bbbk)_{cyc}$ since these will correspond to cycles in the quiver path algebra case. We have shown:
\begin{prop}\label{topdegree}
If ${\rm gldim}(A) = g < \infty$, then the image of $\chi: {\rm HH}^{g, *}(A, A) \to {\rm Ext}^{g, *}_A(\Bbbk, \Bbbk)$ is ${\rm Ext}^{g, *}_A(\Bbbk, \Bbbk)_{cyc}$.
\end{prop}

We now turn to the calculations of $Z_\infty{\rm Ext}_A(\Bbbk, \Bbbk)$ for specific classes of algebras.

\subsubsection*{Algebras of global dimension 2}

Restricting now to gldim$(A) = 2$, Prop. \ref{relations} determines the entire $A_\infty$-algebra ${\rm Ext}^{*,*}_A(\Bbbk, \Bbbk)$, since all other products vanish for strict unitality or degree reasons. We are left to determine $Z_\infty{\rm Ext}^{1, *}_A(\Bbbk, \Bbbk)$. 

\begin{prop} Let gldim$(A) = 2$. If $A = A^0$ is finite dimensional, or $A$ is AS-regular over $\Bbbk = k$ generated in degree $-1$, then the image of $\chi: {\rm HH}^*(A, A) \to {\rm Ext}_A(\Bbbk, \Bbbk)$ is the graded centre $Z_{gr}{\rm Ext}_A(\Bbbk, \Bbbk)$.
\end{prop}
\begin{proof} First assume that $A = A^0$ is finite dimensional. Then $Z_{gr}{\rm Ext}^{1, *}(\Bbbk, \Bbbk) = 0$ by the No Loop Conjecture, and the result follows applying Prop. \ref{topdegree} on ${\rm Ext}^{2, *}_A(\Bbbk, \Bbbk)$. For the other case, AS-regular algebras over $\Bbbk = k$ generated in degree $-1$ are well-known to be Koszul, and the result follows from \ref{BGSS}.
\end{proof}

This proposition makes it hard to come up with natural examples in global dimension $2$ where $Z_\infty{\rm Ext}^*_A(\Bbbk, \Bbbk)$ differs from $Z_{\rm gr}{\rm Ext}^*_A(\Bbbk, \Bbbk)$. Such examples will come up more naturally in the next class of algebras.

\subsubsection*{d-Koszul algebras}

The theory of quadratic Koszul duality laid out in Sect. \ref{Koszuldualitysection} extends naturally to $d$-homogeneous algebras, with the minor difference that $d$-Koszul algebras are no longer characterized by formality of ${\rm RHom}_A(\Bbbk, \Bbbk)$, but rather minimize the higher structure on ${\rm Ext}^*_A(\Bbbk, \Bbbk)$ with regards to the constraints of Prop. $\ref{relations}$. We follow the exposition of Dotsenko-Valette \cite{MR3110804}, based on the results of He-Lu \cite{MR2172343}. Berger introduced $d$-Koszul algebras in \cite{MR1832913}. Readers primarily interested in the examples should skip ahead to Thm. \ref{d-koszul-structure} of He-Lu, which gives the $A_\infty$-structure on ${\rm Ext}_A(\Bbbk, \Bbbk)$.

Fix a locally finite $\Bbbk$-bimodule $V = V^{\leq 0}$ or $V = V^{\geq 2}$. To a $d$-homogeneous datum $(V, R)$, where $R \hookrightarrow V^{\otimes d}$, we associate an algebra and coalgebra pair as follows: the algebra $A = A(V, R)$ is given by $TV/\langle R \rangle$, which is universal over maps $V \to A$ such that $V^{\otimes d} \to A^{\otimes d} \xrightarrow{m^{(d)}} A$ factors through $V^{\otimes d}/R$. 

We could define $C$ by the dual universal property, but the requirement that $C$ models ${\rm Tor}^A_*(\Bbbk, \Bbbk)$ for ``good'' $d$-homogeneous algebras forces us to impose the condition that $\Delta_d: C_2 \to C_1^{\otimes d}$ correspond to $\iota: s^{2}R \to R \hookrightarrow (V)^{\otimes d} \to (sV)^{\otimes d}$, with the other coproducts $\Delta_n: C_2 \to C_1^{\otimes n}$ zero. Hence we do not define $C(V, R)$ but rather $C(sV, s^{2}R)$ as depending on the shifted data $(sV, s^2R)$.

Call an $A_\infty$-coalgebra {\bf $(2,d)$-reduced} if its differential is zero, $\Delta_n = 0$ for $n \neq 2, d$, and if composing $\Delta_d$ with itself in any order vanishes. The coalgebra $C = C(sV, s^2R)$ is the universal $(2,d)$-reduced coalgebra equipped with maps $C \to sV$ and $C \to s^2R$ such that there is a factorization
\[
\xymatrixrowsep{1pc}\xymatrix@1{C \ar@{.>}[rd] \ar[r]^{\Delta_d} & C^{\otimes d}  \ar[r] & (sV)^{\otimes d} \\
                                 & s^2R \ar[ur]_{s^{d}\iota s^{-2}}
            }
\]
and such that if $d \neq 2$ the composition $C \xrightarrow{\Delta_2} C^{\otimes 2} \to  (sV)^{\otimes 2}$ is zero.

For $C'$ a $(2,d)$-reduced coalgebra, any such maps factor through a unique strict morphism $C \to C'$; in particular, $C$ receives a canonical map from ${\rm Tor}_{\leq 2, *}^A(\Bbbk, \Bbbk)$, at least once its existence has been established.

Following Dotsenko-Valette, one constructs $C = k \oplus \overline{C}$ directly as follows: let $\overline{C}_{(n)}$ be the weight $n$ subspace of $\overline{T}^{co}V$ given by $\overline{C}_{(n)} = \bigcap_{p+d+q=n} V^{\otimes p}\otimes R \otimes V^{\otimes q}$. Let $\overline{C}_{p,q}$ := $s^p\overline{C}_{(q)}$ sit in homological degree $p$ with weight $q$, and define $\overline{C} = \big(\bigoplus_{n \geq 0} \overline{C}_{2n, nd} \big) \oplus \big(\bigoplus_{n \geq 0} \overline{C}_{2n+1, nd+1}\big)$. For $d = 2$ this recovers $C$ inside $T^{co}(sV) \subset BA$ from Sect. \ref{Koszuldualitysection}. 

Let $\big(\overline{D}, \overline{\Delta} \big)$ stand for the subcoalgebra $\bigoplus_{n \geq 0} \overline{C}_{(n)} \subset \overline{T}^{co}V$. We define  $\overline{\Delta}_2$ and $\overline{\Delta}_d$ on $\overline{C}_{(nd)}$ and $\overline{C}_{(nd+1)}$ by iterating $\overline{\Delta}$ and projecting down:
\begin{align*}
& \overline{\Delta_2}: \overline{C}_{(nd+1)} \xrightarrow{{\rm pr}\overline{\Delta}^{(2)}} \left(\bigoplus_{n_1 + n_2 = n} \overline{C}_{(n_1d+1)} \otimes \overline{C}_{(n_2d)} \right) \bigoplus \left(\bigoplus_{m_1 + m_2 = n} \overline{C}_{(m_1d)} \otimes \overline{C}_{(m_2d +1)}\right)\\
& \overline{\Delta_2}: \overline{C}_{(nd)} \xrightarrow{{\rm pr }\overline{\Delta}^{(2)}} \bigoplus_{n_1 + n_2 = n} \overline{C}_{(n_1d)} \otimes \overline{C}_{(n_2d)}\\
& \overline{\Delta_d}: \overline{C}_{(nd)} \xrightarrow{{\rm pr }\overline{\Delta}^{(d)}} \bigoplus_{i_1 + ... + i_d = n-1} \overline{C}_{(i_1d+1)} \otimes ... \otimes \overline{C}_{(i_dd+1)}.
\end{align*}
The remaining coproducts are zero. Being careful with signs, one extends this to $s^{2n}C_{(nd)} = C_{2n,nd}$ and $s^{2n+1}C_{(nd+1)} = C_{2n+1,nd+1}$, and then to $C = k \oplus \overline{C}$ as a strictly counital coalgebra. We let the reader verify that $\Delta_2$ is coassociative. Since $\Delta_d$ is zero on odd weights, composing it in any order vanishes, and the Stasheff identities are a simple verification from there. This yields a $(2,d)$-reduced coalgebra. Note that $\overline{\Delta}_d: \overline{C}_{(d)} \to (\overline{C}_{(1)})^{\otimes d}$ corresponds to $s^2 R \to R \hookrightarrow V^{\otimes d} \to (sV)^{\otimes d}$.

We have attached to $(V, R)$ an algebra and coalgebra pair $A = A(V, R)$ and $C = C(sV, s^2R)$. The map $\tau: C \to sV \to V \to A$ is easily seen to be a twisting cochain, in that $\phi_\tau: \Omega C \to A$ is a morphism of dg algebras. We say that $A$ is {\bf $d$-Koszul} if $\tau$ is acyclic, so that $\Omega C \xrightarrow{\sim} A$ which implies that $C = {\rm Tor}_*^A(\Bbbk, \Bbbk)$ as $A_\infty$-coalgebras. By a suitable extension of Thm \ref{FT-TC}, one can show that $\tau$ being acyclic is equivalent to the $d$-Koszul complex $C \otimes^{\tau}\! A$ being a resolution of $k_A$, where $C \otimes^\tau\! A$ is the twisted complex with differential 
\begin{align*}
C \otimes A \xrightarrow{\rho_R \otimes 1} C \otimes \Omega C \otimes A \xrightarrow{1 \otimes \phi_\tau \otimes 1} C \otimes A \otimes A \xrightarrow{1 \otimes m_2} C \otimes A
\end{align*}
where $\rho_R = \sum_{n \geq 2} \rho_R^{(n)}$ is the right coaction with $\rho_R^{(n)} = (1 \otimes (s^{-1})^{\otimes n-1}) \Delta_n: C \to C \otimes (s^{-1}\overline{C})^{\otimes n-1} \subset C \otimes \Omega C$. Using the above identifications and formulas for $\Delta_n$ and $\tau$, this differential becomes
\begin{align*}
    & C_{(nd)} \otimes A \xrightarrow{pr\Delta^{(d)}\otimes 1} C_{((n-1)d+1)} \otimes C_{(1)}^{\otimes (d-1)} \otimes A \xrightarrow{1 \otimes {\rm mult}} C_{((n-1)d+1)} \otimes A \\
    & C_{(nd+1)} \otimes A \xrightarrow{pr\Delta^{(2)} \otimes 1} C_{(nd)} \otimes C_{(1)} \otimes A \xrightarrow{1 \otimes {\rm mult}} C_{(nd)} \otimes A.
\end{align*}

In particular the differential alternates between maps of degree $1$ and $d-1$ in the weight grading on $A = TV/\langle R \rangle$. Acyclicity of this complex is usually given as the definition of $d$-Koszulity (\cite{GMMVZ}, \cite{MR1832913}).

The coalgebra $D = D(V, R) = k \oplus \overline{D}$ dualises to the $d$-homogeneous dual $A^{\vee} := D^* = T(V^*)/\langle R^{\perp} \rangle$ where $R^{\perp} \subseteq (V^*)^{\otimes d}$ is the annihilator of $R$ under the natural pairing. This lets us describe the $A_\infty$-structure on $C^* = {\rm Ext}^*_A(\Bbbk, \Bbbk)$ via $A^{\vee}$. Let $A^{\vee}_{(n)}$ be the tensor weight $n$ component of $A^{\vee} = T(V^*)/\langle R^{\perp}\rangle$, and let $A^{\vee}_{[nd]} = s^{-2n}A^{\vee}_{(nd)}$, $A^{\vee}_{[nd+1]} = s^{-2n-1}A^{\vee}_{(nd+1)}$ have the implicit cohomological degree shift.
\begin{thm}[He, Lu, \cite{MR2172343}]\label{d-koszul-structure} Let $A$ be $d$-Koszul, then there is an isomorphism 
\begin{align*}
& {\rm Ext}^{2n, *}_A(\Bbbk, \Bbbk) = A^{\vee}_{[nd]}\\
& {\rm Ext}^{2n+1, *}_A(\Bbbk, \Bbbk) = A^{\vee}_{[nd+1]}.
\end{align*}
The products $m_n$ are given by $m_n = 0$ for $n \neq 2, d$, and otherwise
\begin{align*}
  & m_2: \begin{cases} A^{\vee}_{[id]} \otimes A^{\vee}_{[jd]} \to A^{\vee}_{[(i+j)d]}\\  
  A^{\vee}_{[id+1]} \otimes A^{\vee}_{[jd]} \to A^{\vee}_{[(i+j)d + 1]}\\
  A^{\vee}_{[id]} \otimes A^{\vee}_{[jd+1]} \to A^{\vee}_{[(i+j)d + 1]}
  \end{cases}\\
  \\
  & m_d: A^{\vee}_{[i_1d+1]} \otimes ... \otimes A^{\vee}_{[i_dd+1]} \to A^{\vee}_{[(i_1 + ... + i_d + 1)d]}
\end{align*}
where $m_2$ is the product in $A^{\vee}$ and $m_d = m_2^{(d)}$ is the iterated product in $A^{\vee}$. There are no other products.
\end{thm}

% Comments on hypotheses.
% Did they first assume d-homogeneous?
For $A$ strongly connected, Green, Marcos, Mart\'inez-Villa and Zhang in \cite{GMMVZ} have shown that $d$-Koszulity of a $d$-homogeneous algebra is equivalent to ${\rm Ext}_A^{*, *}(\Bbbk, \Bbbk)$ being generated over $\Bbbk$ in cohomological degree $1$ and $2$. When $A = kQ/I$ with $I$ generated by a set $\rho$ of monomial paths in $kQ_d$, $d$-Koszulity is characterized by $\rho$ being a $d$-covering: for three paths $p,q,r$ of length $\geq 1$, if $pq, qr$ are in $\rho$ then every subpath of $pqr$ of length $d$ is in $\rho$. The $d$-truncated path algebras $kQ/(Q_d)$ are therefore $d$-Koszul.

Thm. \ref{d-koszul-structure} has for immediate consequences:
\begin{thm}\label{even} Let $A$ be $d$-Koszul. Then in even cohomological degrees, the $A_\infty$-centre of ${\rm Ext}_A(\Bbbk, \Bbbk)$ agrees with the graded centre:
\begin{align*}
    & Z_\infty{\rm Ext}^{even}_A(\Bbbk, \Bbbk) = Z_{gr}{\rm Ext}^{even}_A(\Bbbk, \Bbbk).
\end{align*}
\end{thm}
This follows since the higher products vanish on even inputs. Next, let $\mathcal{N}il$ stand for the homogeneous nilradical ideal.
\begin{cor}
Let $A$ be $d$-koszul and assume that $\overline{A}$ is nilpotent. Then the shearing map $\chi$ induces an isomorphism of graded algebras 
\[{\rm HH}(A, A)/\mathcal{N}il \cong Z_{gr}{\rm Ext}_A(\Bbbk, \Bbbk)/\mathcal{N}il.
\]
\end{cor}
Indeed since $m^{(n)}: \overline{A}^{\otimes n} \to \overline{A}$ is $0$ for $n \gg 0$ one easily sees that elements in ${\rm Hom}^{\pi}(BA, \overline{A})$ are nilpotent since then $f^{\smile n} = m^{(n)}(f \otimes ... \otimes f)\Delta^{(n)} = 0$ for $n \gg 0$, and thus ${\rm ker}\ \chi$ consists of nilpotent elements. If $d > 2$ then odd cohomological elements in ${\rm Ext}_A(\Bbbk, \Bbbk)$ are nilpotent, while for $d = 2$ there is no higher structure, and the result follows.

We are now in a position to calculate interesting examples using Thm. \ref{d-koszul-structure}.
\begin{exmp}\label{CPn}
Let $A = k[x]/(x^d)$ with $|x| = s$ even and $d > 2$. Then $A^! = {\rm Ext}^{*, *}_A(k, k) \cong Sym(\partial_x, \eta) = \bigwedge(\partial_x) \otimes k[\eta]$ with $|dx| = (1, -s)$ and $|\eta| = (2, -sd)$. Note that this is graded commutative; we will show that, characteristic depending, $\chi$ typically isn't onto, so it isn't always $A_\infty$-commutative. The non-trivial higher products are of the form\footnote{One may use $-\eta$ as representing the relation $x^d = 0$ in Prop. \ref{relations}.}
\begin{align*}
& m_d(\partial_x \otimes ... \otimes \partial_x) = \eta \\
& m_d( \eta^{i_1} \partial_x \otimes ... \otimes  \eta^{i_d} \partial_x) = \eta^{(\sum_{m=1}^{d}i_m)+1}
\end{align*}
Let us calculate the higher commutators. From above we see that $ad(\eta^j) = 0$, while $ad( \eta^{j}\partial_x)$ is given on input $\partial_x^{\otimes d-1}$ by
\[
ad( \eta^j \partial_x)(\partial_x \otimes ... \otimes \partial_x) = \sum_{i=0}^{d-1} (-1)^i (-1)^{i} m_d(\partial_x^{\otimes i} \otimes  \eta^{j} \partial_x \otimes \partial_x^{\otimes d-1-i}) = d \cdot \eta^{j+1}
\]
This is $0$ if and only if char $k \mid d$. When char $k \nmid d$, observe that this isn't in the image of a coboundary $\partial_\pi(p)$ for $p \in {\rm hoder}(A^!, A^!) = {\rm Hom}^\pi(\overline{BA}^!, A^!)$: writing $(s\partial_x)^{\otimes d-1} = [\partial_x | ... | \partial_x]$, we have
\begin{align*}
\partial_\pi(p)\big( [\partial_x | ... | \partial_x]\big) & = \partial(p)\big([\partial_x | ... | \partial_x]\big) + \sum_{n = 1}^{d-2}[\pi, ..., \pi; p]_{n, 1}\big([\partial_x | ... | \partial_x]\big)\\
                                                & = 0 + [\pi, p]\big([\partial_x | ... | \partial_x]\big)\\
                                                & = \pm \partial_x \cdot p([\partial_x | ... | \partial_x]) + \pm p([\partial_x | ... | \partial_x]) \cdot \partial_x
\end{align*}
and this is never of the form $d \cdot \eta^{j+1}$ for any $p: \overline{BA}^! \to A^!$. It follows that $[ad(\partial_x\eta^j)] \neq [0] \in {\rm H}\,{\rm hoder}(A^!, A^!)$, and so
\begin{align*}
    Z_\infty{\rm Ext}_A(k, k) = \begin{cases}
                                                k[\eta] & \textup{if char $k \nmid d$} \\
                                                \bigwedge(\partial_x) \otimes k[\eta] & \textup{if char $k \mid d$}. \\
                                                \end{cases}
\end{align*}

\end{exmp}
Note that for $d = p^r$ and char $k = p$ surjectivity follows since $k[x]/(x^d)$ is isomorphic to the Hopf algebra $k[\Z /d \Z]$.

\begin{exmp}
Let $A = kQ/(Q_3)$, where $Q$ is the quiver

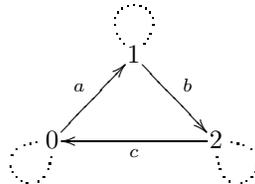
\[
\centerline{\xymatrix@1{ & 1 \ar@{.}@(ul, ur) \ar[rd]^{b} & \\
                        0 \ar@{.}@(d, l) \ar[ru]^{a} & & 2 \ar@{.}@(r, d) \ar[ll]^{c}
}}
\]
bound by relations $\{cba, acb, bac \}$. The $3$-homogeneous dual is given by $A^{\vee} = kQ^{op}$
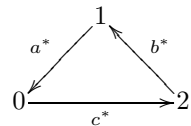
\[
\centerline{\xymatrix@1{ & 1 \ar[ld]_{a^*} & \\
                        0 \ar[rr]_{c^*} & & 2 \ar[lu]_{b^*} 
}}
\]
and we have $A^! = {\rm Ext}_A(\Bbbk, \Bbbk)$ given by $kQ'/I$ where $Q'$ is
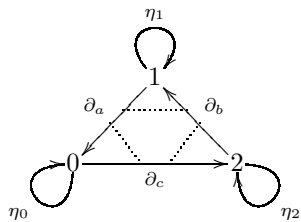
\[
\centerline{\xymatrix@1{ & 1 \ar@{.}[]-<1.7em,1.7em>;[d]-<0.5em,0em> \ar@{.}[]+<1.7em,-1.7em>;[d]+<0.5em,0em> \ar@{.}[]+<1.2em,-1.2em>;[]+<-1.2em,-1.2em> \ar^{\eta_1}@(ul, ur) \ar[ld]_{\partial_a} & \\
                        0 \ar^{\eta_0}@(d, l) \ar[rr]_{\partial_c} & & 2 \ar^{\eta_2}@(r, d) \ar[lu]_{\partial_b} 
}}
\]
where $\partial_a := s^{-1}a^*$, $\eta_0 = s^{-2}(a^* b^* c^*)$ is the 3-cycle in $A^{\vee}$ cohomologically shifted by $2$, and similarly for the rest. The relations are given by vanishing relations for the dotted paths, and otherwise $\eta_0 \partial_a = \partial_a \eta_1$, $\eta_1 \partial_b = \partial_b \eta_2$ and $\eta_2 \partial_c = \partial_c \eta_0$ inherited from $A^{\vee}$. The degree $2$ class $\eta := \eta_0 + \eta_1 + \eta_2$ is then central and generates a polynomial subalgebra $k[\eta] \subseteq Z_{gr}{\rm Ext}_A(\Bbbk, \Bbbk)$. The only higher products are triple products given by all variants of $
m_3(\partial_a \otimes \partial_b \otimes \partial_c) = \eta_0$, extended multilinearly over $k[\eta]$.

The lack of odd oriented cycle shows that $\eta$ generates $Z_{gr}{\rm Ext}_A(\Bbbk, \Bbbk)$, and since $|\eta|$ is even, Thm. \ref{even} shows that it is in the image of $\chi$. It follows that $Z_\infty{\rm Ext}_A(\Bbbk, \Bbbk) = k[\eta]$, and that ${\rm HH}^*(A, A)/\mathcal{N}il \cong k[\eta]$.
\end{exmp}

\begin{exmp}
Let $A = kQ/(Q_3)$ where $Q$ is the quiver
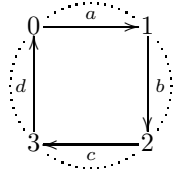
\[
\centerline{\xymatrixcolsep{3pc}\xymatrixrowsep{3pc}\xymatrix@1{0 \ar@{.}@/_9pt/[d] \ar[r]^a  & \ar@{.}@/_9pt/[l] 1 \ar[d]^b  \\
3 \ar@{.}@/_9pt/[r] \ar[u]^d & 2 \ar@{.}@/_9pt/[u] \ar[l]^c
}}
\]
bound by relations $\{dcb, cda, dab, abc\}$. Its $3$-homogeneous dual is $A^{\vee} = kQ^{op}$
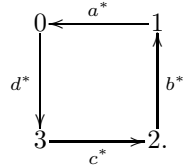
\[
\centerline{\xymatrixcolsep{3pc}\xymatrixrowsep{3pc}\xymatrix@1{0 \ar[d]_{d^*} & 1 \ar[l]_{a^*}  \\
3 \ar[r]_{c^*} & 2. \ar[u]_{b^*}
}}
\]
The algebra $A^! = {\rm Ext}_A(\Bbbk, \Bbbk)$ is given by $kQ'/I$ where $Q'$ is
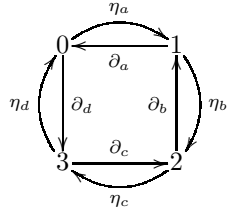
\[
\centerline{\xymatrixcolsep{3pc}\xymatrixrowsep{3pc}\xymatrix@1{
                  0 \ar@/^9pt/[r]^{\eta_a} \ar[d]^{\partial_d} & 1 \ar@/^9pt/[d]^{\eta_b} \ar[l]^{\partial_a}  \\
           3 \ar@/^9pt/[u]^{\eta_d} \ar[r]^{\partial_c} & 2 \ar@/^9pt/[l]^{\eta_c} \ar[u]^{\partial_b}
}}
\]
with $\partial_a = s^{-1}a^*$, $\eta_a = s^{-2}b^*c^*d^*$ as before, and similarly for $\partial_b, \partial_c, \partial_d, \eta_b, \eta_c, \eta_d$. The relations are given by zero composition amongst $\{ \partial_a, \partial_b, \partial_c, \partial_d \}$ and otherwise by equating the cycles $\partial_a \eta_a = \eta_d \partial_d$, $\partial_d \eta_d  = \eta_c \partial_c$, $\partial_c \eta_c  = \eta_b \partial_b$ and $\partial_b \eta_b  = \eta_a \partial_a$.

Let $\gamma_0 := \partial_a \eta_a = \eta_d \partial_d$ stand for the simple cycle at $0$ and $\eta_0 := \eta_d \eta_c \eta_b \eta_a$ the long cycle at $0$, and similarly define $\gamma_1, \gamma_2, \gamma_3, \eta_1, \eta_2, \eta_3$. Let $\gamma = \sum_{i=0}^3 \gamma_i$ and $\eta = \sum_{i=0}^3 \eta_i$. Tedious but straightforward calculations in $A^{\vee}$ show that $\gamma, \eta$ generate $Z_{gr}{\rm Ext}_A(\Bbbk, \Bbbk) \cong \bigwedge(\gamma) \otimes_k k[\eta]$ with $|\gamma| = 3$ and $|\eta| = 8$. Since $|\eta|$ is even, we know that $\eta$ is in the image of $\chi$, with only classes of the form $\eta^r \gamma$ left to determine.

To compute the higher products, denote $\partial_{i \to j}$ the degree $1$ arrow from $i$ to $j$, and denote by $\eta_{i \to j}$ some monomial path in $\{\eta_0, \eta_1, \eta_2, \eta_3\}$ going from $i$ to $j$ of length $\eta_{ij}$. Thm \ref{d-koszul-structure} shows that all higher products are sums of products of the form
\[
m_3(\eta_{n \to o} \ \partial_{m \to n}\otimes \eta_{l \to m}\ \partial_{k \to l} \otimes \eta_{j\to k} \ \partial_{i \to j}) = \eta_{i \to o}
\]

where $\eta_{i \to o}$ is the path of length $\eta_{jk} + \eta_{lm} + \eta_{no} + 1$ going from $i$ to $o$. In particular we see that
\[m_3(\eta^r \gamma \otimes \eta_{l \to m} \partial_{k \to l} \otimes \eta_{j \to k}\partial_{i\to j}) = m_3(\eta_{l \to m} \partial_{k \to l} \otimes \eta^r \gamma \otimes \eta_{j \to k}\partial_{i\to j}) = m_3(\eta_{l \to m} \partial_{k \to l} \otimes \eta_{j \to k}\partial_{i\to j} \otimes \eta^r \gamma) = \eta_{i \to m}
\]
hence $ad(\eta^r \gamma)(\eta_{l \to m} \partial_{k \to l} \otimes \eta_{j \to k} \partial_{i \to j}) = 3\cdot \eta_{i \to m}$. This is $0$ if and only if char $k \mid 3$, and when char $k \nmid 3$ similar calculations to example \ref{CPn} show that $[ad(\eta^r \gamma)] \neq [0]$ in ${\rm H}\,{\rm hoder}(A^!, A^!)$. Thus we have shown the following:
\begin{align*}
    Z_\infty{\rm Ext}_A(\Bbbk, \Bbbk) = \begin{cases}
                                                k[\eta] & \textup{if char $k \neq 3$}, \\
                                                \bigwedge(\gamma) \otimes_k k[\eta] & \textup{if char $k = 3$}. \\
                                                \end{cases}
\end{align*}

\end{exmp}

% Remarks about support theory.

% Introduce strict $A_\infty$-centre, write open question.

\section{The intersection morphism and free loop space fibrations}\label{stringtop}

There is a longstanding and well-understood relationship between Koszul duality and delooping. Its simplest incarnation lies in the use of the bar construction and twisting cochains as models for classifying spaces and principal fibrations. In this context, one has the classical theorem of Adams:
\begin{thm}[\cite{Ad}, \cite{FHTloop}] Let X be a finite simply-connected pointed simplicial complex. Then there is a quasi-isomorphism of dg algebras \footnote{Recall that for any space $X$, the diagonal map $\Delta: X \to X \times X$ induces a dg coalgebra structure on singular chains $C^{{\rm sing}}_*(X; k)$ via the Alexander-Whitney map; $C_*(X; k)$ stands for a modified dg coalgebra quasi-isomorphic to $C^{{\rm sing}}_*(X; k)$, see \cite{FHTloop} for definition. For a pointed space $X$, $\Omega X$ is the Moore loop space on $X$; it is a topological monoid and similarly induces on normalized singular chains $CN_*(\Omega X; k)$ a dg algebra structure via the Eilenberg-Zilber map.}
\[
\Omega C_*(X; k) \xrightarrow{\sim} CN_*(\Omega X; k).
\] 
\end{thm}
Equivalently, $CN_*(\Omega X; k)$ is a model for the Koszul dual to $A = C^*(X; k)$; in particular, ${\rm Ext}^{*}_{A}(k, k) \cong {\rm H}_*(\Omega X; k) = {\rm H}_*(\Omega X)$. % Should I prove this?

In characteristic zero, this relationship extends to Commutative and Lie Koszul duality: to any simply-connected space, Sullivan attaches a dg commutative algebra $A_{\rm PL}(X; k)$ of ``rational differential forms'' quasi-isomorphic to $C^*(X; k)$. Dually, one associates a dg Lie algebra $(\mathcal{L}_X, d)$, the Quillen model, whose cohomology $H^*(\mathcal{L}_X)$ recovers the homotopy Lie algebra $\pi_*(\Omega X)\otimes_{\mathbb{Z}} k$. Knowledge of the Sullivan or Quillen model of $X$ over $\mathbb{Q}$ determines its rational homotopy type $X_{\mathbb{Q}}$, and these are Koszul dual Commutative and Lie algebras. The relationship between the Associative and Lie Koszul duals is informally as follows: 
\[
\centerline{\xymatrixcolsep{5pc}\xymatrix@1{A_{PL}(X; k) \ar[d]^{{\rm Lie}} \ar[r]^{{\rm Asso}} & CN_*(\Omega X; k)\\
                        (\mathcal{L}_X, d) \ar[r]^{{\rm Envelope}} & \mathcal{U}(\mathcal{L}_X, d) \ar[u]_{\rotatebox[origin=c]{90}{$\sim$}}
}}
\]
where the quasi-isomorphism of dg algebras on the right is a result of Majewski (see \cite{Maj}, \cite[Ch. 26]{FHTbook}).

The goal of this section is to find a topological interpretation for Theorem \ref{maintheorem}. We first need to understand the Hochschild cohomology of $A = C^*(X; k)$. From now on, assume that $X$ is a simply-connected closed oriented smooth $d$-manifold, and let $LX$ be the free loop space on $X$. Much has been written about the relationship between free loop spaces and Hochschild (co)homology and we quote here only what is relevant for us, based on the results of Chas and Sullivan \cite{CS}.

Let ${\rm ev}: LX \to X$ be the evaluation map and $LX \times_{X} LX$ the space of ``composable'' pairs of loops. The Pontryagin product $LX \times_X LX \to LX$ extends that on the base loop space, yielding a map ${\rm H}_*(LX \times_X LX) \to {\rm H}_*(LX)$. Given a cycle $c \in C_n(LX \times LX)$ transverse to the subspace $LX \times_X LX$ in a suitable sense, Chas and Sullivan produce a cycle $c' \in C_{n-d}(LX \times_X LX)$ and from there a map ${\rm H}_n(LX \times LX) \to {\rm H}_{n-d}(LX \times_X LX) \to {\rm H}_{n-d}(LX)$. Next, the fibration $\Omega X \to LX \to X$ induces maps on homology ${\rm ev}_*: {\rm H}_*(LX) \to {\rm H}_*(X)$ as well as an Umkehr or ``wrong way'' map $I: {\rm H}_{*+d}(LX) \to {\rm H}_{*}(\Omega X)$, thought of as intersecting cycles down.
\begin{thm}[Chas, Sullivan \cite{CS}]

The map ${\rm H}_n(LX \times LX) \to {\rm H}_{n-d}(LX)$ gives rise to an associative product
\[
\mu: {\rm H}_{*+d}(LX)\otimes {\rm H}_{*+d}(LX) \to {\rm H}_{*+d}(LX).
\]
The loop product $\mu$ is graded-commutative and there are algebra morphisms 
\[
\xymatrix@1{ & {\rm H}_{*+d}(LX) \ar_{ev_*}[ld] \ar^{{\rm I}}[rd]  & \\
{\rm H}_{*+d}(X) & & {\rm H}_*(\Omega X)}
\]
the bottom-left algebra structure coming through Poincar\'e duality.
\end{thm}

\begin{thm}[Cohen, Jones \cite{CJ}]\label{cohen-jones} Let $A = C^*(X; k)$, then
there is an algebra isomorphism 
\[\Phi: {\rm H}_{*+d}(LX) \xrightarrow{\cong} {\rm HH}^*(A, A).
\]
\end{thm}

\begin{thm}[F\'elix, Thomas, and Vigu\'e-Poirrier \cite{FTVP}] There is an isomorphism\footnote{Possibly different from Adams' isomorphism.} $\Theta: {\rm H}_*(\Omega X) \cong {\rm Ext}^*_A(k, k)$ making the following diagram commute:
\[
\xymatrixcolsep{4pc}\xymatrix@1{{\rm HH}^*(A, A) \ar[r]^{\chi} & {\rm Ext}^*_A(k, k) \\
{\rm H}_{*+d}(LX) \ar[u]^{\Phi}_{\rotatebox[origin=c]{-90}{$\cong$}} \ar[r]^{I} & {\rm H}_*(\Omega X). \ar[u]^{\Theta}_{\rotatebox[origin=c]{-90}{$\cong$}}}
\]
\end{thm}
Theorem \ref{maintheorem} then implies:
\begin{cor}\label{intersectionimage} The image of
$I: {\rm H}_{*+d}(LX) \to {\rm H}_*(\Omega X)$ is the $A_{\infty}$-centre $Z_\infty{\rm H}_*(\Omega X)$.
\end{cor}
Note that $\Theta$ may not lift to an $A_\infty$-morphism and we instead impose on 
${\rm H}_*(\Omega X)$ the $A_\infty$-algebra structure of ${\rm Ext}^*_A(k, k)$ coming from this isomorphism. By uniqueness of minimal models this is (possibly non-strictly) isomorphic to the usual $A_\infty$-structure on ${\rm H}_*(\Omega X)$, according to Adams' Theorem.

The following calculations are not new but we include them as ways to exemplify Corollary \ref{intersectionimage}.
\begin{exmp}[fake commutative ring] Let $X = S^{n}$ be an even dimensional sphere. The dga $C^*(S^n; k)$ is formal, and so let $A = {\rm H}^*(S^n) = k[x]/(x^2)$ be the ``exterior algebra'' with $|x| = n$ even. This algebra is Koszul in the sense of Sect. \ref{Koszuldualitysection}, with Koszul dual $A^! = {\rm Ext}^*_A(k, k) = k[\partial_x]$ a polynomial algebra, where $\partial_x = s^{-1}x^*$ has odd cohomological degree $1-n$, corresponding to a class of homological degree $n-1$ in ${\rm H}_*(\Omega S^n)$. When char $k \neq 2$, the odd class $\partial_x$ is not in the graded centre of $k[\partial_x]$ since it would square to zero, and the image of $I$ is the even polynomial subalgebra $k[\eta] \subseteq k[\partial_x]$ generated by $\eta = \partial_x^2$. When char $k = 2$, this distinction goes away and $I$ is surjective.
\end{exmp}

\begin{exmp}\label{CP}
Let $X = \C\mathbb{P}^{n-1}$ for $n \geq 3$. One can again see that $C^*(X; k)$ is formal\footnote{Associative algebras of the form $k[z]/(z^n)$ are intrinsically formal by simple Hochschild cohomology calculations using results of Kadeishvili \cite{MR1029003}.}, so letting $A = {\rm H}^*(\C\mathbb{P}^{n-1}) = k[x]/(x^n)$ with $|x| = 2$ goes back to example \ref{CPn}. It follows that ${\rm im}\big(I) \subseteq {\rm H}_*(\Omega \C\mathbb{P}^{n-1})\big)$ is a polynomial subalgebra generated by the class in homological degree $2(n-1)$ when char $k \nmid n$, while $I$ is surjective when char $k \mid n$.
\end{exmp} % Does this betray a hidden H-space structure on p-localization?

\begin{exmp} Let $\Q \subseteq k$ and let $G$ be a compact simply-connected Lie group. Then $C^*(G; k)$ is formal and ${\rm H}^*(G) \cong {\rm Sym}(V^{odd})$ is an exterior algebra generated in odd degrees. This reduces to the BGG example \ref{BGG}, and so $I$ is onto. This example was treated geometrically over any coefficient ring $k$ by Chas and Sullivan, where surjectivity follows since the fibration $\Omega G \to LG \to G$ can be split using the group structure.
\end{exmp}
% Discussion of the Cohen-Jones-Yan spectral sequence.
Cohen, Jones and Yan discuss in \cite{CJY} the homology Serre spectral sequence coming from the fibration $\Omega X \to LX \to X$, and show that it gives rise, suitably regraded, to a multiplicative 2nd quadrant homology spectral sequence
\[
E_{p,q}^2 = {\rm H}^{-p}(X; H_q(\Omega X)) \implies {\rm H}_{p+q+d}(LX)
\]
Our calculations of the intersection morphism $I$ in the previous examples agree with the image of the edge morphism ${\rm H}_{*+d}(LX) \to {\rm H}_*(\Omega X)$. Shamir investigated in \cite{Sha} a similar spectral sequence for any `finite type' simply-connected dga $A$
\[
E_{p,q}^2 = {\rm H}^{-q}(A)\otimes {\rm Ext}_A^q(k, k) \implies {\rm HH}^{p+q}(A, A)
\]
whose edge morphism ${\rm HH}^*(A, A) \to {\rm Ext}^*_A(k, k)$ is the shearing morphism $\chi$. It is very plausible that this spectral sequence recovers the one studied by Cohen-Jones-Yan for $A = C^*(X; k)$. The results of the present paper suggests a link between the higher commutators against twisting cochains introduced in Sect. \ref{models-for-HH} and the differentials in these spectral sequences. This will be investigated in a later paper.

The image of the intersection morphism is typically quite small compared to the size of ${\rm H}_*(\Omega X)$. In characteristic zero, some of this can be ascribed to the expectation that the graded centre of ${\rm H}_*(\Omega X; k) \cong  \mathcal{U}(\pi_{*}(\Omega X)\otimes k)$ be small. However, this is no longer true if the Lie bracket on $\pi_{*}(\Omega X) \otimes_{\Z} k$ vanishes, since then ${\rm H}_*(\Omega X)$ is graded commutative, and yet $I$ is typically not surjective, e.g. ex. \ref{CP}. The discrepancy is precisely accounted for by cor. \ref{intersectionimage}: without vanishing of the higher Massey-Lie brackets on $\pi_{*}(\Omega X) \otimes_{\Z} k$, the algebra ${\rm H}_*(\Omega X)$ is far from $A_\infty$-commutative, and so the image of $I$ can still be quite restricted. Thinking along those lines, we recover algebraically the following proposition of F\'elix, Thomas and Vigu\'e-Poirrier on surjectivity of the intersection morphism:

\begin{prop}[F\'{e}lix, Thomas and Vigu\'{e}-Poirrier, \cite{FTVP}]\label{FTV-P} 
The intersection morphism ${\rm H}_{*+d}(LX; \Q) \xrightarrow{I} {\rm H}_*(\Omega X; \Q)$ is surjective if and only if $X$ is rationally homotopy equivalent to a product of odd dimensional spheres (equivalently, its cohomology ${\rm H}^*(X; \Q) = {\rm Sym}(V^{\rm odd})$ is an exterior algebra on odd generators).
\end{prop}

\begin{prop}\label{HHsymmetric}
Let $A$ be a strongly connected dg commutative algebra over a field $k$ of characteristic zero, with minimal associative Koszul dual $A^{!}$. Then the shearing map $\chi: {\rm HH}^*(A,A) \longrightarrow A^{!}$ is surjective if and only if $A$ is quasi-isomorphic to a symmetric algebra with trivial differential.
\end{prop}
\begin{proof}
Let $L$ be a minimal model for $\mathcal{L}^{co}\! A$, which exists since A is strongly connected\footnote{
It follows from \ref{Lie-com-homotopy} that any strongly connected $L_\infty$-coalgebra $L$ admits a minimal model. That is, there is a morphism $L'\to L$ from a minimal $L_\infty$-coalgebra such that $\mathcal{C}L'\to \mathcal{C}L$ is a quasi-isomorphism.
}. % i.e. A admits a minimal Tate construction $\mathcal{C}L\xrightarrow{\sim} A$.
According to \ref{envelope} we can use $A^{!}= \mathcal{U}(L^*)$ as a model for the minimal associative Koszul dual of $A$. By theorem \ref{maintheorem} if ${\rm HH}^*(A,A)\to A^{!}$ is onto then $A^{!}$ is $A_{\infty}$-commutative, so ${A^{!}}^{\rm Lie}$ is completely abelian by \ref{completelyabelian}. This implies that the sub-$L_{\infty}$-algebra $L^*\hookrightarrow  {A^{!}}^{\rm Lie}$ is also completely abelian, and hence so is the $L_{\infty}$-coalgebra $L$. The quasi-isomorphism $\mathcal{C} L \xrightarrow{\sim} A$ witnesses the main implication of the theorem. The other implication is well-known, see \ref{BGG}.
\end{proof}
This proposition is false in positive characteristic as the example $A = k[x]/(x^p)$ shows. It would be interesting to establish a positive characteristic analog.

Now prop. \ref{HHsymmetric} applied to $A = A_{\rm PL}(M; \Q)$, coupled with the result that symmetric algebras are intrinsically formal as graded commutative algebras implies that $I$ is surjective if and only if ${\rm H}^*(X; \Q) \cong {\rm Sym}(V)$, in which case ${\rm H}^*(X; \Q) = {\rm Sym}(V^{\rm odd})$ is an exterior algebra for dimension reasons.

% Questions, discussions, remarks, ...
%
\bibliographystyle{plain}
\bibliography{HKC}

\def\cprime{$'$}
\begin{thebibliography}{10}

\bibitem{Ad}
J.~F. Adams.
\newblock On the cobar construction.
\newblock {\em Proc. Nat. Acad. Sci. U.S.A.}, 42:409--412, 1956.

\bibitem{Baranovsky}
Vladimir Baranovsky.
\newblock A universal enveloping for {$L_\infty$}-algebras.
\newblock {\em Math. Res. Lett.}, 15(6):1073--1089, 2008.

\bibitem{MR3291613}
Tobias Barthel, J.~P. May, and Emily Riehl.
\newblock Six model structures for {DG}-modules over {DGA}s: model category
  theory in homological action.
\newblock {\em New York J. Math.}, 20:1077--1159, 2014.

\bibitem{BL}
H.~J. Baues and J.-M. Lemaire.
\newblock Minimal models in homotopy theory.
\newblock {\em Math. Ann.}, 225(3):219--242, 1977.

\bibitem{BGS}
Alexander Beilinson, Victor Ginzburg, and Wolfgang Soergel.
\newblock Koszul duality patterns in representation theory.
\newblock {\em J. Amer. Math. Soc.}, 9(2):473--527, 1996.

\bibitem{MR2016697}
Clemens Berger and Ieke Moerdijk.
\newblock Axiomatic homotopy theory for operads.
\newblock {\em Comment. Math. Helv.}, 78(4):805--831, 2003.

\bibitem{MR1832913}
Roland Berger.
\newblock Koszulity for nonquadratic algebras.
\newblock {\em J. Algebra}, 239(2):705--734, 2001.

\bibitem{BGG}
I.~N. Bern{\v{s}}te{\u\i}n, I.~M. Gel{\cprime}fand, and S.~I. Gel{\cprime}fand.
\newblock Algebraic vector bundles on {${\bf P}^{n}$} and problems of linear
  algebra.
\newblock {\em Funktsional. Anal. i Prilozhen.}, 12(3):66--67, 1978.

\bibitem{BuchweitzCanberra}
Ragnar-Olaf Buchweitz.
\newblock Hochschild {C}ohomology of {K}oszul algebras, talk at the
  {C}onference on {R}epresentation {T}heory, {C}anberra, July 2003.

\bibitem{MR2461267}
Ragnar-Olaf Buchweitz, Edward~L. Green, Nicole Snashall, and {\O}yvind Solberg.
\newblock Multiplicative structures for {K}oszul algebras.
\newblock {\em Q. J. Math.}, 59(4):441--454, 2008.

\bibitem{CS}
Moira Chas and Dennis Sullivan.
\newblock String topology.
\newblock {\em preprint}, 1999.

\bibitem{CJ}
Ralph~L. Cohen and John D.~S. Jones.
\newblock A homotopy theoretic realization of string topology.
\newblock {\em Math. Ann.}, 324(4):773--798, 2002.

\bibitem{CJY}
Ralph~L. Cohen, John D.~S. Jones, and Jun Yan.
\newblock The loop homology algebra of spheres and projective spaces.
\newblock In {\em Categorical decomposition techniques in algebraic topology
  ({I}sle of {S}kye, 2001)}, volume 215 of {\em Progr. Math.}, pages 77--92.
  Birkh\"auser, Basel, 2004.

\bibitem{MR3110804}
Vladimir Dotsenko and Bruno Vallette.
\newblock Higher {K}oszul duality for associative algebras.
\newblock {\em Glasg. Math. J.}, 55(A):55--74, 2013.

\bibitem{FHTloop}
Yves F{\'e}lix, Stephen Halperin, and Jean-Claude Thomas.
\newblock Adams' cobar equivalence.
\newblock {\em Trans. Amer. Math. Soc.}, 329(2):531--549, 1992.

\bibitem{FHTbook}
Yves F{\'e}lix, Stephen Halperin, and Jean-Claude Thomas.
\newblock {\em Rational homotopy theory}, volume 205 of {\em Graduate Texts in
  Mathematics}.
\newblock Springer-Verlag, New York, 2001.

\bibitem{FMT}
Yves F{\'e}lix, Luc Menichi, and Jean-Claude Thomas.
\newblock Gerstenhaber duality in {H}ochschild cohomology.
\newblock {\em J. Pure Appl. Algebra}, 199(1-3):43--59, 2005.

\bibitem{FTVP}
Yves F\'elix, Jean-Claude Thomas, and Micheline Vigu\'e-Poirrier.
\newblock Loop homology algebra of a closed manifold.
\newblock {\em preprint}, 2003.

\bibitem{MR2640648}
Benoit Fresse.
\newblock Props in model categories and homotopy invariance of structures.
\newblock {\em Georgian Math. J.}, 17(1):79--160, 2010.

\bibitem{MR1321701}
Murray Gerstenhaber and Alexander~A. Voronov.
\newblock Homotopy {$G$}-algebras and moduli space operad.
\newblock {\em Internat. Math. Res. Notices}, 3:141--153 (electronic), 1995.

\bibitem{getzler-jones}
E.~Getzler and J.~D.~S. Jones.
\newblock Operads, homotopy algebra, and iterated integrals for double loop
  spaces.
\newblock In {\em 15 T. Kashiwabara – On the Homotopy Type of Configuration
  Complexes, Contemp. Math. 146}, pages 159--170, 1995.

\bibitem{MR1261901}
Ezra Getzler.
\newblock Cartan homotopy formulas and the {G}auss-{M}anin connection in cyclic
  homology.
\newblock In {\em Quantum deformations of algebras and their representations
  ({R}amat-{G}an, 1991/1992; {R}ehovot, 1991/1992)}, volume~7 of {\em Israel
  Math. Conf. Proc.}, pages 65--78. Bar-Ilan Univ., Ramat Gan, 1993.

\bibitem{GMMVZ}
E.~L. Green, E.~N. Marcos, R.~Mart{\'{\i}}nez-Villa, and Pu~Zhang.
\newblock {$D$}-{K}oszul algebras.
\newblock {\em J. Pure Appl. Algebra}, 193(1-3):141--162, 2004.

\bibitem{MR3574211}
E.~L. Green, N.~Snashall, {\O}.~Solberg, and D.~Zacharia.
\newblock On the diagonal subalgebra of an {E}xt algebra.
\newblock {\em J. Pure Appl. Algebra}, 221(4):847--866, 2017.

\bibitem{MR0394720}
V.~K. A.~M. Gugenheim and J.~Peter May.
\newblock {\em On the theory and applications of differential torsion
  products}.
\newblock American Mathematical Society, Providence, R.I., 1974.
\newblock Memoirs of the American Mathematical Society, No. 142.

\bibitem{MR2172343}
Ji-Wei He and Di-Ming Lu.
\newblock Higher {K}oszul algebras and {$A$}-infinity algebras.
\newblock {\em J. Algebra}, 293(2):335--362, 2005.

\bibitem{Hinich}
Vladimir Hinich.
\newblock Homological algebra of homotopy algebras.
\newblock {\em Comm. Algebra}, 25(10):3291--3323, 1997.

\bibitem{Huebschmann}
Johannes Huebschmann.
\newblock On the construction of {$A_\infty$}-structures.
\newblock {\em Georgian Math. J.}, 17(1):161--202, 2010.

\bibitem{HMS}
Dale Husemoller, John~C. Moore, and James Stasheff.
\newblock Differential homological algebra and homogeneous spaces.
\newblock {\em J. Pure Appl. Algebra}, 5:113--185, 1974.

\bibitem{MR1029003}
T.~V. Kadeishvili.
\newblock The structure of the {$A(\infty)$}-algebra, and the {H}ochschild and
  {H}arrison cohomologies.
\newblock {\em Trudy Tbiliss. Mat. Inst. Razmadze Akad. Nauk Gruzin. SSR},
  91:19--27, 1988.

\bibitem{MR1258406}
Bernhard Keller.
\newblock Deriving {DG} categories.
\newblock {\em Ann. Sci. \'Ecole Norm. Sup. (4)}, 27(1):63--102, 1994.

\bibitem{MR2067371}
Bernhard Keller.
\newblock {$A$}-infinity algebras in representation theory.
\newblock In {\em Representations of algebra. {V}ol. {I}, {II}}, pages 74--86.
  Beijing Norm. Univ. Press, Beijing, 2002.

\bibitem{KellerDIH}
Bernhard Keller.
\newblock Derived invariance of higher structures on the {H}ochschild complex.
\newblock {\em preprint}, 2003.

\bibitem{Kellercocat}
Bernhard Keller.
\newblock {$A$}-infinity algebras, modules and functor categories.
\newblock In {\em Trends in representation theory of algebras and related
  topics}, volume 406 of {\em Contemp. Math.}, pages 67--93. Amer. Math. Soc.,
  Providence, RI, 2006.

\bibitem{KellerICM}
Bernhard Keller.
\newblock On differential graded categories.
\newblock In {\em International {C}ongress of {M}athematicians. {V}ol. {II}},
  pages 151--190. Eur. Math. Soc., Z\"urich, 2006.

\bibitem{MR2794666}
Henning Krause and Yu~Ye.
\newblock On the centre of a triangulated category.
\newblock {\em Proc. Edinb. Math. Soc. (2)}, 54(2):443--466, 2011.

\bibitem{MR1327129}
Tom Lada and Martin Markl.
\newblock Strongly homotopy {L}ie algebras.
\newblock {\em Comm. Algebra}, 23(6):2147--2161, 1995.

\bibitem{LH}
Kenji Lef{\`e}vre-Hasegawa.
\newblock {\em Sur les $A_{\infty}$-cat{\'e}gories}.
\newblock PhD thesis, Universit{\'e} Denis Diderot – Paris 7, November 2003.

\bibitem{LV}
Jean-Louis Loday and Bruno Vallette.
\newblock {\em Algebraic operads}, volume 346 of {\em Grundlehren der
  Mathematischen Wissenschaften [Fundamental Principles of Mathematical
  Sciences]}.
\newblock Springer, Heidelberg, 2012.

\bibitem{LPWZ}
Di~Ming Lu, John~H. Palmieri, Quan~Shui Wu, and James~J. Zhang.
\newblock Koszul equivalences in {$A_\infty$}-algebras.
\newblock {\em New York J. Math.}, 14:325--378, 2008.

\bibitem{Maj}
Martin Majewski.
\newblock A proof of the {B}aues-{L}emaire conjecture in rational homotopy
  theory.
\newblock In {\em Proceedings of the {W}inter {S}chool ``{G}eometry and
  {P}hysics'' ({S}rn\'\i, 1991)}, number~30 in II., pages 113--123, 1993.

\bibitem{MR1890736}
James~E. McClure and Jeffrey~H. Smith.
\newblock A solution of {D}eligne's {H}ochschild cohomology conjecture.
\newblock In {\em Recent progress in homotopy theory ({B}altimore, {MD},
  2000)}, volume 293 of {\em Contemp. Math.}, pages 153--193. Amer. Math. Soc.,
  Providence, RI, 2002.

\bibitem{MR1969208}
James~E. McClure and Jeffrey~H. Smith.
\newblock Multivariable cochain operations and little {$n$}-cubes.
\newblock {\em J. Amer. Math. Soc.}, 16(3):681--704 (electronic), 2003.

\bibitem{MR3438933}
Cris Negron.
\newblock {\em Alternate {A}pproaches to the {C}up {P}roduct and {G}erstenhaber
  {B}racket on {H}ochschild {C}ohomology}.
\newblock ProQuest LLC, Ann Arbor, MI, 2015.
\newblock Thesis (Ph.D.)--University of Washington.

\bibitem{MR2830562}
Leonid Positselski.
\newblock Two kinds of derived categories, {K}oszul duality, and
  comodule-contramodule correspondence.
\newblock {\em Mem. Amer. Math. Soc.}, 212(996):vi+133, 2011.

\bibitem{MR0265437}
Stewart~B. Priddy.
\newblock Koszul resolutions.
\newblock {\em Trans. Amer. Math. Soc.}, 152:39--60, 1970.

\bibitem{Quillen}
Daniel Quillen.
\newblock Cyclic cohomology and algebra extensions.
\newblock {\em $K$-Theory}, 3(3):205--246, 1989.

\bibitem{MR814187}
Michael Schlessinger and James Stasheff.
\newblock The {L}ie algebra structure of tangent cohomology and deformation
  theory.
\newblock {\em J. Pure Appl. Algebra}, 38(2-3):313--322, 1985.

\bibitem{Sha}
Shoham Shamir.
\newblock A spectral sequence for the {H}ochschild cohomology of a coconnective
  {DGA}.
\newblock {\em Math. Scand.}, 112(2):182--215, 2013.

\bibitem{MR2044054}
Nicole Snashall and {\O}yvind Solberg.
\newblock Support varieties and {H}ochschild cohomology rings.
\newblock {\em Proc. London Math. Soc. (3)}, 88(3):705--732, 2004.

\bibitem{MR2276263}
Bertrand To{\"e}n.
\newblock The homotopy theory of {$dg$}-categories and derived {M}orita theory.
\newblock {\em Invent. Math.}, 167(3):615--667, 2007.

\bibitem{BBCD}
Justin Young.
\newblock {\em {B}race {B}ar-{C}obar {D}uality}.
\newblock PhD thesis, \'{E}cole polytechnique f\'ed\'erale de {L}ausanne, 2013.

\end{thebibliography}

\end{document}